\documentclass[a4paper,10pt]{amsart}
\usepackage{amsfonts, amsthm, amsmath, amssymb,array,float,tikz,tikz-qtree}
\usepackage[pagebackref=true]{hyperref}
\usepackage{latexsym,xspace,enumerate}
\usepackage{graphicx}
\usepackage{ulem}
\PassOptionsToPackage{unicode}{hyperref}
\PassOptionsToPackage{naturalnames}{hyperref}
\hypersetup{
  colorlinks=true,
  linkcolor=blue,
  citecolor=red!50!black,
  urlcolor=red!50!black
}

\usepackage{cleveref}
\usepackage{subcaption}
\def \cal {\mathcal}

\def\E {\mathbb E} 

\def\m  {{\boldsymbol m}}

\def \c {\underline c}
\def\m {\underline m}
\def\n {\underline n}
\def\tilde {\widetilde}

\newtheorem{theorem}{Theorem}%[section]
\newtheorem{proposition}[theorem]{Proposition}

\newtheorem{definition}[theorem]{Definition}

\newtheorem{lemma}[theorem]{Lemma}

\title[Shannon weights for dynamical sources]{Shannon weights for binary dynamical  recurrent sources of zero entropy}
\author[A. Akhavi, E. Cesaratto, F. Paccaut, P. Rotondo, B. Vall\'ee]{Ali Akhavi, Eda Cesaratto, Fr\'ed\'edric Paccaut, Pablo Rotondo and Brigitte Vall\'ee}
\begin{document}
\maketitle
\begin{center} 
\today

\medskip

 \end{center}
 
    \begin{abstract}
      A  probabilistic source is defined as  the set of  infinite words  (over a given denumerable  alphabet)  endowed  with a probability $\mu$. The paper deals  with  general binary sources  where   the distribution of any symbol (0 or 1)  may depend on an unbounded  part of  the previous  history.   The paper  studies  Shannon weights:   whereas  the classical  Shannon entropy $\cal E_\mu$ is the average amount of information  brought by one symbol  of the  emitted word,  the  Shannon weight  sequence deals   with the average amount of information $\underline m_\mu (n)$  that is  brought  by the emitted  prefix of length $n$.  For a source with a non zero entropy,  the estimate   $\underline m_\mu (n) \sim  \cal E_\mu \cdot n$  thus holds. 
   
   \noindent The paper   considers  the model of dynamical sources,  where a source word is emitted as an encoded  trajectory of a dynamical system of the unit interval, when endowed with probability $\mu$.  It focus on sources  with zero entropy and      gives explicit constructions for sources whose Shannon weight sequence    satisfies  $\underline m_\mu (n)= o(n)$, with a prescribed  behaviour.  
   In this case, sources with zero entropy lead to  dynamical systems  built on maps with  an indifferent fixed point. 
       This class  notably contains the celebrated  Farey  source,  which presents  well-known intermittency phenomena.  
    
       \noindent 
  Methods are based on  analytic combinatorics and generating functions, and    they are enlarged, in the  present dynamical  case, 
 with  dynamical  systems tools (mainly transfer operators). 
\end{abstract}

 \tableofcontents

%Pour savoir qui a rajout\'e quoi, voici les couleurs : \brigitte{Brigitte}, \eda{Eda}, \pablo{Pablo}, \ali{Ali}, \fred{Fr\'ed\'eric}.

%Pour ajouter un commentaire ou autre chose, j'ai d\'efini les commandes \texttt{\textbackslash brigitte, \textbackslash eda, \textbackslash pablo, \textbackslash ali, \textbackslash fred}.

%Par exemple, \texttt{\textbackslash eda$\{$bla bla$\}$} donne \eda{bla bla}, un commentaire en bleu, la couleur d'Eda !

\section{Introduction} 

 {\sl Sources and weights.} 
   The paper studies probabilistic models on   infinite words. In information theory contexts,   a source  is a  process defined   on  the set of  infinite words over a given finite alphabet $\Sigma$,  endowed with a probability $\mu$.   It is a sequence of random variables $X_1 \ldots X_n \ldots$ where $X_n$ is the $n$-th  emitted symbol.  
   
   \smallskip
     The paper focuses on binary sources ($\Sigma = \{0, 1\}$) and   studies their  Shannon weight 
    defined  in Section \ref{weight1}. The classical  Shannon entropy  $\cal E_\mu$ --when it exists--
    is the average amount of information   enclosed in one  symbol  of the  emitted word, whereas the  Shannon weight  sequence  deals with the average  amount  $ \underline m_\mu (n)$ of information enclosed in the  emitted prefix of length $n$.  We are interested in the asymptotic behaviour of the sequence $n \mapsto \underline m_\mu(n)$ (as $n \to \infty$). 
    
    \smallskip
  When the sequence $n \mapsto  (1/n)\, \underline m_\mu (n) $ has a limit,  it is equal to the entropy   $\cal E_\mu$;  when    $\cal E_\mu$ is positive,  Shannon weights  satisfy   the estimate   $\underline m_\mu (n)  \sim \cal E_\mu  \cdot n$,    
  and they are  thus 
  of linear order. 
  But, for a source of zero entropy,  the Shannon weights will be 
   a main object of interest:  their behaviour  provide an extra knowledge on the probabilistic properties of the source,  
        and  the   paper  
   answers the general question: 
   
  \centerline  {\sl   Describe explicit sources with a prescribed Shannon weight $\underline m_\mu (n) = o(n)$. }

 \medskip {\sl Our models of sources.}
 The paper  studies  a  general   binary  source  (with  alphabet $\Sigma = \{0, 1\}$) which is furthermore {\it recurrent}\footnote {Such a source produces  infinite words which contain  almost surely  an infinite number of ones. A precise definition is  given in  Definition \ref{defR}.}; it is first  based on the   fundamental decomposition of the set  $ \{0, 1\}^\star $ of finite binary words, 
\begin{equation} \label{decintro} 
\{0, 1\}^\star = \cal {U}^\star  \cdot  \{0\}^\star \,  , 
\end{equation}  that deals with the set $\cal U = \{0\}^\star 1 $ of blocks,  each block being  a  sequence of zeroes  ended by   a symbol one. We consider mainly recurrent  binary sources with zero entropy, described  along the  decomposition \eqref{decintro},  and aim at estimating their Shannon weights.

  \smallskip
The two simplest model of sources --   the memoryless sources which  emit   iid symbols, or the  fixed length  (aperiodic) Markov chains,  
 when  the distribution of any next symbol   only depends on a bounded number of previous symbols -- have a  strictly positive entropy. 
 We  then focus  here  on   another general  model of sources,  which  may present   more complex correlation  phenomena,  
 the model of dynamical sources.  
 Such  a  dynamical source  is associated with a  measured dynamical system $({\cal I}, T, \mu)$ of the  unit interval $\cal I$, 
and a word emitted by the source is here the encoded trajectory of an input $x \in {\cal I}$ under  the probability $\mu$ defined on ${\cal I}$.
 As we are mainly interested in sources of zero entropy, we are led  to dynamical systems with indifferent fixed points, whose  encoded trajectories 
   present intermittency phenomena, with long runs of zeroes. A typical instance is the Farey source.

   \medskip 
{\sl Block source. }  Starting with a binary  recurrent  source $\cal S$,  the decom\-position \eqref{decintro} na\-tu\-rally  introduces 
another source --no longer  a binary source-- that   indeed  generates blocks and  
is called the block source, denoted as $\cal B$.   The block source $\cal B$   will be  a central object in our analysis.  We deal  within a general context  where the  source $\cal S$ is a dynamical  source whose block source $\cal B$  has nice properties : it is  associated with a dynamical system  of the {\sl Good Class}, and it has   notably  an invariant probability measure $\nu$  with a finite positive entropy ${\cal E}_\nu(\cal B)$.   In this context, the block source  $\cal B$ is often  more ``regular'' and easier to study than the initial source $\cal S$, and we often ``replace''\footnote{As a main instance of this fact, the Farey source is ``replaced'' by the source defined by the Gauss map and associated with continued fraction expansions} in the study  the source $\cal S$ by  the source $\cal B$.  

\medskip
  {\sl Generating functions.}   We  extensively use   methods of analytic combinatorics,  in the same vein as in the book of Flajolet and Sedgewick \cite{FS}.  We   first describe  a source ${\cal S}$ with probability $\mu$ via the probabilities $p(w)=p_\mu(w)$ that the emitted  (infinite) word begins with the prefix $w$. Due to \eqref{decintro},  the number  $n(w)$ of ones  in the prefix $w$ is also central, and     we  associate  with  the source  ${\cal S}$ a trivariate generating function
\begin{equation}  \label{Lambdaintro}\Lambda_\mu (v,t, s ) := \sum_{w \in  \Sigma^\star }  v^{|w|} \, t^{n(w)}\, p(w)^s \, .
\end{equation} 
The mean number  $\underline n_\mu(n)$ of ones and  the  mean average amount of information $\underline m_\mu(n)$ contained   in a prefix of length $n$  
are our final objects of interest, but 
 the  distribution $q_\mu(n) = \mu[W >n] $ of the waiting time $W$ of the symbol one  
 is  also an important (auxilliary) parameter. There are  finally three generating functions  (gf's in shorthand notation) of interest: 
the gf's $N_\mu(v)$ of the number of ones,    $M_\mu(v)$  of Shannon weights and $Q_\mu (v)$ of the waiting time, 
\begin{equation} \label{MNintro}
 Q_\mu(v) = \sum q_\mu(n)\,  v^n, \quad N_\mu(v) = \sum \underline n_\mu(n)\,  v^n, \quad  M_\mu(v) = \sum \underline m_\mu(n)\,  v^n \, .
\end{equation} 
Remark that the gf's $N_\mu(v)$  and $M_\mu(v)$ 
 are   (formal)  derivatives of  $\Lambda_\mu ( v, t, s)$ with respect to  $t$ or to $s$ at $(t, s) = (1, 1)$.  
	
\medskip
{\sl Dynamical analysis.}
All  our analyses,   performed in the dynamical context, are based on  the {\it dynamical analysis method} designed  by Vall\'ee and  then extended by  Cesaratto  and Vall\'ee in various papers (see for instance \cite{BaVa},   \cite{Vadyn}, \cite{CeVa}).   This method  mixes  analytic combinatorics tools  and dynamical systems tools,  in a general framework  where the (secant)  transfer operator of the  dynamical system  is used as a generating operator which itself generates   the fundamental probabilities $p(w)$ of the source. This  method  was  already widely used  in the analysis of  Euclidean algorithms  (\cite{BaVa}, \cite{VaEuc},\cite{Vabin}), in arithmetical studies (\cite{CeVa1})  or in contexts of information theory (see \cite{Vadyn}), but  it occurs for the first time in the study of Shannon weights.

\smallskip Using decomposition \eqref{decintro},  and directly dealing with generating functions associated with dynamical sources,     the paper  first obtains in Theorem~\ref{triA}   a precise (formal)  expression  of the trivariate  generating function  $\Lambda_\mu(v, t, s)$  in terms of 
transfer operators,  where 
the  weighted transfer operator ${\mathbb G}_{v, s}$  of the block system  $\cal B$  plays a central role.  We are  then led  to dynamical systems $\cal S$ 
for which the block system  ${\cal B}$  belongs to  an important subclass of dynamical systems, the {\sl Good Class}, already mentioned and  deeply studied in  \cite{CeVa}. There,  the operator  ${\mathbb G}_{v, s}$  fullfills (on a convenient functional space) nice dominant spectral properties, and it will be possible to take the derivatives of function $\Lambda_\mu$.

\medskip {\sl The four main results.}  The paper studies, as  its first main result  stated in Theorem \ref{inv}, invariant densities for the initial source $\cal S$ and for the block source $\cal B$.  Then, it obtains  in   Theorem \ref{proexpgenfinbis},   two  relations  (and one of renewal type) between  the three gf's $Q_\mu(v)$, $N_\mu(v)$ and $M_\mu(v)$.   
It remains to ``extract'' the coefficients $\underline n_\mu(n),\,  \underline m_\mu(n), q_\mu(n)$  of the gf's. As  these gf's have positive coefficients,  a good knowledge of their behaviour    on the real segment $[0, 1[$ (as $v \to 1^-$)  provides, with Abelian and Tauberian Theorems,  asymptotic estimates for  the   mean number of ones $\underline n_\mu(n)$ and Shannon weights $\underline m_\mu(n)$ in terms of $q_\mu(n)$, described in  our third main result, Theorem   \ref{final1}. It  derives in particular   an (indirect)  relation between the mean number $\underline n_\mu(n)$ of ones  and the Shannon weights, namely   $$ \underline m_\mu(n) \sim  {\cal E}_\nu(\cal B)\, \underline n_\mu(n) \,  \quad (n \to \infty)\, , $$ 
that involves the entropy ${\cal E}_\nu(\cal B)$ of the block system (wrt to its invariant measure $\nu$). 
 The previous relation
can be viewed as an extension of the  Abramov formula, 
 -- which relates the  the entropy $ \cal E_\pi (\cal S) $ of the initial source with respect to its invariant measure $\pi$ and the entropy $\cal E_\nu(\cal B)$ of the block source, when they are both positive and finite--  to the case  
when the initial dynamical system $\cal S$ has zero entropy and  possesses a block system in the  Good Class.

\smallskip The fourth main  result,  whose main steps are summarized   in Section \ref{firstexh}, is described in Theorem  \ref{exh1}: it     
explains how to find an explicit source  with given Shannon weights and answers the central question of the paper.

\medskip{\sl  Related  works. } 
In the context of dynamical systems, the number of ones   has been already  deeply studied,  notably by Thaler and  Aaronson,  with methods of Infinite Ergodic Theory  (\cite{Aa} \cite{Tha1}).   
 Aaronson provides a renewal equation of the same type  as ours, 
but with respect to another type of measure, and only for the number of ones. These authors do not study Shannon weights.  We   compare  our results  with  methods of Infinite Ergodic Theory in Section \ref{dynana2}.

\smallskip
  Our results are obtained via analytic combinatorics and dynamical analysis. Our  approach is furthermore  mainly applied to  sources with zero entropy; even though this  zero-entropy class  is natural and has been remarked a long time ago, it  is less  understood and  less investigated. There is  yet a recent renewal of  interest for these sources:   the paper \cite{les3} describes a general process   for  building in a natural way sources of zero entropy from sources of positive entropy; moreover,  the paper    \cite{les4} describes   and studies new extensions of entropies of various kinds to various weights. 

\medskip {\sl Plan of the paper. }  Section \ref{general}  introduces  the main objects (sources, Shannon  weights, waiting time, block source), and Section \ref{sec:DR} first 
 transfers the main concepts of Section \ref{general} to the dynamical context.  Section \ref{main} states the four main results. Then Section \ref{dynana1} describes the main tools --the trivariate gf $\Lambda$ and the three gf's $Q, N, M$-- and proves  the  first theorem on invariant densities (Theorem \ref{inv}). The next Section \ref{dynana2}   proves the second main result (Theorem  \ref{proexpgenfinbis}).  Then, Section \ref{3-4}  describes the precise version of  Abelian and  Tauberian Theorems of use (that involves slowly varying functions) and  explains how to deduce the third   result (Theorem \ref{final1}) from Theorem  \ref{proexpgenfinbis}.   The sequel of  the paper (with Sections \ref{sec:dri} and  \ref{sec:DRIL})  builds explicit instances, based on maps with indifferent fixed points,  and leads to 
 the answer of the main question of the paper.

\section{Sources, Shannon weights, block sources, waiting time} \label{general}

The present section  describes the main objects of interest: sources in Section \ref{gensource}, with various points of view, then  the main costs of interest: Shannon weights in Section \ref{weight1}, and number of ones  in Section \ref{weight2}. With  the waiting time $W$ of symbol one, the word decomposes into a sequence of blocks,  and Section  \ref{wait} introduces the recurrence property (Property ${\cal R}$), that allows to define the source which emits blocks, called the block source.  This will be a central object of the paper.

\subsection{General sources} \label{gensource}

A general source  is a probabilistic process, denoted by  ${\cal S}$,  that emits infinite words with symbols from a given alphabet $\Sigma$ (finite or denumerable), one symbol at each discrete time $ t= 1, \ldots, t= n, \ldots$. Here $\mathbb N$ will denote the set $\{1,2,\dots\}$ of positive integers.

\smallskip{\sl Sequence of random variables. } 
Consider some abstract probability space $(\Omega, {\cal O}, \Pr$).  When, for $n \ge 1$,   $X_n$ (defined on $\Omega$)  
denotes the random variable   whose value at $\omega \in \Omega$ is  the symbol  $\alpha_n$ emitted at time $t = n$, namely $ \alpha_n= X_n(\omega) $,  the source ${\cal S}$  will be  identified with the sequence $ X= (X_n)_{n \ge 1}$ of random variables. %and it emits on  the alea $\omega$ the infinite word $\alpha = (\alpha_1 \alpha_2 \ldots \alpha_n \ldots =  X(\omega) $.   

\smallskip
In the sequel, any sequence (finite or infinite) of symbols will be identified with the concatenation of its terms (written without commas).  Any word emitted by the source is thus a one-sided infinite word $\alpha = X(\omega) =\alpha_1\ldots \alpha_n\ldots \in \Sigma^{\mathbb N}$.

 \smallskip 
 \paragraph {\sl  Fundamental  domains, fundamental probabilities  and Kolmogorov Property.}
For each $k\ge 1$,  the projection    $\pi_k:\Sigma^{\mathbb N}\to\Sigma^k$    associates  with the infinite word  $\alpha\in\Sigma^{\mathbb N}$ its finite prefix $  w =(\alpha_1\alpha_2 \ldots \alpha_k )$, and the set  
\begin{equation} \label{Iw}  {\cal I}_w=\pi_{|w|}^{-1} (\{w\}) \hbox{ gathers the elements of $\Sigma^\mathbb N$ that begin with the prefix $w$}.
\end{equation}It is called the cylinder  (or the fundamental domain) associated with the prefix $w$.  The space $\Sigma^{\mathbb N}$  is endowed with the product $\sigma$-algebra, i.e., the $\sigma$-algebra $\cal F$ generated by the sets $\{\pi_k^{-1}(w),\ w\in\Sigma^k, \ k\ge 1\}$. Since these sets are a base for the product topology, $\cal F$ is also the Borel $\sigma$-algebra on $\Sigma^\mathbb N$. 
Then,  the  sequence $\cal P$ of fundamental partitions  ${\cal P}_k$ of order $k$, 
 \begin{equation} \label{defP}
 \cal P = ({\cal P}_k), \quad   {\cal P}_k= \{{\cal I}_w \mid w \in \Sigma^k\} , \qquad 
  p(w) = \mu( {\cal I}_w)\, , 
 \end{equation}
   generates the $\sigma$-algebra $\cal F$. 
The distribution $ \mu = \mu_{X}$ of the sequence of random variables $ X= (X_n)_{n \ge 1}$,  is  defined  on $\cal F$, and  is thus completely specified by the  fundamental probabilities $p(w) = \mu({\cal I}_w)$ that a word begins with the prefix $w$: this is  known as the Kolmogorov principle.  

  \medskip
\paragraph {\sl A recursive point of view on the source.}
We consider 
the following two {\sl mappings} $\sigma$ and $T$ --$T$ is also called the shift-- defined on $\Sigma^{\mathbb N} $: 
For  an infinite word $\alpha \in \Sigma^{\mathbb N}$,  with $\alpha = \alpha_1\alpha_2, \ldots$, 
$\sigma(\alpha)= \alpha_1$  is the  first symbol of  $\alpha$ and $T(\alpha)$ is the infinite suffix $(\alpha_2 \alpha_3 \ldots)$, so that   the symbol $\alpha_n$  emitted at time $n$ is $\alpha_n = \sigma(T^{n-1}(\alpha))$ and 
 the  (infinite) emitted word $\alpha$ is then written as 
\begin{equation} \label{sigmaT} 
\alpha = (\sigma(\alpha), \, T(\alpha)) = (\sigma(\alpha),\,  \sigma(T(\alpha)), \ldots , \sigma(T^n (\alpha) \ldots)\, .
 \end{equation}

 \medskip 
\paragraph {\sl Costs.} 
The main objects of our study are   {\sl costs}. A cost $c$   is a map 
$c: \Sigma^\star\rightarrow {\mathbb R}$, 
 and, for any $k \ge 0$,     
we consider the mean value $\underline{c}_\mu(k)$ of $c$ on $\Sigma^k$, 
 \begin{equation} \label{costCk} 
\c_\mu(k)  := 
\sum_{w \in \Sigma^k} p_\mu(w) c(w)\, .  
\end{equation}
We wish to study the asymptotic behaviour of the sequence $(\c_\mu(k))_{k \ge 0}$.  As it is usual in analytic combinatorics, we  deal with the associated generating function 
\begin{equation} \label{gengenfun}
 C_\mu(v) = \sum_{k \ge 0}  \c_\mu(k)\,  v^k \, .
 \end{equation}
 
 The cost $c$ is increasing  if 

\centerline{$ c(w\cdot i) \ge c(w)$ for any $w \in \Sigma^\star$ and $i \in \Sigma$, }
where $w\cdot i$ denotes the concatenation of the word $w$ and the letter $i$.  
This means that the sequence $k \mapsto c \circ\pi_k$ is increasing and  
the cost $c$
can be  defined on the infinite words as a cost $c:  \Sigma^{\mathbb N} \rightarrow {\mathbb R} \cup \{+ \infty\}$  with 
\begin{equation} \label {cinfinite}
 c(\alpha) := \lim_{k \to \infty}  c\circ \pi_k(\alpha) = \sup\{ c(w)\mid w  \hbox{ finite prefix of } \alpha \}  \, . 
\end{equation} Then the following holds 
\vskip 0.1 cm 
\centerline { $\c_\mu(k) = \E_\mu[c \circ \pi_k],\, \quad  \c_\mu(k) \ge \c_\mu(k-1)$,  } 
 \vskip 0.1 cm %Then, 
 and  the generating function  $C_\mu(v)$  defined in \eqref{gengenfun}   
 satisfies the following:  

 \begin{lemma}  \label{diffe}For any increasing cost $c$, the generating function $(1-v) C_\mu(v)$ has positive coefficients   and the  partial sums of these coefficients  coincide  with $\c_\mu(n)$. 
 \end{lemma}
 
 \begin{proof}  
 The series $(1-v) C_\mu(v)$   satisfies 
  $$(1-v) C_\mu(v) =  \sum_{k \ge 0} \c_\mu(k)  \,  v^k - \sum_{k \ge 0} \c_\mu(k) \,  v^{k +1} = \c_\mu(0) + \sum_{k \ge 1}( \c_\mu(k)  -\c_\mu(k-1) )\, v^k\, \, .  $$
  \end{proof}

%%%%%%%%%%%%%%%%%%%%%%%%%%%%%%%%%%%%%%%%%%%%%%%%%%%%
  \subsection{Shannon weights}  \label{weight1}
 We consider the cost $m$ defined  on $\Sigma^\star$  that associates with  the prefix $w \in \Sigma^\star$  the amount of information it  brings.  Then the cost  $m$ satisfies\footnote{ 
  We   may also consider  that  any prefix  $w$ (even though it is not emitted)  always provides  the supplementary  information  relative to its length $|w|$ and, in this case, the amount of information brought by  the prefix $w$  is 
 $$ \tilde  m (w) = \log |w| +  m (w)\, .$$\label{footnote:information}}
 $$ m(w) = |\log p_\mu(w)|  \quad \hbox{if the prefix $w$ is emitted, } \qquad m(w) = 0 \quad\hbox{if not, }$$
  and  is increasing. As in \eqref{costCk} and \eqref{gengenfun},  we  associate,    
 \vskip 0.1 cm
\centerline{
$
\m_\mu(k) =\displaystyle \sum_{w \in \Sigma^k} p_\mu(w) m(w), \qquad M_\mu(v) = \sum_{k \ge 0}  \m_\mu(k)\,  v^k \, .$ } 

The mean value   $\m_\mu (k)$ is  written as\footnote{ For an event $A$, $[\![ A ]\!]$ is the Iverson's bracket  equal to 1 if $A$ arises and 0 if not} 
\begin{equation} \label{mmu} 
  \m_\mu (k) =  \displaystyle \sum_{w \in \Sigma^k} [\![ p_\mu(w) \not = 0]\!] \, p_\mu(w) |\log p_\mu(w)|=  \sum_{w \in \Sigma^k}  p_\mu(w) |\log p_\mu (w)| \, , 
 \end{equation}the second equality being  due to the convention $x \log x = 0$ when $x= 0$. Thus  $\m_\mu(k)$ coincides with  the joint entropy $H(X_1,\dots,X_k)$ of the random variables that define the source (see \cite{CoTho}).
 
 \begin{definition}\label{ps-ent} For the source $(\Sigma^\mathbb N, \mu)$, 
 \hfill
 \begin{itemize}
 \item[--]  the  sequence $k \mapsto \m_\mu (k)$  is called the  sequence of Shannon weights;   
 \smallskip
 
 \item[--] the  limit of  the sequence $ (1/k)\,  \m_\mu(k)$  --when it exists--   is called the entropy of the source,  denoted as ${\cal E}_\mu $.  \end{itemize}
 \end{definition} 

\smallskip 
In  information theory contexts (see for instance \cite{CoTho}), the entropy of the source is sometimes called \textit{entropy rate}. 
   A sufficient condition for the convergence of the sequence  $ (1/k)\,  \m_\mu(k)$  is the sub-additivity of the sequence $k \mapsto \m_\mu (k)$.  And a sufficient condition for this sub-additivity is the stationarity of the sequence $(X_n)_{n\in{\mathbb N}}$ or equivalently (See \cite{PoYu}) the fact that $\mu$ is invariant by  the shift  $T$, i. e., 
   \vskip 0.1 cm
   \centerline{ 
     $\mu(T^{-1}  G) = \mu(G)$ for any $G \in {\cal F}$.}
   
   \smallskip
     When the  entropy $\cal E_\mu$ is positive, the Shannon weights  $\m_\mu(k)$ are of linear order. 
   But, in the  case of zero  entropy,     the Shannon weights  are $o(n)$ and their precise asymptotic study provides an  extra interesting  knowledge on  the source. 
 
 \medskip
 The idea of extending entropy into a sequence of weights is not new and has already  been proposed  (for instance) in the paper \cite{les4},  notably  in the context of dynamical sources (see Section \ref{SW}).  However,  the   weights   introduced there extend  other types of  entropy,  defined   in Dajani and Fieldsteel \cite {DaFi},  (not the Shannon entropy),  and are  described in Definitions 2.1 and 2.2 of  \cite{les4}.  To the best of our knowledge, the  interest in the  precise  asymptotic study of Shannon weights, that   describes the evolution  of $|\log p(w)|$ when $ |w|$ tends to $\infty$, appears to be new. 

\smallskip 
With the point of view of Footnote \ref{footnote:information}, the average amount of information on $\Sigma^k$ is \vskip 0.1 cm \centerline{$ \tilde \m_\mu (k) = \log k + \m_\mu(k)$,}\vskip 0.1 cm    and this version of Shannon weights satisfies $\tilde \m_\mu (k) \ge \log k .$

%%%%%%%%%%%%%%%%%%%%%%%%%%%%%%%%%%%%%%%%%%%%%%%%%%%%%%%%%%%%%%%%%%%%%
 \subsection{Binary sources and number of ones. } \label{weight2}

The paper is devoted to the case of  binary sources ${\cal S}$, i.e.,  sources  with a binary alphabet $\Sigma= \{0, 1\}$. There
is a  fundamental decomposition
\begin{equation} \label{decfond} 
\{0, 1\}^\star =  {\cal U}^\star \cdot {\cal Z} \, , 
\end{equation}  
\begin{equation}  \label{dec-blocks} 
 \hbox{with} \ \    {\cal U}  :=\{0\}^\star 1 =  \{ 0^{u-1}1 \mid u \ge 1\}, \qquad {\cal Z} :=\{0\}^\star = \{0^u \mid u\ge 0\} \, .
 \end{equation} 
 An element $U$ of ${\cal U}$ is called a block, and an element of ${\cal Z}$ is called a zero-tail.  
 The  set ${\cal U}$ is a set of words and may thus be called a code (in coding theory). This set ${\cal U}$  has an important property:   a proper prefix of a  codeword $U$ of ${\cal U}$ does not belong to ${\cal U}$, and ${\cal U}$  is  called a prefix code (in coding contexts). 
 
 \medskip
 The fundamental decomposition \eqref{dec-blocks}  provides,  for any  finite word $w$,  the writing
 \begin{equation} \label{decbin}
 \left \{
 \begin{array}  {llll}
w &= (0^{u_1-1} 1) \, \ldots 
 \, (0^{u_{\ell-1}-1} 1) \ (0^{u_\ell-1} 1) \,&0^{u_{0}}   &( \ell \ge 1)\cr
w &=   &0^{u_{0}},  & ( \ell = 0 \hbox {\ \ or  \ \ } w= \epsilon) \cr
\end{array}
\right\} \, , 
\end{equation}
where $\ell \ge 0$ is the number of ones contained in $w$,  the integers $u_i$ satisfy $u_i \ge 1$  (for  $1 \le i \le \ell$) and $u_0 \ge 0$.

\medskip
{\sl Number of ones.} 
We are thus led to another important cost $n: \Sigma^\star  \rightarrow {\mathbb N}$,  where $n(w)$  equals 
the number of ones in $w$ (it coincides with the number of blocks).  The cost $n$ is increasing, and, as  in \eqref{costCk} and \eqref{gengenfun},  we associate    
 \vskip 0.1 cm
\centerline{
$
\n_\mu(k) =\displaystyle \sum_{w \in \Sigma^k} p_\mu(w) n(w), \qquad N_\mu(v) = \sum_{k \ge 0}  \n_\mu(k)\,  v^k \, . $ }
\vskip 0.1 cm
The present paper studies  the asymptotic behaviour of  the sequence $k \mapsto \n_\mu(k)$.

As the cost $n$ is increasing, it may be defined on infinite words $\alpha$, with values in $\underline {\mathbb N}= {\mathbb N} \cup \{ + \infty\}$, as in \eqref{cinfinite}
and the equality holds 

\centerline{ $\n_\mu(k) = \E_\mu [n\circ \pi_k]$\, .}
  The  random variable $ n \circ \pi_k$  is expressed via  an ergodic sum in terms of   the set   ${\cal  J}$   of words which begin with the symbol one, 
\begin{equation} \label{ergo1}
n\circ \pi_k(\alpha) = \sum_{\ell = 0} ^{k-1} {\bf 1}_{\cal J} \circ T^\ell (\alpha)\, , \quad \cal J := \{ \alpha \in \Sigma^{\mathbb N} \mid \sigma( \alpha) = 1 \}\, , 
\end{equation}
and this is why (Infinite)  Ergodic Theory arises in a natural way in the study of $\n_\mu(k)$.    The present paper uses  another direction, and presents  alternative proofs   for the asymptotic estimates of  the mean number  $\n_\mu (k) $ of ones,    based on analytic combinatorics,  dynamical methods,  and functional analysis.   We think that this method  of proof may be of independent interest.

%%%%%%%%%%%%%%%%%%%%%%%%%%%%%%%%%%%%%%%%%%%%%%%%%%%%%%%%%%%%%%%%%%%%%%%%%%%%%%%%%%%%%%%%%
\subsection {Recurrence, waiting time and block sources. }
\label{wait}
We   consider the decreasing  sequence ${\cal N}_\ell$ of subsets of  $\{0, 1\}^{\mathbb N}$, and its intersection ${\cal N}$,   
\begin{equation}  \label{N}
{\cal N}_\ell := \{  \alpha \in  \{0, 1\}^{\mathbb N} \mid n(\alpha)    \ge  \ell \}\, \ \  (\ell \ge 1),    \qquad  {\cal N}:= \bigcap_{\ell \ge 1} {\cal N}_\ell \, ,
\end{equation} 
so that ${\cal N}$ is the set of  (infinite) words having an infinite number of ones. Then, for any $\alpha \in {\cal N}$, and for  any $k_0$,  there exists $k \ge k_0$ for which  
$ T^k(\alpha)  \in {\cal J} $.  
\medskip
\paragraph {\sl  Recurrence.} \label{recurrence}
 We consider sources that   satisfy,  when restricted to  the set $\cal N$,  a recurrence property with respect to the occurrence of symbol one, that is expressed in terms of the fundamental domains ${\cal I}_w$ introduced in \eqref{Iw}: 
 
 \begin{definition}  \label{defR}The source $(\Sigma^{\mathbb N}, \mu)$  is  recurrent if  \vskip 0.1 cm
 \centerline{for any $w \in \Sigma^\star$, the sequence  $\mu({\cal I}_{w \cdot 0^n})$ tends to 0 as $n \to \infty$.}
 \end{definition}
 
  \smallskip
Remark that, if the source $(\Sigma^{\mathbb N}, \mu)$ is recurrent, then  any source  $(\Sigma^{\mathbb N}, \mu')$ with $\mu \equiv\mu'$\footnote{Two measures are equivalent if they admit the same subsets of zero measure} is   recurrent too.

 \begin{proposition} {\rm [Recurrence]}

  \begin{itemize} \label{lemmaR}
  \item[$(a)$]  \textit{(i)}  When the source $( \{0, 1\}^\mathbb N, \mu) $ is  recurrent,  
   one has, with $p=p_\mu$  
   
 \vskip 0.1 cm 
 \centerline{$p(w) 
=\sum_{U \in {\cal U}} p (w \cdot U)$\, \quad for any $w \in \Sigma^\star .$}
\vskip 0.1 cm   
\noindent 
\textit{(ii)} For any $\ell\ge 1$,  for any sequence $(U_1,  U_2,  \ldots , U_\ell)$ of $\ell$ blocks, 
\vskip 0.1 cm 
\centerline{$p(U_1\cdot U_2\cdot \ldots \cdot U_\ell) = \sum_{U \in {\cal U} } p (U_1\cdot U_2\cdot \ldots  \cdot U_\ell \cdot U)\, $; }
\vskip 0.2 cm
\centerline{  $\sum_{(U_1, U_2, \ldots U_\ell)\in {\cal U}^\ell}  p (U_1\cdot U_2 \cdot \ldots \cdot U_\ell) = 1\, $; \hskip 2cm }

\smallskip 
\item[$(b)$] The equivalence holds, and involves the set $\cal N$ defined in \eqref{N}, 
\vskip 0.1 cm 
\centerline{ The source $(\{0, 1\}^\mathbb N, \mu)$ is recurrent  $ \Longleftrightarrow
\mu (\cal N) = 1$.}\vskip 0.1 cm

\end{itemize} 
 \end{proposition}

\begin{proof}

\smallskip 
Item $(a)(i)$.   For any $w \in \Sigma^\star$, one has

\centerline{$  p(w) = p(w\cdot 0) + p(w \cdot 1)= p(w\cdot 00) + p(w\cdot 01) + p(w\cdot 1) \qquad $}

so that, for any $i$,

\centerline{$ p(w) = p(w\cdot 0^i) + p(w\cdot 0^{i-1} 1) +  p(w\cdot 0^{i-2} 1)+ p(w \cdot 0^{i-3 } 1)+ \ldots  +p(w\cdot 1)\, .$}

Letting $i \to \infty$  proves  Item $(a)(i)$.  

\smallskip
Item $(a)(ii)$.  We apply  $(a)(i)$ to $w= U_1\cdot U_2 \cdot \ldots U_\ell$. \\
 Item $(a)(iii)$ is proven  by induction.  For $\ell = 1$, this is  $(a) (i)$ applied  to  $w  = \epsilon$.  Now
$$ \sum_{(U_1, U_2, \ldots U_\ell, U_{\ell +1})\in {\cal U}^{\ell+1}} p (U_1\cdot U_2\cdot \ldots \cdot U_\ell\cdot U_{\ell +1}) 
= \sum_{U \in {\cal U}} \sum_{(U_1, U_2, \ldots U_\ell)\in {\cal U}^{\ell}}  p(U_1\cdot U_2 \cdot \ldots \cdot U_\ell\cdot U) $$ 
\vskip -0.3cm 
$$ =  \sum_{(U_1, U_2, \ldots U_\ell)\in {\cal U}^{\ell} }  \sum_{U \in {\cal U}} \ p (U_1\cdot U_2 \cdot \ldots \cdot U_\ell\cdot U) $$
and  we use  $(a)(ii)$ together the induction hypothesis. .
 
\smallskip
Item $(b)$. The set ${\cal N}_\ell$ coincides with the set of infinite words $X$ which begin with a sequence of $\ell$ blocks.
Then 
\vskip 0.1 cm 
\centerline{ $\mu[{\cal N}_\ell] = \sum_{(U_1, U_2, \ldots, U_\ell) \in {\cal U}^\ell}   p(U_1\cdot U_2 \cdot \ldots U_\ell)\,  $}
\vskip 0.1 cm
that equals 1 with Item $(a)(ii)$. 
Conversely,  if a source is not recurrent, then there exists a finite word $w$ such that $\lim_i p(w0^i)$ is strictly positive. It gives rise to an infinite word $\alpha = w\cdot 0^{\mathbb N}$  for which $ \mu(\alpha) $ is strictly positive and $n(\alpha)= n(w)$ is finite. Then 
$\alpha$ does not belong to ${\cal N}$ and $\mu [{\cal N}] $ is stricly less than 1.
\end{proof}

  \medskip
  \paragraph {\bf \sl Waiting time.} 
  With  the set ${\cal  J}$ of words which begin with the symbol one,  described  in \eqref{ergo1},      the  {\it waiting time}  $W$ to hit $\cal J$ 
  is  defined  on the set $\cal N $ as     \vskip 0.1cm 
 \centerline{ 
  $ W(\alpha)=  1+  \min \{k\mid \sigma(T^k(\alpha)) = 1\} = 1+   \min \{k \ge 0 \mid T^k(\alpha) \in {\cal J}\}. $}
  \vskip 0.1cm 
  There is  $(1+)$ in the formula because  
    the symbol produced at time $t$ is the symbol $\sigma (T^{t-1} (\alpha))$.   
 This waiting time $W$       is the first time  $t \ge 1$ when  there is a symbol `1` at position $t$.  This  gives rise to    three random variables, all defined  on $\cal N$,  the first block $U$,  its length $u$,   and the block shift $\widehat T$ 
   $$U(\alpha)   = 0^{W(\alpha)-1} 1,\qquad  u = W(\alpha) ,  \qquad \widehat T(\alpha) = T^{W(\alpha)} (\alpha) \, .$$
   The block shift $ \widehat T$ %(defined on $\cal N$)  
  is  also  called the block shift  because it ``jumps'' over the block.   The infinite word $\alpha \in {\cal N}$ is  thus written   as  
    \vskip 0.2 cm
   \centerline{$ \alpha = (W(\alpha), \ \widehat T (\alpha) )$, }
   \vskip 0.1 cm 
 and this  defines,  as in \eqref{sigmaT},   and on the set $\cal N$,  a source ${\cal B}$   with  alphabet $\mathbb N$,  
  with
 \begin{equation} 
 \label{enc} 
\sigma_{\cal B}(\alpha) = W(\alpha) , \quad T_{\cal B}(\alpha) = \widehat T(\alpha) \,  . \end{equation}

\smallskip
\paragraph {\sl  Block source.} \label{block} 

When the source $\cal S$ is recurrent, the map  $\widehat T: {\cal N} \rightarrow {\cal N}$ is defined $\mu$-almost everywhere, and  any  infinite word  of $\cal N$ emitted by ${\cal S} $ is  encoded   as  in \eqref{enc}.  
 This   defines  on $\cal N$  a source  ${\cal B}$  with alphabet ${\mathbb N}$, and, by definition of this  source ${\cal B}$,   the equality holds  for $U_i = 0^{u_i-1} 1$
 \vskip 0.1 cm
 \centerline{
 $ \mu [ {\cal S} \hbox{ emits the  prefixes }  U_1, U_2 \ldots, U_\ell]= \widehat\mu[{\cal B} \hbox { emits  the  prefix  }   u_1, u_2 \ldots u_\ell ]\, .$}
 \vskip 0.1 cm
 This leads to the next definition which introduces one of the  main tools in the present paper. 

\begin{definition} \label{blockdef} {\rm [Block source]}
When the source $\cal S = ( \{0, 1\}^\mathbb N, \mu) $ is  recurrent,  it 
gives rise  on the set $\cal N$  with full measure $\mu[\cal N]  = 1$ to a  block source ${\cal B} = ( {\mathbb N}^\mathbb N, \widehat \mu)$, where    $\widehat \mu$ is defined by the 
   equalities \vskip 0.1 cm 
 \centerline{$  p_{\widehat \mu} (u_1 u_2 \ldots u_\ell) = p_\mu  (0^{u_1-1} 1 \cdot  0^{u_2-1} 1 \ldots  0^{u_\ell-1} 1 ) \, ,  $} 
 \vskip 0.1 cm
 \noindent that   relate  the fundamental probabilities  $p_\mu$ of the source $ \cal S$  and the  fundamental probabilities $p_{\widehat\mu}$ of  the source $ \cal B$.   
 \end{definition}
 
\subsection{Distribution of the waiting time and generating function $Q$.} 
 For  a  recurrent source $\cal S =  ( \{0, 1\}^\mathbb N, \mu)$, and  with  Proposition \ref{lemmaR}, the  waiting time $W$ and the block source $\cal B$ are well-defined (almost everywhere), and one thus deals with three main objects: 

\smallskip 
{\sl  Level sequences related to the distribution of $W$ (wtd sequences).}  There are  two   level sets sequences attached with the waiting time $W$, namely,   
\vskip 0.1 cm
\centerline{$n \mapsto [W>n]$, \quad $n \mapsto [W= n]$.} The probabilities of the associated  events  define the waiting time distribution (in shorthand {\sl wtd}), 
\begin{equation}\label{qmu}
\renewcommand{\arraystretch}{1.2}
\hskip 1 cm 
\left\{ \begin{array} {llllll}
  q_{\mu}(k)& = \mu[W>k]= \mu({\cal I}_{0^k})& (k \ge 0) \, \cr 
    r_{\mu}(k) &= \mu[W=k] = \mu({\cal I}_{0^{k-1}1})& (k \ge 1)\,   \cr 
    \end{array} \right \}\, , 
 \end{equation}  and,  for a recurrent source,  the sequence  $q_{\mu}(k)$ tends to 0. 
 
 \smallskip {\sl Generating function $Q_\mu(v)$.} The   generating function $ Q_{\mu}(v)$  
 of  the sequence $k \mapsto q_{\mu}(k)$ is defined as
  \begin{equation}  \label{Q} Q_{\mu}(v) = \sum_{k \ge 0} q_{\mu}(k) \, v^k \, \quad \hbox{and satisfies}   \quad Q_{\mu} (v) \to  \E_{\mu}[W], \quad (v \to 1^-)\, .
 \end{equation}

%%%%%%%%%%%%%%%%%%%%%%%
%%%%%%%%%%%%%%%%%%%%%%

 \section{Dynamical  sources of  various types. }   \label{sec:DR}
 
 We now consider particular sources of  the previous Section, created by measured  dynamical systems;     we first define this concept in  Section \ref{dynsources}, based  on the notion  introduced in  \cite{Vadyn},  and   we  introduce in Section \ref{gendynsys} the notion of generating  dynamical sources, associated with  measured  dynamical systems.  The general framework of Section \ref{wait} is then transfered in Section \ref{waitbis} into  the framework of dynamical sources.   We   define in particular    the Class $\cal {DR}$ which gives rise to  recurrent sources,  then,  the Class 
   $\cal {DRBGC}$  which gathers  sources  of the Class $\cal{DR}$   for which the block source is of  Good Class.  This Class  will be central in dynamical analysis of Sections \ref{dynana1} and \ref{dynana2};  it   contains an emblematic source, the Farey source described in Section \ref{Fareycase}.
   
   \medskip
    We use the same notations as in Section \ref{general}. The alphabet (finite or denumerable) is denoted by $\Sigma$, an infinite word (i.e. an element of $\Sigma^\mathbb N$) is denoted  by $\alpha$ whereas a finite word  (i.e. an element of $\Sigma^\star$) is denoted  by $w$. Moreover,  the  Lebesgue measure on  the unit interval $\cal I = [0,1]$ is denoted by $\tau$.

%%%%%%%%%%%%%%%%%%%%%%%%%%%%%%%%%%%%%%%%%%%%%%%%%%%%%%%%%%%%%%%%%%%%%%%%%%%%%%%%%      
   \subsection {A class of dynamical systems. } \label{dynsources}

 \begin{definition}   \label{ds}{\rm [(Measured) dynamical system.]}  
  A dynamical system $\cal S = (\cal I, T)$  is defined 
     by four elements 
   \begin{itemize} 
  \item[$(a)$]  an  alphabet $\Sigma$  (at most)  denumerable
    \item[$(b)$]  a  {topological } partition $({\cal I}_w)_{w \in \Sigma}$   of the unit interval ${\cal I}= [0, 1]$, namely a family of  disjoint open intervals  $({\cal I}_w)_{w \in \Sigma}$ 
     whose closures $\overline {\cal I}_w$  cover the  interval $\cal I$. 
               \item[$(c)$]  a coding  map  $\sigma$ 
          which is constant and equal to $w$ on each ${\cal I}_w$,  
     \item[$(d)$] a  map $T$ defined on  the union  of open intervals  $({\cal I}_w)_{w \in \Sigma}$,   
    for which   each  restriction $T|_{\cal I_w}:{\cal I}_w\rightarrow ]0, 1[ $ is  strictly monotonic, {\rm surjective},  of class ${\cal C}^2$.   The inverse branch $h_w: ]0, 1[ \rightarrow {\cal I}_w$ of the branch $T|_{\cal I_w}$ is extended into a map  $ \cal I \rightarrow \overline{{\cal I}_w}$ also denoted by $h_w$, that is thus  of Class $\cal C^2$.
   \end{itemize}
\smallskip

 $(e)$  When $\cal I$ is endowed with  a  probability measure $\mu$, with a strictly positive  density $ \phi  \in {\cal C}^1([0, 1])$,  
 this gives rise to  a  measured dynamical system $(\cal I, T, \mu)$.  
 
 \smallskip
   The  dynamical system $(\cal I, T)$ is called the underlying dynamical system of the measured dynamical system $(\cal I, T, \mu)$
 
 \end{definition}
 
There are  thus three main hypotheses  in the definition: the  surjectivity hypothesis  asked for   the branch $T|_{\cal I_w}$, and two regularity hypotheses, namely the $\cal C^2$ regularity asked for the branches $T|_{\cal I_w}$ and the $\cal C^1$  regularity asked for  the density $\phi$.    Whereas  it is  often asked  in this   context that measure $\mu$ be invariant by $T$, i.e., 
 \vskip 0.1 cm 
 \centerline{ $\mu(T^{-1}  G) = \mu(G)$ for any measurable set $G$, }
 \vskip 0.1 cm
 this is not the case here, and,  the only  hypothesis asked for measure $\mu$ is the hypothesis \ref{ds}$(e)$ on its density $\phi$.

  \medskip 
   The coding map $\sigma$ is only defined on  the set ${\cal I} \setminus C$ where $C=  \{c_w \mid w \in \Sigma \}$
   gathers the subdivision points of the partition $({\cal I}_w)$.    
   Each subset $T^{-n } (C)$ is at most denumerable, and it is the same for their union 
 that is of measure $0$ for the probability  $\mu$   with a density $\phi \in {\cal C}^1$.  Then,  the subsets 
  \vskip 0.1 cm 
 \centerline{
 $ {\cal I }_k^\ast = {\cal I} \setminus \bigcup_{n =  0}^kT^{-n } (C), \quad {\cal I}^\ast = {\cal I} \setminus \bigcup_{n\ge 0}T^{-n } (C)$, }
 \vskip 0.1 cm  are measurable sets   with  a  $\mu$-measure equal to 1.
 Furthermore,   the  trajectory  ${\cal T}(x) =  (x, Tx, \ldots T^k(x) \ldots)$ of  $x  \in {\cal I}^\ast$  
    under the dynamical system $(\cal I, T)$,  gives rises to the encoded trajectories (truncated or  not) 
    $$  M_k (x) = (\sigma(x), \sigma(Tx), \ldots \sigma (T^k(x))\, ,\quad   M(x) = (\sigma(x), \sigma(Tx), \ldots \sigma (T^n(x)), \ldots)\, .  $$
     The map $M_k$ is defined on ${\cal I}_k^\ast$, and, due to the surjectivity hypothesis $(d)$, the image  $M_k ({\cal I}_k^\ast)$ is the whole  set $\Sigma^k$, and $M: {\cal I}^\ast \rightarrow \Sigma^\mathbb N$ is a surjection. 
Moreover, there  is  a  semi-conjugacy 
      \begin{equation} \label{conjugT} T \circ M (x) =    M \circ T(x)\, , 
      \quad (x \in {\cal I}^\ast)
     \end{equation}
between the shift $T$ on infinite words  defined in Section \ref{gensource}  and the map $T$ that defines the dynamical system : this is why they are denoted (with a slight abuse of notation)  by the same letter $T$.   
 
   \medskip
With a prefix $w= w_1w_2  \ldots w_k \in \Sigma^k$, we associate  the  map $h_{w} = h_{w_1} \circ h_{w_2} \circ  \ldots\circ h_{w_k} $;  with  Item$(d)$ of Definition  \ref{ds}, the map $h_w$ is well defined on $\cal I$ and  the (closed)  interval  
${\cal I}_w= h_w(\cal I)$  is called the fundamental interval  relative to prefix $w$. 

\medskip{}
 Denote by   ${\cal I}_w^\ast$  the set  which gathers the reals $x$ for which the encoded trajectory $M(x)$ begins with the word $w$. Then,  the set  ${\cal I}_w^\ast$  is a borelian set which differs from ${\cal I}_w$  by a  denumerable set of points, and,  as the interval $\cal I$ is endowed with a probability $\mu$ that possesses  a density $\phi>0$ of Class ${\cal C}^1$, then  
the probability  $p(w) = p_\mu(w)$ that the encoded  trajectory $M(x)$ begins with $w$ satisfies 
$p(w) =  
\mu ({\cal I}_w^\ast ) = \mu({\cal I}_w)$. 
\vskip 0.1cm 
Finally, we deal with the sequence $\cal P$ of fundamental partitions, 
\begin{equation} \label{defPbis}
\cal P = ({\cal P}_k), \qquad {\cal P}_k = \{{\cal I}_w \mid w \in \Sigma^k\},  \qquad  {\cal I}_w = h_w({\cal I}) \, .
\end{equation}

\medskip 
We   now  compare the two settings, the present setting and the setting of Section  \ref{gensource}. For a clearer comparison, we  compare the notions  of Section \ref{gensource} defined in \eqref{defP}, that are now underlined,  with their present  analogs in \eqref{defPbis}.

\begin {proposition} \label{2.1-4.1}
When a  measured dynamical system $(\cal I, T,  \mu)$ fulfills hypotheses of Definition \ref{ds},  
it  creates a source $ (\Sigma^\mathbb N, \underline \mu)$, 
 as in Section 2,   denoted as $M(\cal I, T, \mu)$.  The measure $\underline \mu$
 is  completely defined by  the   system $(\cal I, T,  \mu)$. 
\end{proposition} 

\begin{proof} The set  ${\cal I}^\ast$ is endowed with $\mu$, and we  use the  surjective map $M: {\cal I}^\ast \rightarrow \Sigma^{\mathbb N}$ and  the  semi-conjugacy  in \eqref{conjugT} to  define a measure $\underline \mu$ on $\Sigma^\mathbb N$ for which the map 
\vskip 0.1 cm 
\centerline{$ M: ({\cal I}^\star, \mu)  \rightarrow  (\Sigma^{\mathbb N}, \underline \mu) $}
\vskip 0.1 cm
relates the  measures of fundamental intervals  via the equalities, 
$\underline \mu (\underline {\cal I}_w) = \mu ({\cal I}_w )$.  
Then,   the measure $\underline \mu$ is  now defined on the fundamental domains $\underline{\cal I}_w$ and extended on $\Sigma^\mathbb N$ with Kolmogorov's Theorem. 
 \end{proof} 
 
 In the sequel,  due to the semi-conjugacy  \eqref{conjugT}, we   consider  and study  Shannon weights  and entropy rate directly  on  measured dynamical systems,  rather on   the dynamical sources they define.
  We  now give main  examples: 
 
\smallskip $(a)$   The Shannon weight  (of depth $n$)   of the  measured dynamical system $(\cal I, T,  \mu)$, denoted as $\m_\mu (n)$,   is  (by definition)   the Shannon weight (of depth $n$) of  the dynamical source $M(\cal I, T, \mu) = (\Sigma^\mathbb N , \underline \mu)$.  \\
 $(b)$ The entropy rate of  the  measured dynamical system $(\cal I, T,  \mu)$, denoted as $\cal E_\mu$  is  (by definition) 
-- and when it exists--  the entropy rate   $\cal E_{\underline \mu}$ of the  dynamical source $M(\cal I, T, \mu) = (\Sigma^\mathbb N, \underline \mu)$. \\
$(c)$ The Shannon weight and the entropy rate   of the  measured dynamical system $(\cal I, T,  \mu)$ write in terms of 
  of the measures $\mu({\cal I}_w) = \underline \mu(\underline{\cal I}_w)$, 
\begin{equation} \label{Pm} 
\m_\mu(n)=  \sum_{w \in \Sigma^n} \mu({\cal I}_w)\, | \log \mu({\cal I}_w) | , \quad \cal E_\mu = \lim_{n \to \infty} \frac 1 n  \m_\mu(n)\, .
\end{equation}

 \subsection{Generating  dynamical systems} \label{gendynsys}
 The dynamical  source  $ M(\cal I, T, \mu)$ 
  is  ``created''  from  the  measured dynamical system $(\cal I, T, \mu)$. With the mapping $M$, the properties
of the dynamical system  relative to the measure $\mu$ are given  to the dynamical source.   But, in this way,  it is not clear how  to compare  --in a topological way--  the  dynamical source  $M(\cal I, T, \mu)$ with  the  measured dynamical system  $(\cal I, T,  \mu)$ itself.  We indeed  need  a supplementary condition:

\begin{definition} \label{defgen} {\rm [Generating  dynamical systems]}  A dynamical system  $\cal S= ({\cal I}, T)$ is generating if the sequence $\cal P$ of its fundamental partitions defined in \eqref{defPbis}  generates the Borelian $\sigma$-algebra ${\cal F}_{\cal I}$.  
\end{definition} 
Note that the analog property  always holds in the context of Section \ref{general}: this is due   to the definition of the $\sigma$-algebra $\cal F$ and   the  definition of the sequence $\cal P$  in \eqref{defP}.  

\begin{lemma} \label{lemgen}
For any $x \in {\cal I}^\ast$,  and any $n \ge 1$,  consider the prefix   $x_{\langle n \rangle}$ of length $n$ of $M(x)$,  the fundamental interval  ${\cal I}_n(x)$ relative to the prefix $x_{\langle n \rangle}$, and denote by $\delta_n(x)$ the diameter of ${\cal I}_n(x) $. 

\smallskip
 The three assertions are equivalent:  
\begin{itemize} 
\item[$(i)$]   The  dynamical system  $(I, T, \mu)$ is generating  ; 
\item[$(ii)$] The sequence $\eta_n=   \sup \{\delta ({\cal I}_w) \mid w \in \Sigma^n \}$ tends   to 0 as $n \to \infty$ ; 
\item[$(iii)$] For any $x \in {\cal I}^\ast$, the sequence $\delta_n(x)$   tends to 0  as $n\to \infty$.
\end{itemize} 
\end{lemma} 

Remark that Assertion $(iii)$ seems weaker than Assertion $(ii)$. But the two assertions  are indeed equivalent due to Dini's Theorem.

\begin{proposition} \label{bijdyn}  If the dynamical system $\cal S= (\cal I, T)$  is generating,  then: 
\begin{itemize} 
\item[$(a)$] There exists  a continuous map   $K : \Sigma^{\mathbb N} \rightarrow  {\cal I}^\ast$ for which $K \circ M (x) = x$.  
\item[$(b)$]   There is an homeomorphism between the  two topological  spaces -- $\cal I^\ast$ endowed  with the usual Borel topology, --  $\Sigma^\mathbb N$ endowed with the product topology of Section 2. 
\end{itemize} \end{proposition}

\begin{proof} [Sketch of proof]  (see \cite{PoYu} page 42). 
 With   $\alpha \in \Sigma^\mathbb N$,  written as  $\alpha = M(x)$,  associate  the subset  $ K(\alpha ) \subset {\cal I}$
 \vskip 0.1 cm 
 \centerline{
$ K(\alpha) =  \bigcap_{n \ge 1}   {\cal I}_n(x) $.} 
\vskip 0.1 cm
Now, the following equivalence holds for any  pair $(\alpha, x)$ with $\alpha= M(x)$, 
\vskip 0.1 cm
\centerline{ 
 $ K(\alpha) = \{x\} $   $\Longleftrightarrow $      the sequence $\delta_n (x) \to 0$   as $n \to \infty$ \, .}
 When the system is generating, this entails that $K$ is well defined on $\Sigma^\mathbb N$, and satisfies $K\circ M (x) = x$.  Then $K$ is the inverse mapping of $M$. 
 \end{proof}

In the sequel,  we will deal a priori with any dynamical source associated (as in Proposition   \ref{2.1-4.1}) with a dynamical system --generating or not--. However,  generating dynamical systems $\cal S = (\cal I, T)$   will be highly  interesting   here in the context of  entropy. Indeed,   the  homeomorphism   of Proposition  \ref{bijdyn}  allows us to use  the dynamical  point of view   in the computation of  entropy. This is described  in the following proposition that plays an important role   in the sequel:

\begin{proposition} \label{ent-mu}
 When a generating dynamical system $\cal S = (\cal I, T)$ admits an invariant   probability measure $\mu$,   then the entropy  (rate)  of the  dynamical source $M(\cal I, T, \mu)$     coincides with  the metric entropy   
 of the dynamical system $ (\cal I, T)$, with respect to its invariant measure $\mu$,  denoted as $\cal E_\mu (\cal S)$.
 \end{proposition}
 
 \begin{proof} See \cite{PoYu}  pages 80-90. 
 \end{proof}

  \subsection {Dynamical systems in the Class $\cal{DR}$. } \label{waitbis} 
   Consider now the particular case  of Definition \ref{ds} when the alphabet  is $\Sigma = \{0, 1\}$.

  \begin{definition} \label{defDR}
  Consider a binary dynamical system  $(\cal I, T)$  fulfilling the hypotheses of Definition \ref{ds}, the   (topological) partition of ${\cal I}$ into two subintervals ${\cal I}_0= ]0, c[ $ and ${\cal I}_1= ]c, 1[$ (with $c \in ]0, 1[$),    and the inverse branches   $a, b$  (of Class $\cal C^2$) of the direct surjective branches  $A, B $ of $T$.  \\The system is tent shaped if the branch $a$ is increasing and the branch $b$ decreasing, with $a(0) = 0, \, a(1)= c, \, b(0)=1, \,  b(1) = c$.     \\  
  The system  $(\cal I, T)$   belongs to the class $\cal {DR}$ (for dynamical and recurrent)  if   the branches $a$ and $b$    moreover  satisfy 
        \vskip 0.1 cm
        \centerline{
   $ \forall x\in [0,1],  \ \  a'(x)  >0,   \ \  b'(x)<0;  \quad \forall x\in ]0,1],  \ \    a'(x)< 1, \  \  -1 <  b'(x) \, . $} 
   When the branch $a$ satisfies the condition $a'(0) = 1$,  the point $0$ is an indifferent point for the branch $a$ (i.e., $a(0)= 0, a'(0) = 1$) and this leads to  the class $\cal {DRI}$. 
 
   \end{definition} 

       Figure \ref{fig:ab} represents  an instance of a tent-shaped system. One finds many instances of  graphs of  tent-shaped systems  in Section \ref{exDRIL}. 
      
      \begin{figure}[H]
    \centering
    \includegraphics[scale=0.40]{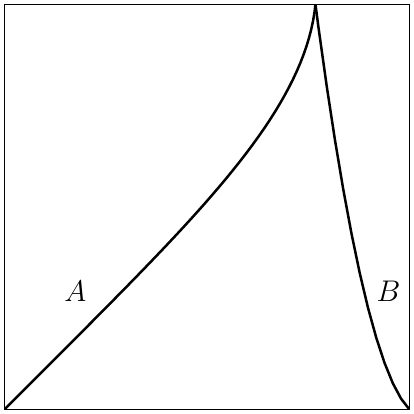} \includegraphics[scale=0.40]{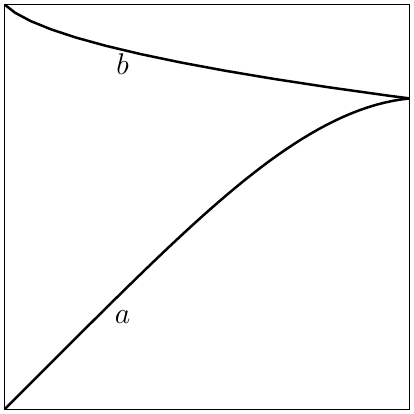}
    \caption{A tent-shaped dynamical system with its two direct branches (left) and its two inverse branches (right).}
    \label{fig:ab}
\end{figure}

  {\sl Even if the system $(I, T)$ is defined via its direct branches $(A, B)$,  the paper mainly deals  with inverse branches $(a, b)$ and  hypotheses  will be asked  directly  on  the inverse branches, as in the previous Definition \ref{defDR}. }

   \smallskip
The  Mean Value formulae, valid for any $(x, y)$ with  $0 \le x <y \le 1$, involve  $z$ and $t \in ]x, y[$, and 
   \vskip 0.1 cm 
   \centerline{  
   $|a(y) -a(x)|= a'(z)(y-x), \qquad |b(y) -b(x)|=|b'(t)|(y-x)$\, .} 
   \vskip 0.1 cm 
    \noindent Then, the  reals  $t, z$ are in $ ]0, 1[$,     and $0< a'(z) <1,  0<|b'(t)| <1$.          
 This entails that  the inverse branches $a$, $b$ of depth one    satisfy  {\sl strict}  Lipschitz conditions,  \begin{equation} \label{Lip} 
   \forall (x, y), \ \   0 \le  x  < y \le 1  \quad 0 <  |a(y) - a(x) | <  y -x, \quad   0<  |b(y) - b(x) | < y -x.   
   \end{equation}
With an easy recursion, any inverse branch $h_w$ (of any depth)  is weakly contracting and satisfies the inequality 
\vskip 0.1 cm 
\centerline {$ |h_w (x) -h_w(y)| \le |x-y|,    \quad 0 \le x \le y \le 1$.}  \vskip 0.1 cm  

\smallskip    
  For  a system $(\cal I, T)$  of the Class $\cal{DR}$, 
where the  absolute values  $|a'(0)|, |b'(0)|$   are strictly less than 1, 
the real 
\vskip 0.1 cm
\centerline{$\rho = \max\{\sup_{[0,1]}|a'|,\sup_{[0,1]}|b'|\} $ is strictly less than 1.} In this case, for any measure $\mu$ that fulfills Def \ref{ds} $(e)$,  the measure $\mu(\cal I_w)$  satisfies  $\mu(\cal I_w)\le C\rho^n$, for some constant $C>0$ and the  entropy rate  $\cal E_\mu$ of  the  measured dynamical system $(\cal I, T,  \mu)$ satisfies  $\cal E_{\mu}\geq |\log\rho|$.  Such a system is not really interesting in our context which focuses on systems of zero entropy.

\smallskip This is why we will focus in the sequel on systems of the class ${\cal {DRI}}$ for which the branch $a$  admits an indifferent point at $x = 0$, (i.e., satisfies  the conditions $a(0) = 0, a(0) = 1$). In this case,   the derivatives $(a^n)'(0)$ equal 1 for any $n$. There are   two  subcases, the case $\cal {DRI}_s$  (where $a'(0) =1, \ b'(0) >-1$) or the case 
$\cal {DRI}_w $ (where  $ a'(0) = 1, \ b'(0) = -1$). In particular, the case  $\cal {DRI}_w $  occurs for an emblematic instance, the Farey source (See Section \ref{Fareycase}).

    \subsection{Recurrence.}
   The Class $\cal{DR}$ plays a central role in the recurrence setting, as it has been  described in Section \ref{wait}.

   \begin{proposition} 
   \label{recDR0}
  Consider a  dynamical system  $(I, T)$  of the  Class $\cal{DR}$ and  the sequence $q(n) = a^n(1)$. Denote by $\tau$ the Lebesgue measure.      Then 
   
   \begin{itemize} 
   \item[$(i)$] The set 
 $[W>n]$ coincindes with the interval $[0, q(n)[$ and   the equality $q_\tau (n) = q(n) = a^n(1)$ holds between the wtd $q_\tau (n)$ with respect to the Lebesgue measure and $q(n)$.   The level set 
$[W= m]$ is an   interval  of $\cal I$ whose interior  is $ {\cal J}_m = ]q(m) ,  q(m-1)[$.

\item[$(ii)$] The sequence $n \mapsto q(n)$ tends to 0 (when $n \to \infty$)
   
 \item[$(iii)$]   
  Any    measured dynamical system $(\cal I, T, \mu)$ associated with   the dynamical system  $(I, T)$ defines a recurrent dynamical source $M(\cal I, T, \mu)$. 

 \end{itemize} 
 \end{proposition} 

 \begin{proof} $(i)$ The set $[W>n]= {\cal I}_{0^n}$ is  written in terms of  iterates of the inverse branch $a$ as $[W>n] = ]0, a^n(1)[$.  
   
 \smallskip  $(ii)$ For any $ x\in]0,1]$, the sequence $a^n(x)$ tends to 0:  by induction,  indeed,    the sequence $x \mapsto a^n(x)$ is  non zero and strictly decreasing, due to \eqref{Lip}.  It is thus 
	  convergent  and, as $a$ is continuous,  its limit is the only fixed point for $a$, namely $0$.
	  
\smallskip	 $(iii)$  Any interval ${\cal I}_{w\cdot 0^n}$ associated  with a word $w \in \{0, 1\}^\star$ is written as  
\vskip 0.1 cm 
\centerline{ ${\cal I}_{w\cdot 0^n} = h_w({\cal I}_{0^n}) =  h_w([W >n]) =[
h_w(0), h_w(q(n))]. $}
\vskip 0.1 cm
As $\phi \in {\cal C}^1$, the function $F$ given by $F(x) = \int_0^x \phi(t) dt$ belongs to ${\cal C}^2([0, 1])$, and
\begin{equation*}
\begin{split}
\mu ({\cal I}_{w 0^n} ) &= \mu (h_w({\cal I}_{0^n}))=\int_{h_w(0)}^{h_w(q(n))} \phi(u) du\\
&= F(h_w(q(n)))-F(h_w(0))= (h_w(q(n)) -h_w(0)) \, F'(x_w)
\end{split}
\end{equation*}
with $x_w \in ]h_w(0), h_w(q(n))[$.
\vskip 0.1 cm
As  $F$ is of Class ${\cal C}^2$,  one has $|F'(x)| \le A$ for any $x \in \cal I$ and   some $A$. Moreover,   
   the inverse branch $h_w$ is weakly contracting and $ |h_w(0) -h_w(q(n))| \le q(n) $.    Then   the inequality 
 \vskip 0.1 cm
 \centerline{  $\mu ({\cal I}_{w 0^n} )  \le  A  \, q(n)$}
 \vskip 0.1 cm  holds. Thanks to $(ii)$, as $n$ tends to $\infty$,the measure $\mu ({\cal I}_{w 0^n} )$ tends to 0.  
  With Definition \ref{defR},  the source $M(\cal I, T, \mu)$ is recurrent. 
  \end{proof}

\smallskip	  
	It is now possible to  ``transfer'' the definitions of Section \ref{wait} into  the context of  Class $\cal{DR}$.
	
	\smallskip
	--   The wtd $q_\mu(n) = \mu[W>n]$  of the  measured dynamical system $(\cal I, T, \mu)$  is (by definition) the wtd $\underline \mu [W>n]$ 
  of the  dynamical source  $M(\cal I, T, \mu)$, defined in \eqref{qmu}.  
  \\-- This defines the generating function  $Q_\mu(v)$  of the  measured dynamical system $(\cal I, T, \mu)$  as the generating function $Q_{\underline \mu}(v)$ of the dynamical source  $M(\cal I, T, \mu)$ defined in \eqref{Q}.

  \subsection{Block source}
  We  now associate with a dynamical system $(\cal I, T)$ of the Class $\cal{DR}$  another dynamical system,  denoted as  $(\cal I, \widehat T)$, called the block dynamical system of $(\cal I, T)$,   and defined in the context of Definition \ref{ds}  as follows. It uses the level sets  defined in Proposition \ref{recDR0} and the fact that the sequence $q(n)$ tends to 0 ($n \to \infty$)   (see Proposition \ref{recDR0}$(ii)$).

   \begin{definition} \label{blocksys}  {\rm [Block dynamical system]}  When the system $(\cal I, T)$ belongs to the Class $\cal{DR}$, the  block dynamical system  
 $ (\cal I,\widehat  T) $ is  described  in the context of Definition \ref{ds} as follows: 
  \begin{itemize}

      \item[$(a)$] the  alphabet  is $\mathbb N$

     \item[$(b)$]  the   {topological } partition   is defined by the intervals ${\cal J}_m = ]q(m) ,  q(m-1)[$  for $m \ge 1$, and  the  set of subdivision points  is $\widehat  C = \{ q(n) \mid n \ge 0 \} \cup \{0\}$.

    \item[$(c)$]  the coding map $\widehat \sigma$ is  constant on each ${\cal J}_m$ and equals $m$.

    \item[$(d)$] the  map $\widehat T$ is  defined on  the union  of  intervals  ${\cal  J}_m$ for $m \ge 1$.    
          The inverse of the   restriction $\widehat T_m : {\cal J}_m \rightarrow ]0, 1[ $ is  
      $g_m:  ]0, 1[ \rightarrow {\cal J}_m$, that   is extended into a map  $ \cal I \rightarrow [q(m), q(m+1)]$  equal to 
      $ g_m = a^{m-1} \circ b$. 
      \end{itemize}
      \end{definition} 
      
      The following proposition (whose proof is clear and thus omitted) now relates the two notions: block dynamical systems, just defined here  and  block sources of Section 2, defined in Definition \ref{blockdef}.
      
      \begin{proposition}  {\rm [Block source in the context of Class $\cal{DR}$]} \label{blockDR0} Consider a  measured dynamical system $(\cal I, T, \mu)$ whose underlying system $(\cal I, T)$   belongs to the Class $\cal{DR}$.   Then,  the dynamical source $M(\cal I, \widehat T, \mu)$ is the block  source of the source $M(\cal I, T, \mu)$. 
      \end{proposition} 
  
\medskip  
 Just like in Section \ref{dynsources}, the block dynamical system  admits a sequence $\cal Q= ({\cal Q}_n)$ of fundamental partitions,  with
   \begin{equation} \label{Qn}
  {\cal Q}_n = \{ 
  g_{m_1} \circ g_{m_2} \circ \ldots \circ g_{m_n} (\cal I) \mid  m_i \ge 1 \}\, .
  \end{equation}

  The following result  describes the impact of the behaviour of the block  system  on the system itself. 
  
  \begin{proposition} \label{blockgen} 
  Consider a  dynamical system $\cal S= (\cal I, T)$ of the class $\cal{DR}$, and its associated block dynamical system $\cal B= (\cal I, \widehat T)$. If  the block dynamical system $\cal B$  is generating,  then the dynamical system $\cal S$ is itself generating.
\end{proposition} 

\begin{proof} 
  We consider  an element of $\cal I^\ast $ and   the sequence $\delta_n(x)$ defined in Lemma \ref{lemgen}. Using Lemma \ref{lemgen}$(iii)$, we have to prove that the sequence $\delta_n(x)$ tends to 0. As the sequence $n \mapsto \delta_n(x)$   is weakly decreasing,    it is enough to  build  a subsequence   of the sequence $\delta_n(x)$  that tends to 0. Such a subsequence is associated with  a  subsequence $w_{(n)}$  of prefixes 
   of $M(x)$ for which   $ \delta({\cal I}_{ w_{(n)}}) $  tends to 0.  
 
 \smallskip As the infinite word $M(x)$ contains an infinite number of ones,    located at  indices $i_j$ with  a stricly increasing sequence $ j \mapsto i_j$,  the sequence $i_n$ tends to $\infty$ as $n \to \infty$. Consider,  for any $n$, the smallest prefix $w_{(n)}$ of $M(x)$ that contains a number of ones equal to $n$.   The prefix   $w_{(n)}$ is  thus written as a concatenation of $n$ blocks, and, as the block system $\cal B$ is generating, with  Lemma \ref{lemgen} Item $(ii)$, $ \lim_{n\to\infty} \delta({\cal I}_{ w_{(n)}})=0 $. This ends the proof \end{proof}

\subsection{The Good Class} 
 
For a  dynamical system $(\cal I, T)$ of the class $\cal{DR}$,  each  inverse branch  $g_m= a^{m-1} \circ b : {\cal I} \rightarrow {\cal J}_m$  of  the block dynamical system     is   surjective   of class ${\cal C}^2$. 
  But we need stronger properties for the block dynamical system, and we ask $(\cal I, \widehat T)$ to 
  belong to the Good Class, which has been already strongly used in \cite{BaVa} and in \cite{CeVa}.  The definition is now recalled:   

 \begin{definition} [Good Class] \label{good_class}
   Consider a   (tent-shaped) dynamical  system $\cal S$, its block  system ${\cal B}$,  its block  fundamental intervals ${\cal J}_m$, the set  $\cal G:= \{ g_m= a^{m-1} \circ b \mid m \ge 1\}$  of  inverse branches of depth one, and  the set ${\cal G}^n$  of  inverse branches of depth $n$. 
   
   \smallskip 
   The block dynamical  system    $\cal B$   belongs to the Good Class if the following holds \begin{itemize} 
 \item[$(a)$] The convergence abscissa  of its Dirichlet series 
 $ \Delta(s) = \sum_{m \ge 1} |{\cal J}_m|^s $
   is stricly less than 1.
  \item [$(b)$] The  system ${\cal B}$ is strictly contracting : there exists a contraction ratio $\eta <1$ 
  \vskip 0.1 cm
  \centerline{ $
  \eta = \limsup_{n \to \infty}  \eta_n , \qquad\hbox{ with} \quad  \eta_n = \max \{ |g'(x)|\mid  g \in  {\cal G}^n, x \in  I\} \, .$}
  \vskip 0.1 cm 
    
      \item[$(c)$]   Its inverse branches $g_m = a^{m-1} \circ b$  (of depth one)  satisfy the  {\rm bounded distortion}  property, 
      \begin{equation} \label{distora} \exists  L\ge 1, \ \forall m \ge 1, \ \forall (x, y) \in {\cal I}^2, \quad \hbox{one has} \quad   \displaystyle { \frac 1 L \le  \left| \frac {g'_m (x) }{g'_m(y)}\right| \le   L}\, .\end{equation}
      \end{itemize}
      \end{definition}

    There are some important remarks. 
    
    \smallskip
    -- First, a block system  of the Good Class is generating: this is  due to Property \ref{good_class}$(b)$ and Lemma \ref{lemgen}$(ii)$.
    
     \smallskip
      -- Second, we will prove  later on, in  Section \ref{DRGC},  that  the block  system of a system of Class $\cal{DR}$  already satisfies Item$(b)$ -- moreover,  in a  strong sense--.  
  Then, when we ask a system of the Class $\cal{DR}$ to have a block system of the Good Class,  the remaining  conditions to fulfill are   Conditions $(a)$ and $(c)$ of Definition \ref{good_class}.  They play an important role in dynamical analysis of Sections \ref{dynana1} and \ref{dynana2}.

\subsection{The  Class $\cal{DRBGC}$.} \label{DRVBGC}
 We now define our main class of interest:  

\begin{definition}  \label{interest}  {\rm [Class $\cal{DRBGC}$]} This is the class of   measured dynamical systems  $(I, T, \mu)$   of Definition \ref{ds}  for which 
\begin{itemize} \item[$(a)$]  the system $(I, T)$ belongs to  the class $\cal {DRI}$ (see Definition \ref{defDR});  
\item[$(b)$] the block system $(I,  \widehat T)$ belongs to the Good Class (see Definition \ref{good_class}).

\end{itemize}
\end{definition}   We deal  with the Class $\cal{DRI}$ in Item $(a)$, because we are mainly interested in systems with zero entropy (See Section \ref{waitbis}).  We recall that there is no hypothesis asked for the measure $\mu$ except that it fulfills hypothesis \ref{ds}(e) of Definition \ref{ds}.

 \subsection{An emblematic instance: the Farey system.}  \label{Fareycase}
 
    The Farey   dynamical system   plays an important role in Number Theory,  notably because its block dynamical system coincides with the Gauss map,  as we recall it later on.   This  also  provides a natural example of dynamical system with intermittent type,  
    that is studied for instance in  Prellberg's paper \cite{Pre}. We will see that  the Farey dynamical system belongs to the $\cal{DRIS}$ class,  introduced and studied in Section \ref{sec:dri}, and even to the Class $\cal{DRIL}$ of Section \ref{sec:DRIL}. 
 
 \smallskip 
 The Farey map $T$  possesses  two  complete inverse branches $a$ and $b$
   \renewcommand{\arraystretch}{1.3}
\[\left\{\begin{array}{l ll }  a : [0, 1]  \rightarrow [0,1/2]& \quad  a(x)={x}/(1+x)\, , \cr 
   b: [0, 1] \rightarrow  [1/2,1]& \quad  b(x)= 1/(1+x)\, .\cr
  \end{array}\right. 
  \]   
 As  the iterates $a^n$ of the  first inverse branch $a$ satisfy  $a^{n}(x) = x/(1+ n x)$, the sequence  $q(n)= a^n(1) $ is $q(n)= 1/(n+1)$ and $(a^n)'(x) =  1/(1+ nx)^2$.    Moreover, as  the derivative  $b'$ of the  second inverse branch $b$ satisfies  $b'(x) = -1 /(1+x)^2$,  the Farey source thus belongs to the class $\cal {DRI}_w$.

 \smallskip
 The block system   of the Farey system
  is   defined via the partition 
  ${\cal J}_m = [q(m) ,  q(m-1)]$  for $m \ge 1$ with $q(m) = 1/(m+1)$ and the local inverses $g_m$ are $g_m = a^{m-1} \circ b$.  For any $m \ge 1$, the branch 
$a^{m-1}(x)$ equals  $ x/(1+ (m-1) x)$, and   the equality $$a^{m-1}\circ b(x) = \frac {b(x)}{1+ (m-1) b(x)} = \left(\frac 1 {1+x} \right) \left( \frac {1+x}{1+x+ (m-1) }\right) = \frac 1 {m+x} $$
proves that $a^{m-1} \circ b$ coincides with the linear fractional transformation (LFT in shorthand) $g_m : x \mapsto 1/(m+x)$ that  is  indeed the inverse of the Gauss map  on the interval ${\cal J}_m$.  One has  thus proven  that  the block  map $\widehat T$  of the Farey map 
is the  Gauss map, 
$$ \widehat T(0) = 0, \qquad \widehat T(x) = \left \{  \frac 1 x  \right \}\quad (x \not = 0)\, .$$ 

\begin{figure}[h]
        \centering
        \includegraphics[scale= 0.55]{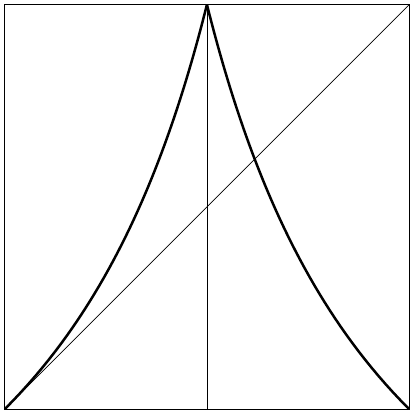}\hspace{20pt}
        \includegraphics[scale=0.55]{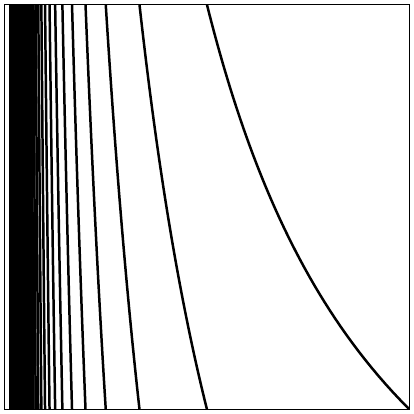}
        \caption{The   Farey map and the Gauss map}
        \label{fig:Fa-Ga}
    \end{figure}

\section{Statement of the four main results.}  \label{main}

This Section states the four main results of the paper that are all obtained for a source of the Class $\cal{DRBGC}$ introduced in  Section \ref{DRVBGC}.  The first main result  (Theorem \ref{inv}) describes invariant densities, of the block source, and of the source itself. Section \ref{renewal} 
states the second  main result of the paper (Theorem \ref{proexpgenfinbis})   that   relates the generating functions $N_\mu$ and $M_\mu$ with the generating function $Q_\mu$ associated with  the distribution of the waiting time. It can be viewed as a renewal result.  Singularity analysis, described in Section \ref{singana}, performs a round trip  between singularities of generating  functions and their coefficients.  It deals in a natural way with slowly varying functions precisely described in Section \ref{sec:SV}. 
Section \ref{coeff}   returns 
to the mean values of interest, the mean number of ones and the Shannon weights,  and obtains the third  main  result, stated as Theorem \ref{final1}.
Section  \ref{firstexh} describes the path for obtaining sources  with prescribed Shannon weights, and gives   a summary of main results  obtained later on  in Sections \ref{sec:dri} and \ref{sec:DRIL}.

%%%%%%%%%%%%%%%%%%%%%%%%%%%%%%%%%%%%%%%%%%%%%%%%%%%%%%%%%%%%
 \subsection{The  first main  result: invariant densities.}

 \begin{theorem}   \label{inv}   Consider a  measured  dynamical system  $(I, T, \mu)$   of the Class $\cal{DRBGC} $, and denote  by $W$ the waiting time of symbol one.  The following holds 
 \begin{itemize}  
 
\item[$(a)$]  The dynamical system  $(I, \widehat T)$ 
 has an  invariant  probability measure  $\nu$  with a strictly positive density $\psi$ in  ${\cal C}^1 ([0, 1])$    and a finite entropy ${\cal E}_\nu(\cal B)$. 
 
\item[$(b)$]  The dynamical system  $(I, T)$  admits an  invariant function $\phi_0$, whose associated measure $\int_{\cal I} \phi_0(t) dt $ equals  $\E_\nu[W]$. 
When   the mean $\E_\nu[W]$ is finite, the measure   $\pi$   associated  with $\rho = (1/\E_\nu[W])\phi_0$  is an invariant  probability  measure and $\rho$  satisfies $\lim_{x \to 0} \rho (x) = \infty$. 
\end{itemize} 
  \end{theorem} 
  
  Theorem \ref{inv} will be proven in Section \ref{dynana1}.

  %%%%%%%%%%%%%%%%%%%%%%%%%%%%%%%%%%%%%%%%%%%%%%%%%%%%%%%%%%%%

  \subsection{The second main result:  renewal equations} \label{renewal}
 We   study  two costs  in parallel:  the Shannon weight $m$ defined in Section \ref{weight1},  and the number $n$  of ones  introduced in Section \ref{weight2}.   The following  theorem relates    the behaviours   of the four generating  functions of interest $N_\mu(v), M_\mu(v), Q_\mu(v), Q_\nu(v)$,  as  the real $v \to 1^-$. 
  
   \begin{theorem}\label{proexpgenfinbis} {\rm [Generating functions $N_\mu(v)$ and  $M_\mu(v)$]}  Consider  a   measured dynamical system $\cal S = (\cal I, T, \mu)$  of the class  $\cal{DRBGC}$, the invariant  measure $\nu$  of  the block system $\cal B = (I,  \widehat T)$, and  its  finite entropy ${\cal E}_\nu(\cal B)$.      Denote by $W$ the waiting time of symbol one, and by  $ Q_\nu(v)$ the generating functions  of the wtd sequence defined in \eqref{Q} relative to  %the measure $\mu$ or 
   the invariant measure $\nu$ of the block system $(\cal I, \widehat T)$.
   
  \smallskip  Consider the two  generating functions  --$N_\mu(v)$ of  number of ones,  $M_\mu(v)$ of Shannon weights-- resp defined, as in \eqref{gengenfun}, from  the number $n$  of ones  introduced in Section \ref{weight2} and the Shannon weight $m$ defined in Section \ref{weight1}. 
  
   \medskip
    The following holds, 
 as  the real $v$ tends to $1^-$,  
 
 \smallskip $(A)$ the   dominant behaviours  of  the two  generating functions $N_\mu(v)$ and $M_\mu(v)$    are related 
  to the dominant behaviour of  the  generating function  $Q_\nu(v)$,  
 \begin{equation} \label{beh}
  (1-v) N_\mu(v) \sim_{v \to 1^-}     \frac {1} { (1-v)  Q_\nu( v)}, \quad 
(1-v) M_\mu(v)  \sim_{v \to 1^-}      \frac {  \cal E_\nu({\cal B})}  { (1-v)  Q_\nu( v)} \, . 
\end{equation}

\smallskip $(B)$ 
Consider  the generating function  $ Q_\mu(v)$ of the wtd sequence relative to $\mu$;  then 
\begin{equation} \label{Dmu} 
 \hbox{the ratio}  \ \   \displaystyle \frac {Q_\mu(v)} {Q_\nu(v)}  \hbox{ has a finite limit $D_\mu$ as $v \to 1^-$   and }
  \end{equation}
  \begin{equation} \label{behbis}
 (1-v) N_\mu(v) \sim_{v \to 1^-}     \frac {D_\mu} { (1-v)  Q_\mu( v)}, \ \  
(1-v) M_\mu(v)  \sim_{v \to 1^-}      \frac { D_\mu  \cal E_\nu({\cal B})}  { (1-v)  Q_\mu( v)} \, .
\end{equation}

\end{theorem}

 The proof of  Theorem \ref{proexpgenfinbis} will be done in Section \ref{dynana2}.  We  first remark   that the dominant behaviours in \eqref{beh} do not depend on the measure $\mu$.    We  also  give an interpretation of  the    first relation in  \eqref{behbis} that 
  can be viewed   as a renewal equation between the two generating functions $(1-v) Q_\mu(v)$ and $(1-v) N_\mu(v)$ and the partial sums of their coefficients. 
  
  \smallskip
  With Lemma  \ref{diffe}, the   partial sums of coefficients  of  the two generating functions $(1-v) Q_\mu(v)$ and $(1-v) N_\mu(v)$  are resp.   equal to $q_\mu(n)$ and $\underline n_\mu(n)$ with 
 \begin{itemize} 
 \item[$(a)$]
  $q_\mu(n)$ = the probability  that the  waiting time of the   set ${\cal J}$   is at least equal to $n$, 
  \item[$(b)$] 
  $\underline n_\mu(n)$ = the mean  value of the time spent in the set ${\cal J}$ during the first $n$ steps.
  \end{itemize}
  
  We now  express these gf's in terms of the set ${\cal J}$
  which gathers the infinite words which begin with the symbol one. The relations  hold
    $$   r_\mu (k)=q_\mu(k-1)-q_\mu(k) = \mu\left[ T^{-k } {\cal J}  -\left(  \cup_{\ell <k}  T^{-\ell} {\cal J} \right) \right], \qquad  \underline n_\mu (k)-\underline n_\mu(k-1)  =\mu [T^{-k} {\cal J}]\, ,  $$ 
  and   entail the  two equalities
     $$ (1-v)Q_\mu(v) = 
     1 - \sum_{k \ge 1}  r_\nu(k) v^k = 
1- \sum_{k \ge 1} \mu\left[ T^{-k } {\cal J}  -\left(  \cup_{\ell <k}  T^{-\ell} {\cal J}\right) \right]  v^k\, , $$
     $$(1-v) N_\mu(v) = \underline n_\mu(0) + \sum_{k \ge 1} (\underline n_\mu (k)-\underline n_\mu(k-1) ) v^k  = \sum_{k \ge 1} \mu[T^{-k} {\cal J}] v^k \, .$$  Then, the renewal equation is 
     \begin{equation} \label{ourren}
      \left( \sum_{k \ge 1} \mu[T^{-k} {\cal J} ] v^k \right) \cdot  \left(  1- \sum_{k \ge 1} \mu\left[ T^{-k } {\cal J}  -\left( \cup_{\ell <k}  T^{-\ell} {\cal J}\right) \right]  v^k\right)  \sim_{v \to 1}   D_\mu  \, .
      \end{equation}
     In particular, when  the initial measure is  the  invariant measure of the block dynamical system, the ratio  $D_\mu$
     equals 1.

  \medskip
    Aaronson \cite{Aa}  obtains (another)   renewal equation through the notion of Darling-Kac sets that is described in Section \ref{dynana2}.   In the present paper, we directly obtain the previous renewal  relation with  the derivatives of  the trivariate   function $\Lambda_\mu $.

\subsection{Singularity analysis: Towards the third main result} \label{singana}

 The   general singularity  analysis process   relates the asymptotics of a sequence   and the behaviour of its generating function near the dominant singularity. Generally speaking, 
 complex analysis is very often needed. However, in the present context, due to strong positivity properties, it is enough to deal with real analysis, as  we now explain. 
 
\medskip 
We   perform  a ``round trip'' between the behaviour    (as $n \to \infty$) of a  positive sequence $q_\mu(n)$ 
and the behaviour (as $v$ real tends to $1^-$) of its generating function  $Q_\mu(v)$  
that we now describe more precisely.

\smallskip
We begin with a  (positive) sequence $q_\mu(n)$ 
and study the partial sum $Q^{(\mu)}_n= q_\mu(0)+ q_\mu(1) + \ldots + q_\mu(n-1)$.   There are three steps : 
\begin{itemize} 
\smallskip
\item [$(A)$] [Abelian step] 
 {An  Abelian Theorem } 
 describes  the behaviour of the  generating  function $Q_\mu(v)$   when $v\to 1^-$ along the positive real line from the  asymptotic behaviour of the sequence $Q^{(\mu)}_n$;   
 \smallskip
\item[$(B)$]  Using  Theorem  \ref {proexpgenfinbis}    provides  the behaviour of the two generating   functions $(1-v)N_\mu(v), (1-v) M_\mu(v)$   when $v\to 1^-$ along the positive real line. 
\smallskip
\item[$(C)$] [Tauberian step] We return to the asymptotic estimate  of   the partial sums of  coefficients  of the functions  $(1-v)N_\mu(v), (1-v) M_\mu(v)$  --that resp. coincide with  $\underline n_\mu(n)$,  $\underline m_\mu(n)$-- via  {a Tauberian Theorem}. 
 \end{itemize}

\medskip 
   The following central theorem  (Theorem  \ref{HLKboth}) provides the Abelian step $(A)$ via its implication [$(2) \Longrightarrow (1)$] and the Tauberian  step $(C)$ via its implication [$(1) \Longrightarrow (2)$].  It deals with the set $\cal{SV}$ of slowly varying functions that is now defined.   
         
   \begin{definition}
    A  positive function $V$ defined on $]0,  \infty[$ is said  to be slowly varying
    
    \begin{itemize} 
    \item[$(a)$]  at infinity      if,  for any $t >0$, one has $V(tx)/V(x)\to 1$ as $x \to \infty$.
    
    \item[$(b)$]   at $0$ 
    if $x \mapsto V(1/x)$  is  slowly varying at infinity.
    \end{itemize} The set of  slowly varying functions at infinity is denoted by  $\cal{SV}$ and  the set of  slowly varying functions at zero by  $\cal{SV}_0$. 
     \end{definition} 
        
   The sets $\cal{SV}$ and $\cal{SV}_0$ of slowly varying functions  play an important role in the sequel of the paper.  Main examples  of such functions are provided by powers (positive or negative) of logarithms.   Main properties of such functions are described in Section \ref{sec:SV}. See also the book  \cite{Fe} Section V.III.8 p. 276.

 \begin{theorem} \label{HLKboth} {\rm \cite{P}, \cite{HL}, \cite{K}, revisited in \cite{Fe} section XIII.5 p. 447.}
 
 Consider a real $\rho\ge 0$,  a 
 sequence  $q(n) \ge 0$ and its partial sums $Q_n =  q(0) +q(1)+ \ldots + q(n-1)$ together with  a function $U$ slowly varying at infinity.    Assume that  the generating function $Q(v)=\sum q(n)v^n$ is convergent for $ v \in [0, 1[$.
 
 \smallskip The  two assertions (1) and (2), are equivalent  $$  (1) \quad  Q(v) 
 \sim_{1^-}   \frac 1 {(1-v)^{\rho}}\,  U\left( \frac 1 {1-v}\right)  \quad 
   (2) \quad  Q_n\sim_{\infty} \frac 1{ \Gamma(\rho +1)} n^\rho\,  U(n) \ .
  $$
 \end{theorem}

   \medskip
   This strategy of proof  ``with a round trip''  has already been used  in different contexts. The  steps $(A)$ and $(C)$  are  similar  as here,  but  the ``middle step'' $(B)$   is a relation between  two generating functions that may differ from here.      
 For instance, with a middle step $(B)$ of  another flavour, already described  at the end of previous Section and made more precise in  Section \ref{dynana2}, Aaronson   obtains distributional limit theorems in Infinite Ergodic Theory \cite{Aa2}.  
   
   \smallskip Here,   the step $(B)$   is directly  obtained via  operations on generating functions.

\subsection{Sequences of interest}  \label{seq-int}
 When dealing with Theorem \ref{HLKboth}, there are two main cases, the case when $\E_\nu[W]$ is finite or the case when it is infinite.    \\
  In the  case when  $\E_\nu[W]$ is infinite, 
   we  are interested in wtd   sequences $q_\nu
 (n)$   which tend to zero and   whose associated series is divergent. We then   consider       sequences  which  are written in terms of sequences of the Class $\cal{SV}$ as        \begin{equation}  \label{qdiv1} q_\nu(n) = n^{-\beta}\,  V_n ,   \quad V_n = %V_n^{(\mu)}= 
 V(n) \quad \hbox{where $ V$ satisfies }       \end{equation}
         \begin{itemize} 
       \item[$(i)$]  $ V\in \cal {SV} $ when $  \beta  \in ]0, 1[ $; 
       \item[$(ii)$]  $V \in \cal{SV} $ satisfying $ V_n = \Theta( (\log n)^\delta)$  with $\delta >-1  $  when $\beta = 1$; 
       \item[$(iii)$]   $V \in \cal{SV} $ with $V_n \to 0$   and $\sum_{k<n} V_k  \to \infty $ when $\beta = 0$ .
       \end{itemize}  
       
       When we wish  more  explicit results,  we  focus on  sequences  $q_\nu(n)$  that satisfy 
      \begin{equation}  \label{qdiv} q_\nu(n)  \sim  K_\nu \,  n^{- \beta} 
  ( \log  n)^{\delta} \, , 
   \end{equation} 
  where the   parameter $(\beta, \delta)$ belongs to the set 
\begin{equation}  \label{condSR}  \Gamma_Q = \left( ]0, 1[ \times {\mathbb R} \right) \bigcup \left( \{1\} \times ]-1, + \infty[)\right) \bigcup\left( \{0\}\times ]  -\infty,  0[\right) \, .
  \end{equation}

   \subsection{The third main result}    \label{coeff}
 The  third    main result of the paper  relates  estimates of  number of ones and Shannon  weights with estimates of the waiting time distribution:

\begin{theorem} \label{final1}
{\rm [Number of ones and Shannon weights]} Consider  a   measured dynamical system $\cal S = (\cal I, T, \mu)$  of the class  $\cal{DRBGC}$, the invariant  measure $\nu$  of  the block system $\cal B = (I,  \widehat T)$,  and the  finite entropy ${\cal E}_\nu(\cal B)$.    Denote by $W$ the waiting time of the symbol one.  

    \medskip
    The following holds:

 \smallskip
  (A)  In the case when  $\E_\nu[W]$ is infinite  and 
 the distribution $q_\nu(n)$ of the waiting time $W$    satisfies  \eqref{qdiv1},  
the   average number of ones $\underline m_\mu(n)$ is $o(n)$, 
and  
the   asymptotic estimates   hold for  the Shannon weights $\underline m_\mu(n)$,  \begin{equation} \label {heu} \underline m_\mu(n)  \sim {\cal E}_\nu(\cal B) \, \underline  n_\mu(n) \, . \end{equation}

\smallskip (B)  In the case when  $\E_\nu[W]$ is infinite and the sequence    $q_\nu(n)$ satisfies  \eqref{qdiv},  with  a triple  $(K_\nu, \beta, \delta)$ that satisfies \eqref{condSR}, 
then the  Shannon weights  admit the following asymptotic  estimates 
\renewcommand{\arraystretch}{1.6}
 \begin{equation} \label{estN2}
  \underline m_\mu(n) \sim  \left \{\begin{array} {ll} \displaystyle 
 \frac { {\cal E}_\nu(\cal B)} {K_\nu} \,   \  
    \frac { ({\delta +1}) n^{\beta}}{(\log n)^{\delta +1}} \    &(\beta = 1) \cr
 \displaystyle  \frac {{\cal E}_\nu(\cal B)}  {K_\nu}  \,   \ 
 \frac  1 {\Gamma (\beta +1 ) \Gamma (1-\beta)} \,  \frac{n^{\beta} }{(\log n)^{\delta }}  &  (0\le  \beta < 1)  \cr 
  \end{array} \right\}.
  \end{equation}
  
   \medskip  (C) When $\E_\nu[W]$ is finite,  
    the average  number of ones  and the Shannon weights  are of linear order,  
   \begin{equation} \label{estN3}
  \underline n_\mu(n) \sim \frac 1 {\E_\nu[W]}  \, n, \quad \underline m_\mu(n) \sim   \frac {\cal E_\nu(\cal B)}  {\E_\nu[W]} \, n \, ,   
  \end{equation} 
  the entropy  rate  $\cal E_{\underline \mu} $ of the source $(\Sigma^\mathbb N, \underline \mu)$  exists,  and satisfies 
 $   \cal E_{\underline \mu} =  \displaystyle  {\cal E_\nu(\cal B)} / {\E_\nu[W]}$.

 \end{theorem} 

\smallskip
We first  remark that all the previously described estimates do not depend on the measure $\mu$. Moreover,    with Assertions $(A) $ and $(C)$, and because we are  mainly interested in sources $\cal S$ of zero entropy, we are (mainly) led to sources whose wtd sequence $q_\nu(n)$ is associated with a series $Q_\nu(v)$ which tends to $\infty$ as $v \to 1^-$.

\smallskip
Second,  there exists   a heuristics that explains the estimate proven in Assertion $(A)$ as follows.  Consider a  random prefix of length $n$ emitted by the source $S$ endowed with probability $\mu$.  This prefix contains an ``average''  number of ones (and thus an ``average''  number of full blocks)  equal to $\underline n_\mu(n)$. Each full block brings an ``average'' amount of information equal to ${\cal E}_{\nu} (\cal B)$. Then the total  ``average'' amount of information brought by  a random prefix of length $n$, equal  (by definition) to $\underline m_\mu(n)$ satisfies  \eqref{heu}. 
  However,  we do not succeed in transforming this heuristics into a direct probabilistic proof. In particular,  as  the prefix  may be not a full block, its tail  is  not easy to deal with. The {(indirect)} analytic proof developed in Sections  \ref{dynana2} and \ref{3-4} is based on the  dominant expressions of the generating functions $N_\mu$ and $M_\mu$ near $v=  1^-$ that  both  involve  a double quasi-inverse (see Eqns \eqref{Nmu1}  and \eqref{Mmu1}). 
  
\smallskip Third,   Assertion $(C)$ cannot be applied  to the invariant  measure $\pi$ of the source $\cal S= (I, T)$ studied in Theorem \ref{inv}. As it is proven in Theorem \ref{inv},  its density $\rho$ is infinite at $x = 0$, and does not belong to $\cal {C}^1([0, 1])$.  If Assertion $(C)$ could be applied to measure $\pi$, this  would entail  the estimate 
$$ m_\pi (n) \sim  n \cdot {\cal E_\pi}(\cal S), \quad \hbox{and thus} \quad  {\cal E_\pi}(\cal S) =   \frac {\cal E_\nu(\cal B)}  {\E_\nu[W]} \, , $$
  the last equality being  known as Abramov formula. 

\smallskip
The proof of Theorem \ref{final1} is provided in Section \ref{3-4}.

%%%%%%%%%%%%%%%%%%%%%%%%%%%%%%%%%%%%%%%%%%%%%%%%%%%%%%%%%%%%%%

\subsection{Towards the fourth result.}  \label{firstexh}
 
Consider   a source $\cal S$  of the Class $\cal{DRBGC}$, with
  the pair    $(\beta_Q, \delta_Q)$ of its wtd sequence  that  belongs to the set $\Gamma_Q$ defined in   \eqref{condSR}. Item $(B)$ of  Theorem \ref{final1} proves that  its 
 Shannon weights  are of the general form 
 \begin{equation} \label{mmu0}
  \underline m_\mu(n)  = \Theta\left( {n^{\beta_M}} {(\log n)^{\delta_M}}\right)\, , 
   \end{equation}
     where the pair $(\beta_M, \delta_M)$  is defined from the pair $(\beta_Q, \delta_Q)$ via the injective  map  defined on $\Gamma_Q$,    \begin{equation} \label{QM}  (Q \rightarrow M): (\beta_Q, \delta_Q) \mapsto (\beta_M, \delta_M) \quad \hbox{with}
   \end{equation}
 $$ \beta_M = \beta_Q \quad \hbox{ and}\quad  \left[   \delta_M = -\delta_Q   \quad (\beta_Q <1) \quad \hbox{or} \quad  \delta_M = - (\delta_Q+1)  \quad (\beta_Q =1)\right] \, , $$ 
 whose image is   \begin{equation} \label{GammaM}
   \Gamma_ M = 
  \left( ]0, 1[ \times {\mathbb R} \right) \bigcup \left( \{1\} \times ]- \infty, 0[\right) \bigcup \left( \{0\} \times ]0,  +\infty[\right) .
   \end{equation}
As we wish to  exhibit a source $\cal{S} $   of the class $\cal{DRBGC}$ for which Shannon weights satisfy Eqn \eqref{mmu0},   our first task is   to prove    our three main  Theorems  (Theorem \ref{inv}, Theorem   \ref{proexpgenfinbis} and Theorem \ref{final1}). Then,   if we know  a source $\cal S$ of the class $\cal{DRBGC}$   with   a prescribed wtd of pair $(\beta_Q, \delta_Q)$, we  may use  the mapping  described in \eqref{QM}.

\medskip
In  Sections \ref{sec:dri} and \ref{sec:DRIL}, we will  introduce the two following classes $\cal{DRIS}$ and $\cal{DRIL}$   and obtain the following results, along  the following main steps: 

\medskip  
 {\rm [Exhibiting a source $\cal S$ with prescribed Shannon weights]}

\begin{itemize}  
\item[$(i)$] The $\cal {DRIS} $ Class introduced   in Section \ref{sec:dri} satisfies the following: 
\begin{itemize} 
\item[$(a)$] The inclusion $\cal {DRIS} \subset \cal{DRBGC}$ holds [proven in  Theorem  \ref{deltaa}];  
. 
\item[$(b)$]  For a  source of  the $\cal{DRIS}(\gamma)$ Class,  with $\gamma >1$, the expectation $\E_\nu[W]$   is infinite. For $\gamma<1$,  the expectation $\E_\nu[W]$   is finite [see  Theorem 
  \ref{qneval}]; 
\end{itemize} 

\item[$(ii)$] The class $\cal{DRIL} $ introduced  in Section  \ref{sec:DRIL} is a subclass of the class $\cal {DRIS}$.  A   source of the  $\cal{DRIL} (\gamma, \delta)$  Class has a prescribed wtd described in Proposition  \ref{Aar1}. 

\item[$(iii)$]  Theorem  \ref{exh1} in Section \ref{sec:DRIL} gives a final answer to our ``exhibition''. 

\end{itemize}

 %%%%%%%%%%%%%%%%%%%%%%%%%%%%%%%%%%%%%%%
 %%%%%%%%%%%%%%%%%%%%%%%%%%%%%%%%%%%%%%%

\section{Generating operators and generating functions. The first main theorem.} \label{dynana1}

This    Section aims at proving  Theorem \ref{inv} 
with tools of dynamical analysis.
  The first three subsections     
       present the main  principles of dynamical analysis,    where  {\it secant}  transfer operators   generate probabilities. 
        Then, Section \ref{dyn-gen}  
provides  in   Theorem  \ref{triA}   an  alternative expression of the  trivariate generating function $\Lambda_\mu (v, t, s)$ for a dynamical source   of Class $\cal {DR}$ in terms of   these secant  transfer  operators, together with the secant  function $S[\phi]$ associated   with the density $\phi$ of the measure $\mu$. In the following of the Section (Section \ref{firstgoodclass}),  
the main results on unweighted transfer operators are  first presented, as  they are obtained in previous  works of Cesaratto and Vall\'ee \cite{CeVa}. These previous works  entail that a block source of the Good Class admits an invariant measure $\nu$; we then  focus here on the properties of the entropy, and prove that the entropy  $\cal E_\nu(\cal B)$ follows the Rohlin formula.  \\ 
 The remainder of the Section is devoted to the proof of Theorem \ref{inv}$(b)$. .

%%%%%%%%%%%%%%%%%%%%%%%%%%%%%%%%%%%%%%%%%%%%%%%%%%%%%%%%%%%%%%%%%%%%%%%%%%%%%%%%%%%%%%%%%%%%%%%%
   
   \subsection{Generating operators} \label{trans}
     
 For a complex number $s \in {\mathbb C}$, we associate with   a map $g$  of class ${\cal C}^1(\cal I)$ 
  the tangent transfer operator  ${\mathbf M}_{s} \langle g\rangle $  defined as
$$ {\mathbf M}_{s} \langle g\rangle [f](x) := |g'(x)|^s  \, f(g(x))\, , $$
that transforms a function $f$ defined on $\cal I$ into   the function $ {\mathbf M}_{s} \langle g\rangle [f]$ defined on $\cal I$. 
 In the same vein,  the 
 secant  transfer  operator  ${\mathbb M}_{s} \langle g\rangle$    
  $$ {\mathbb M}_{s} \langle g\rangle [F](x, y) :=  \displaystyle {\left|\frac{g(x)-g(y)}{x-y}\right|^s}\, F(g(x), g(y)) \,, $$
  transforms  a  function 
$F$  defined on ${\cal I}^2$  into another function $ {\mathbb M}_s\langle g\rangle [F]$  defined on ${\cal I}^2$,
  and the equality holds  between the two versions of transfer operators, 
  \begin{equation} \label{tang-sec} 
  {\mathbb M}_{s} \langle g\rangle [F](x, x)  = {\mathbf M}_{s} \langle g\rangle[{\rm Diag\, } F](x), \quad \hbox{with} \quad  {\rm Diag\, } F (x) =   F(x, x)\, . 
  \end{equation}
  We   focus on  a  dynamical system $\cal S = (\cal I, T)$ as defined in Definition \ref{ds},   consider the case when $g = h_w$ is an inverse branch  (of any depth)    associated with a word $w$, and use  the  following notations  that  insist on the dependence with respect to words, 
  \begin{equation} \label{Hsw} 
{\mathbb H}_{s} \langle w\rangle :=  {\mathbb M}_{s} \langle h_w \rangle, \quad  {\mathbf H}_{s} \langle w\rangle :=  {\mathbf M}_{s} \langle h_w \rangle \, .
\end{equation}
There are 
  now   fundamental morphisms,  due to relations
 \begin{equation} 
 \label{mor}  
    {\mathbf H}_{s} \langle u \cdot w\rangle =  {\mathbf H}_{s} \langle w \rangle \circ {\mathbf H}_{s} \langle u \rangle, \quad     {\mathbb H}_{s} \langle u \cdot w\rangle =  {\mathbb H}_{s} \langle w \rangle \circ {\mathbb H}_{s} \langle u \rangle \, . \end{equation}
      The transfer operators of the system $(\cal I, T)$  are (with their two versions, tangent or secant) 
     \begin{equation} \label{Lgene}
      \sum_{w \in \Sigma}  {\mathbf H}_{s} \langle w\rangle, \quad 
    \sum_{w \in \Sigma}  {\mathbb H}_{s} \langle w\rangle \, . 
    \end{equation}

      \subsection{Quasi-inverse $(I-{\mathbb L}_{s, t})^{-1}$.} 
We  now focus on  binary dynamical systems of Definition \ref{ds}. We often replace the symbol $0$ by the symbol $a$, and $1$ by $b$, in order to deal with  more readable formulae, and 
  we  use the  shorthand notations, 
  \vskip 0.1 cm 
  \centerline{
  $ {\mathbb A}_s = {\mathbb H}_s\langle a \rangle,  \qquad  {\mathbb B}_s =  {\mathbb H}_s\langle b \rangle\, ,  $}
  \vskip 0.1 cm
  with an easy adaptation to the tangent versions. 
   The secant  transfer operators of the  system are defined as  the sums   
 \begin{equation} \label{L}
 {\mathbb L}_s=  {\mathbb A}_s  +  {\mathbb B}_s,    \qquad   {\mathbb L}_{s, t}=  {\mathbb A}_s  +  t{\mathbb B}_s \, . 
 \end{equation} 
 Here,  in the second expression,  the formal variable $t$ marks the symbol one.

 There is first a formal  direct decomposition  of the quasi-inverse, 
  $$   (I- v{\mathbb L}_{s, t})^{-1}   =  (I- v( {\mathbb A}_{s} +  t\, {\mathbb B}_s) )^{-1} =  ( I- v{\mathbb A}_s)^{-1} \circ (I- t \, ( v {\mathbb B}_s \circ (I-v {\mathbb A}_s)^{-1} )^{-1} \, , $$
      that  introduces the  (weighted) operator 
      \begin{equation} \label{formalgvs}
      {\mathbb G}_{v, s} = v \, {\mathbb B}_s \circ (1-v {\mathbb A}_s)^{-1}  \, , 
      \end{equation}
      and leads to the writing 
      \begin{equation} \label{formaldec}
         (I- v{\mathbb L}_{s, t})^{-1}  =  ( I- v{\mathbb A}_s)^{-1}  (I- t \mathbb G_{v, s})^{-1} \, . \end{equation}
 We now introduce  another writing of the previous quasi-inverse, using the   two costs   defined on $\Sigma^\star$,  the  length $|w|$ of $w$,  and the number $n(w)$ of ones.  As each of these costs  is additive\footnote{For two words $w$ and $w'$, one has $|w \cdot w'| = |w| +|w'|$, and $n(w\cdot w') = n(w) +n(w')$},   and due to the morphism  
  \eqref{mor}, it comes 
 \begin{equation} \label{Ln}
 v^n {\mathbb L}_s^n=\sum_{w\in\Sigma^n} v^{|w|}  {\mathbb H}_s\langle w \rangle,    \qquad   v^n {\mathbb L}_{s, t}^n=  \sum_{w\in\Sigma^n} v^{ |w|} t^{n(w)}{\mathbb H}_s\langle w \rangle \, ,   \end{equation}
 and the quasi-inverse  also writes as  
\begin{equation} \label{dec-q-inv}    (I- v{\mathbb L}_{s, t})^{-1}   =  \sum_{w\in \{0, 1\}^\star}  v^{|w|}\, t^{n(w)}\,   {\mathbb H}_s\langle w \rangle \, .
\end{equation}

\subsection{The generating function $\Lambda_\mu$ and the secant $S[\phi]$}
 With  costs $c \in \{m, n\}$ defined in  Sections \ref{weight1} and \ref{weight2},  we associate the expectation $\underline c_\mu(k)$ as in Eqn \eqref{costCk} and the generating function $C_\mu(v)$ as in Eqn \eqref{gengenfun}. 
 We look for alternative expressions of  the two generating  functions $M_\mu, N_\mu$.  In this way, their dominant behaviour (their behaviour near their dominant singularity) becomes apparent, and, via singularity analysis, we 
 relate the asymptotics of the two sequences  $\underline m_\mu(k)$ and $\underline n_\mu(k)$ ($k \to \infty$) with  the dominant behaviours of $M_\mu$ and $N_\mu$.
 
 \medskip In order to obtain  alternative expressions for  functions $N_\mu$ and $M_\mu$,  it proves convenient  to  first deal with the following  trivariate generating function 
 \begin{equation} \label{Lambda} \Lambda_\mu  (v,t,s)=\sum_{w\in  \{0, 1\}^\star }v^{|w|}\, t^{n(w)}\, p_\mu(w)^s\, , 
\end{equation}
 where, as previously,    the variable $v$ marks the length $|w|$ of the word $w$ (its number of symbols), and  the variable $t$ marks the number of ones in $w$. 
 We remark  the  easy (and useful) equality, 
 \begin{equation} \label{tjrs} 
 \Lambda_\mu  (v,1, 1) = \sum_{n \ge 0} v^n \sum_{w\in  \{0, 1\}^n }p_\mu(w) =  \sum_{n \ge 0} v^n = \frac 1 {1-v}.
 \end{equation} 

Assume  that we have obtained, in the case of a source of interest,   
an alternative expression for $\Lambda_\mu$. 
 The two generating functions  $N_\mu, M_\mu$ are   indeed (formal)  derivatives  of the function $\Lambda_\mu$, 
 \renewcommand{\arraystretch}{1.5}
 \begin{equation} \label{MN}  \left\{ \begin{array} {lllll}
N_\mu(v) &= {\displaystyle   \sum_{w\in \{0, 1\}^\star}}\,  v^{|w|} \,  n(w)\, p(w)& = 
{\displaystyle\frac{\partial}{\partial t}}\Lambda_\mu(v,t,s)|_{s = 1,t = 1}\, ,\cr
M_\mu(v)&={\displaystyle   \sum_{w\in\{0, 1\}^\star}} v^{|w|}\,  |\log p(w)| \, p(w)
&=- {\displaystyle \frac{\partial}{\partial s}}\Lambda_\mu (v,t,s)|_{s = 1, t= 1}\, , \cr
\end{array} \right \}
\end{equation}
 and  the alternative expression for $\Lambda_\mu$  may be used to obtain expressions for  generating functions $M_\mu$ and $N_\mu$.

\medskip  
  We have  already obtained   alternative expressions  for the quasi-inverse $(I-{\mathbb L}_{s, t})^{-1}$ in  Eqn \eqref{formaldec} and \eqref{dec-q-inv}. 
  We  now  return to the function $\Lambda_\mu(v, t, s)$ itself, with the expression \eqref{Lambda}, and  wish to obtain an alternative expression of $\Lambda_\mu$.

\medskip When $\mu$ is a measure of density $\phi$,   this is the role of the function $S[\phi]$: 
 
  \begin{definition} { \rm [The function $S[\phi]$]} \label{defSphi}   With a function $\phi$ strictly positive   that belongs to ${\cal C}^1([0, 1])$, 
 one associates  the function $S[\phi]$  defined on 
    $ [0, 1]^2$ by  
    \begin{equation} \label{S}
   S[\phi]  (x, y) =   \frac 1 {y-x} \int_x^y \phi(t) dt  
     \, \ \ (x\not = y), \quad S[\phi](x, x) = \phi (x)\,  \ \  .
   \end{equation}
   \end{definition}

\medskip  Remark that  the equality  also holds
\vskip 0.1 cm 
\centerline{$ S[\phi]  (x, y) =  \displaystyle  \left | \frac 1 {y-x} \int_x^y \phi(t) dt \right|\, .$} 
\vskip 0.1 cm  
We now explain  how  this function $S[\phi]$  intervenes in  the expression of  the trivariate function $\Lambda_\mu$ and  how it is involved in Shannon weights. We also describe  its main properties:

     \begin{lemma}  \label{pws} {\rm [Generating operators and Shannon weights.]} 
     Consider  a  measured dynamical system $(\cal S, \mu)$ that satisfies Definition \ref{ds}, together with its associated transfer operators.      The following holds: 
     
       \begin{itemize} 
  \item [$(a)$]  The  secant operator   ${\mathbb H}_{s} \langle w\rangle $  is   a generating operator for  the  fundamental intervals $\cal{I}_w$:
   $$   \mu({\cal I}_w)^s = {\mathbb H}_{s} \langle w \rangle  [S^s[\phi]] (0, 1) \, .$$
   \item[$(b)$]  There is an alternative expression for the function $\Lambda_\mu$, 
$$ \Lambda_\mu(v, t, s) = (I- {\mathbb L}_{s, t})^{-1} [ S^s [\phi]] (0, 1)\, .$$
 \item[$(c)$]  The Shannon weights of the source $(\cal S, \mu)$ are expressed with the transfer operator ${\mathbb L}_s$ of $\cal S$ defined in \eqref{Lgene}, 
\vskip 0.1 cm 
\centerline{  $  \displaystyle \m_\mu(n) = -  \frac {d}{ds}  \bigg( {\mathbb L}_s^n [[S^s[\phi]]  \bigg) {\Big |}_{s = 1}  (0, 1)$\, ; }
\vskip 0.1 cm 
\item[$(d)$]  When   $\phi$  is strictly positive and belongs to  ${\cal C}^1([0, 1])$,     the function  $S[\phi]$  belongs to ${\cal C}^1([0, 1]^2)$ with a norm\footnote{The norm $||\cdot||_0 $ is the norm sup and the norm  $||F||_1$ on ${\cal C}^1([0,  1]^2)$ is  $ ||F||_1 = |||F||_ 0 + ||DF||_0$}
  $$ ||S^s[\phi] ||_0  \le  r_\phi^s, \qquad   ||D(S[\phi]^s)||_0  \le s \,  r_\phi^{s-1} \,  ||\phi'||_0 , $$
  that involves  the constant   $r_\phi = \max (||\phi||_0,  || 1/ \phi||_0)$.

 \end{itemize}
 
 \end{lemma}
 \begin{proof} 
Item   $(a)$ With the definition of  the function $S[\phi]$ and  operators ${\mathbb H}_{s} \langle w \rangle $, one has, for  $(x, y) = (0, 1)$ 
 $$   {\mathbb H}_{s} \langle w \rangle  [S^s[\phi]] (0, 1) = \left|  h_w(0) -h_w(1) \right|^s  [S^s[\phi]] (h_w(0), h_w(1)) $$
  $$   = \left| h_w(0) -h_w(1) \right|^s  
 \left| \frac 1 {h_w(0)-h_w(1)}\right|^s \,\left| \int_{h_w(0)}^{h_w(1)} \phi (u) du\right|^s  = 
 \left|  \int_{h_w(0)}^{h_w(1)} \phi (u) du\right|^s = \mu({\cal I}_w)^s. $$
 
 \smallskip
Item $(b)$  Clear  with  \eqref{dec-q-inv}.   

\smallskip
Item $(c)$ The  $n$-th Shannon weight  is (by definition) the entropy of the  fundamental partition $\cal P_n$. With  Eqn \eqref{Pm} and Item $(a)$, one has
 $$\sum_{w \in \{0, 1\}^n}  \mu(\cal I_w)^s  =   \sum_{w \in \{0, 1\}^n}  {\mathbb H}_{s} \langle w \rangle  [S^s[\phi]] (0, 1)\ = {\mathbb L}_s^n  [S^s[\phi]] (0, 1)\, , $$
 and  taking the (formal) derivative with respect to $s$ leads to Item $(c)$. 
 
 \smallskip
Item $(d)$  The function  $F$ 
 defined by $F(x) = \int_0^x \phi(t) dt$  belongs to ${\cal C}^2 ([0, 1])$ and satisfies $F'(x) = \phi(x)$. For $x \not = y$, one has 
 $$
D (S[\phi]^s) (x, y)  = s \, S[\phi]^{s-1}D (S[\phi]) (x, y),$$
and, with the Mean Value Theorem, 
$$\frac {\partial }{ \partial x} S[\phi] (x, y)  =  \frac 1{(y-x)^2} \left( F(y) -F(x) - (y-x) F'(x)\right) =    \tfrac 1 2 F''(t) =   \tfrac 1 2 \phi'(t)$$
for  some  $ t \in ]x, y[ $. This entails\footnote{ We use the norm $|D(G)(x, y) | =   |  \frac {\partial }{ \partial x} G(x, y)| +  | \frac {\partial }{ \partial y} G(x, y)|$}, for any pair $(x, y)$ with $x \not = y$,  
\vskip 0.1 cm 
\centerline{ 
$| D (S[\phi]) (x, y)| \le  
 ||\phi'||_0$. }
\vskip 0.1 cm
  At the point $(x, x)$ of the diagonal, one has,  with the Mean Value Theorem, with possibly $x = 0$ and any $z$, 
 $$ 
  S[\phi] (z, x) - S[\phi] (x, x)  = 
   \frac 1 {(z-x)} \left(  F(z) -F(x) -   (z-x)  F'(x) \right)= \tfrac 1 2  (z-x) F''(t)$$
   with some $t \in ]z, x[$. This entails  the continuity of $S[\phi]$ at  point $(x, x)$ together with the bound, for any $(x,  x)$,   
   \vskip 0.1 cm 
\centerline{ 
$ | D (S[\phi]) (x, x)| \le  
||\phi'||_0$. }

 \end{proof}
 
 \subsection {Alternative expression for $\Lambda_\mu$} \label{dyn-gen}

We  give a summary of the  results we have obtained   in the previous subsections:  
  
 \begin{theorem}  \label{triA} {\rm [Alternative expressions  for a   measured dynamical system   $(I, T, \mu) $  of the Class $\cal{DR}$]}
 Consider  a  dynamical  system   $(\cal I, T) $  of the Class $\cal{DR}$, and denote by   ${\mathbb L}_s, {\mathbb G}_{v, s}, {\mathbb A}_s$   the  various associated transfer operators.  The following holds 
 \begin{itemize} 
 \item[$(a)$] There is a relation between  the three  (secant) quasi-inverses, 
   \begin{equation} \label{simpeigen}
   (I-v{\mathbb L}_{s, t})^{-1}= (I-v {\mathbb A}_s)^{-1} \circ  (I- t{\mathbb G}_{v, s})^{-1}  \, , 
   \end{equation} and a similar expression for the tangent quasi-inverses.
\item[$(b)$]  The  trivariate generating function  $\Lambda_\mu$      involves  the sequence $q(n)= a^n(1)$, together with the function $S[\phi] $, 
 \begin{equation*} \Lambda_\mu(v, t, s) =  (I- v {\mathbb L}_{s, t})^{-1} [S^s[\phi]]\,  (0, 1) =    \sum_{n \ge 0}  v^n \, q(n)^s \,  (I-t {\mathbb G}_{v, s})^{-1} [S^s[\phi]] \,  (0, q(n)) \, . 
   \end{equation*}
   \end{itemize}
     \end{theorem}
     
     \begin{proof} 
     We have already proven  the  first expression of Item $(b)$. We now focus on the second one: 
 The iterate ${\mathbb A}_s^n$ satisfies, for  a function $L$ defined on $\cal {I}^2$ and $(x, y)\in {\cal I}^2$, 
 $$ {\mathbb A}_s^n[L] (x, y) = \left|  \frac {a^n(x)-a^n(y)}{x-y} \right|^s L(a^n(x),a^n(y)) \, .$$
 The branch $a^n$ satisfies  $a^{n}(0)= 0$ and $a^n(1) = q(n)$; thus, 
 ${\mathbb A}_s^n[L] (0, 1)$ together with $ (I- v{\mathbb A}_s)^{-1} [L] (0, 1)$ are  expressed with the sequence  $ q(n)$ as 
 $$ {\mathbb A}^n_s[L](0, 1)  = \ q(n)^s\, L( 0,  q(n)),  \qquad (I- v{\mathbb A}_s)^{-1} [L] (0, 1)  = \sum_{n \ge 0} v^n\,  q(n)^s\,   L(0, q(n))\, .$$
 \end{proof}

This expression of $\Lambda_\mu(s, v, t)$  will be  the starting point of our further analysis, done in the sequel of this Section and in  the next Section  \ref{dynana2}.

 %%%%%%%%%%%%%%%%%%%%%%%%%%%%%%%%%%%%%%%%%%%%%%%%%%%%%%%%%%%%%%%  

  \subsection{Properties  of the    operators   for a source of  the Good Class}  \label{firstgoodclass}

Our main final interest is the study of   
the   (weighted) operator ${\mathbb G}_{v, s}$ defined in  \eqref{formalgvs}   for a real $v \le 1$  and $s > s_0$  (where $s_0$ is the  convergence abscissa of the  Dirichlet series $\Delta(s)$).  
The operator  ${\mathbb G}_{v, s}$ is an extension  of 
the (unweighted)  secant operator ${\mathbb G}_{1 ,s}$, also denoted   as ${\mathbb G} _s$,  and defined as 
$$ {\mathbb  G}_{s} [F](x, y) =
  \sum_{m \ge 1}  \,   \left| \frac { g_m(x)- g_m(y)}{x-y}\right|^s F(g_m(x), g_m(y))\, ,  
 $$ in terms of the inverse  branches $g_m = a^{m-1} \circ b$. 
  With Definition \eqref{Lgene}, the operator  ${\mathbb G} _s$ is the secant transfer operator of the block  dynamical system $\cal B$. 
When the  system $\cal B$ belongs to the Good Class (see Definition \ref{good_class}),   the operator   ${\mathbb G}_s$   is deeply studied in the    paper  \cite{CeVa}. This paper   first relates the  secant operator  ${\mathbb  G}_{s} $
 to the  
 (classical) tangent operator 
  $$ {\mathbf  G}_{s} [f](x) =
  \sum_{m \ge 1}  \,    |g'_m(x)|^s \, f(g_m(x))
  \,;     $$
   it uses the equality \eqref{tang-sec} 
  that proves  that the tangent  operator coincides with  the secant operator on the diagonal. 
  The paper \cite{CeVa} considers the case  when the operators resp. act   (for the tangent operator) on the functional space ${\cal C}^1({\cal I})$, endowed with the norm
$$||f|| = ||f||_0 + ||f'||_0, \quad  ||g||_0 =  \sup_{x \in {\cal I} } ||g(x)||\, , $$ 
and   (for the secant operator) on  ${\cal C}^1({\cal I}\times{\cal I})$, endowed with the norm
$$||F|| =  ||F||_0 + ||DF||_0, \qquad  ||G||_0=  \sup_{(x, y) \in {\cal I}^2 } ||G(x, y)||\, , $$ and relates its spectral properties on these respective spaces.

   \smallskip  
   In   its Theorem 5.1, the paper   \cite{CeVa}   deals with the notion of quasi-compacity\footnote{ An operator is quasi-compact if the upper part of its spectrum is only composed with eigenvalues of finite multiplicity. }; it   states,   for the operators  ${\mathbf G}_s$ and  ${\mathbb G}_s$,    dominant spectral properties  and  relates  the dominant spectral properties of  ${\mathbb  G}_{s} $ to these of ${\mathbf  G}_ s$.  We  easily  extend   Theorem 5.1 of \cite {CeVa}  to the weighted version of the operators for $0 \le v\le 1$.  We also  explicitely state in    Item $(iii)$ the existence of an invariant density.

  \begin{proposition} \label{GvsProp}
  
  Consider  a dynamical system  $\cal B$ of the Good Class and  real pairs 
$(v, s) $  with $0 \le v \le 1, s >s_0$, where $s_0$ is the  convergence   abscissa of the Dirichlet series $\Delta(s)$.
The following holds   for the  weighted secant operator ${\mathbb G}_{v, s}$

\begin{itemize}  

\item[$(i)$]  The operator ${\mathbb G}_{v, s}$   acts on the space ${\cal C}^1({\cal I}^2) $  and is  quasi-compact.   

\item[$(ii)$] The  decomposition holds   

\vskip 0.1 cm 
\centerline{ 
$ {\mathbb G}_{ v, s} = \lambda( v, s) \, {\mathbb P}_{ v, s} + {\mathbb  Q}_{ v, s} \, , 
  $}
  \vskip 0.1 cm 
 \noindent  where  $\lambda(v, s)$ is the dominant eigenvalue of $\mathbb G_{v, s}$, ${\mathbb P}_{ v, s} $ is the dominant projector  and   ${\mathbb  Q}_{ v, s}$ is the remainder operator of the operator ${\mathbb  G}_{v, s}$. 
 
 \item[$(iii)$] The two dominant eigenvalues (resp. relative to  $\mathbb G_{v, s}$ and $\mathbf G_{v, s}$) are equal.
 At $(v, s) = (1, 1)$ the dominant eigenvalue $\lambda(v, s)$ equals  $1$;  the  two dominant eigenfunctions, the density $\psi = \psi_{1, 1}$  on $\cal I$ and the   density $\Psi = \Psi_{1, 1}$ on $\cal I ^2 $ are  invariant  by
 $\mathbf G_{1}$  (resp. $\mathbb G_{1}$) and the measure $\nu$ of density $\psi$  is invariant by the block map $\hat T$. Moreover, the equalities  hold
 \begin{equation} \label{part}  \Psi(0, 1) = 1, \qquad {\mathbb P}_{1, 1} [F] (x, y) = I[F]  \, \Psi(x, y) \, .
 \quad \hbox{with}  \quad  I[F] =  \int_{\cal I}\, F(t, t) dt  \, .
 \end{equation} 
 \end{itemize}
 \end{proposition} 
 
  \begin{proof}

      Items $(i)$, $(ii)$ and and $(iii)$ 
are easy extensions of their analogs in  Theorem 5.1  of the  paper \cite{CeVa}:   Using   Lasota Yorke inequalities that hold   for   operator ${\mathbb G}_{v, s}$  when  $v \le 1$ and $s > \rho$  entails Hennion's Property and quasi-compacity.

\end{proof}

\subsection {Entropy of the block source.} 

 We give in  the next proposition  a detailed proof of the properties of  the entropy of the source $\cal B$ that were not made completely precise  in \cite{CeVa}. Together, the two propositions \ref{GvsProp} and  \ref{ent-first}  provide a proof for Assertion $(a)$ of Theorem \ref{inv}.

 \begin{proposition}{\rm  [Entropy of the block source]} \label{ent-first}
  The source $\cal B$ admits an entropy  with respect to  the invariant measure $\nu$ that satisfies
  $$ 
  \cal E_\nu(\cal B) =- \lambda'(1) =  \int_0^1 \widehat  {\mathbf G}_1[\psi (t) dt  < \infty, \quad  \widehat {\mathbf G}_1 =  \frac {d} {ds} \,  {\mathbf G}_s {\big |}_{s = 1}\, . 
   $$

  \end{proposition}

\begin{proof}  
    We begin with  the expression of  Shannon weights (of the source $\cal B$) given in Lemma \ref{pws} $(b)$. It  involves the transfer operator ${\mathbb G}_s$ of the system  $\cal B$, and more precisely   the   derivatives  $s \mapsto (d/ds) ({\mathbb G}_s^n) $ of its iterates  at $s = 1$.  As $\cal B$ is of Good Class, its transfer  operator ${\mathbb G}_s$  has,  on a real neighborhood of $s =1$,   a dominant eigenvalue $\lambda(s)$ with $\lambda(1) = 1$ together with a spectral gap.  
 Then, with  Eqn \eqref{part} and   Lemma \ref{pws}$(a)$,    the decomposition holds on this neighborhood, 
 $$    {\mathbb G}_s^n[S^s[\phi]](0, 1)  = \lambda(s)^n \, {\mathbb P}_{ 1, 1}  [S[\phi]] (0, 1)  \, \left(1 + O(\rho^n) \right)=  \lambda(s)^n  \left(1 + O(\rho^n) \right) \, , $$
 where the $O$ is uniform on this neighborhood.   Then, using  Lemma \ref{pws}$(c)$ and  taking the derivative at $s = 1$  leads to  the estimate of  Shannon weights of the source $\cal B$, 
 $$   \underline m_\mu(n)  =  - n \lambda'(1)\,  (1 + O(\rho^n)) \, . $$
 Then,  for any  measure $\mu$  with a strictly positive density of Class ${\cal C}^1$,  the entropy rate  $\cal E_{\underline \mu}$ of the  block source exists,  does not depend on $\mu$  and coincides  with  $-\lambda'(1)$.

 \medskip We  now apply Proposition   \ref{ent-mu} in  this context.  The   system  $\cal B$ is generating, and with  Proposition \ref{GvsProp} Item$(iii)$  
   it admits  an invariant measure (namely, the measure $\nu$  of density $\psi$). Then, with  Proposition  \ref{ent-mu},   the entropy rate $\cal E_{\underline \nu}  = -\lambda'(1) $  coincides with   the metrical entropy  $\cal E_\nu (\cal B)$ of the block dynamical system $\cal B$.

 \medskip  
    We now prove the equality 
 \begin{equation} \label {ent1} 
  \lambda'(1) = \int_{\cal I} \widehat {\mathbf G}_1[ \psi](t) dt \qquad \hbox{$\psi$ the invariant density of $\cal B$} \, , 
  \end{equation}
 that is known as the Rohlin Formula. 
We begin with the perturbative equation 
\vskip 0.1 cm 
\centerline{ ${\mathbf G}_s [\psi_s] = \lambda(s) \,  \psi_s $} 
relative to the eigenvalue $\lambda(s)$, together with its  eigenvector $\psi_s$  with the normalisation $\int_{\cal I} \psi_{s} (t) dt = 1$.  We have $\psi_1 = \psi$ and $\lambda(1) = 1$. 
We   consider  its derivative with respect to $s$  at $s = 1$ and take the integral on $\cal I$.    
     We  first obtain,  two terms, resp. equal to 
   $$ \int_0^1 \left(\frac d{ds} {\mathbf G}_{s}|_{s= 1}\right)   [ \psi](t) dt \, \quad \hbox{and}  \qquad \left( \frac d{ds} \lambda(s) |_{s = 1} \right)   \int_0^1\psi(t) dt =     \lambda'(1)
    \, .    $$
    There are of course two other terms in the derivative at $s = 1$, one for  the left member and one for the   right member, 
       resp. equal to 
   \vskip 0.1 cm 
    \centerline
    {$\int_0^1  {\mathbf G}_1\left[ \underline \psi\right](t) dt \qquad \hbox{and} \qquad  \int_0 ^1   \underline  \psi(t)\,  dt\, $,   \qquad with $  \underline  \psi := (d/ds)\,   \psi_{s}|_{s = 1} $,  }
    \vskip 0.1 cm 
       and indeed equal,  as the operator $ {\mathbf G}_1$ is a density transformer.

 \smallskip 
  There is another expression (more classical)  for the  Rohlin formula, when applied  to a system $\cal B$  of the Good Class  associated with a shift $\widehat T$, 
  \vskip 0.2 cm 
  \centerline
  {$ \cal E_{\nu} (\cal B) = \int_{\cal I} \log |\widehat T'(x)| \, \psi (x) dx$ , }
  \vskip 0.1 cm
  where $\psi$ is the invariant density of the operator ${\mathbf G}_{1}$.  The  previous expression  coincides   with  \eqref{ent1}  due to  the change of variables $x= g_m(u)$ on each interval ${\cal J}_m$. 

   \end{proof}

    \subsection {Invariant measure for a source in the Class  $\cal{DRBGC}$. } \label{sec:inv}
    We now wish to prove Assertion $(b)$ of Theorem \ref{inv}. 
We mainly use   (tangent) transfer operators (all taken at $s = 1$) defined in Section  \ref{trans}, 
 we omit the index $s$,  and we notably  deal with the operators $\mathbf L \, ,\mathbf A \, , \mathbf B \, , \mathbf G\, .$
 
We  recall the decomposition obtained  in Section \ref{trans}
\begin{equation} \label{againL}
(I-  {\mathbf L} )^{-1} =  (I-  {\mathbf A} )^{-1}  \circ (I-  {\mathbf G} )^{-1}, \qquad {\mathbf G} = {\mathbf B} \circ  (I-  {\mathbf A} )^{-1}\, .
\end{equation}

  \begin{proposition} \label{blockJet pasJ}  
     Consider  a  source $\cal S$ of the $\cal{DRBGC}$ Class,    whose block source $\cal B$ admits the invariant density $\psi$  and the invariant probability $\nu$.    Then, 
     
     \begin{itemize} 
     \item[$(i)$] the function  
    $ \phi_0 = (I-{\bf A})^{-1} [\psi]$
    is invariant  by the operator ${\mathbf L}$.   Its associated measure  has a total mass equal to $\E_\nu[W]$. When  the mean $\E_\nu[W]$ is finite,   this leads to an invariant  probability measure $\pi$ with density  $\rho = (1/\E_\nu[W]) \phi_0$, which satisfies $\lim_{x \to 0} \rho(x) = + \infty$.
    
    \smallskip
\item[$(ii)$]   The   following equalities  relate  the  probability measure $ \nu$ of density $\psi$ and  the  probability measure $\pi$ associated with $\rho$,  for any $k\ge 0$, %with $\cal J = b(\cal I)$, 
\vskip 0.1cm  
 \centerline{$ \nu [T^{-k}{\cal J} ]=\pi  [ \cal  J \cap T^{-k-1} ({\cal J} )], \quad  \cal J = b(\cal I)$ . }
  \end{itemize}
  \end{proposition}

 There is a proof for Item $(i)$  in \cite{Sch}, pages 120-121 (Theorem 17.1.3).  Dealing  here with  the fundamental decomposition   described in \eqref{againL} and densities, we  now describe  a direct and simple  proof  of these results,  which  moreover uses  easy relations between various sets collected in the following  lemma (whose  proof is omitted): 

   \begin{lemma} \label{easy}   Consider a  source $\cal S$ of the $\cal{DR}$ Class,  with  $\cal J =  b(\cal I)$ and  the waiting time  $W$ of symbol one.   For $X \subset {\cal I}$,   there are easy relations :  
\begin{itemize} 
\item[$(a)$]
$T^{-k}\cal J\cap[W>k]=[W=k+1] \quad  k\geq 0$
\item[$(b)$]
  $a^n(X)=T^{-n}X\cap[W>n]$
  
  \item[$(c)$] $b(X) = {\cal J}  \cap  T^{-1} (X)$  
  \end{itemize} 
  \end{lemma}

 \begin{proof}  [Proof of Proposition  \ref{blockJet pasJ} ] \ \

\smallskip
Item $(i)$. With Relation \eqref{againL}, 
the function $\rho=  (I-{\bf A})^{-1} [\psi]$  satisfies 
\vskip 0.1 cm 
\centerline{ 
$(I- \mathbf L) [\rho] =(I-\mathbf G)(I- \mathbf A)[\rho]=(I- \mathbf G)[\psi]=0$}
\vskip 0.1 cm 
and  it  is  thus invariant by $ \mathbf L$.  For any $y \in [0, 1]$, one has  \vskip 0.2 cm 
\centerline{ $
 \rho(y)=\sum_{n=0}^{\infty}{\bf A}^n[\psi](y) = \sum_{n=0}^{\infty}(a^n)'(y) \psi(a^n(y))\, .$}
 \vskip 0.1 cm
  As the system $(I, T)$ belongs to the class $ \cal{DRI}$ defined in  Definition \ref{defDR}, the  branch $a$ has an indifferent point at $0$  i.e., $a(0) = 0, a'(0) = 1$, and, as already mentioned in Section \ref{waitbis},  the derivatives $(a^n)'(0)$ are all equal to 1, and $\rho (0) = +\infty$. This implies that $\rho$ cannot be of class $\cal C^1$ on $\cal I$. 
  
 \smallskip 
  Moreover,  for  any   $X\subset \cal I$,   the  equalities hold  \vskip 0.1 cm 
\centerline{$
 \int_X\rho(y) dy=\sum_{n=0}^{\infty}\int_X{\bf A}^n[\psi](y) \, dy = 
\sum_{n=0}^{\infty}\int_{a^n(X)}\psi(t)\, dt = 
\sum_{n=0}^{\infty}\nu(a^n(X)) \, .
 $}
 \vskip 0.1 cm
 Furthermore,   Lemma \ref{easy}$(b)$ with $X= \cal I$   entails the equalities
   \vskip 0.2 cm 
   \centerline{
$ \int_{\cal I} \rho(y) dy=\sum_{n=0}^{\infty}\nu[W>n]=\E_{\nu} [W] $, } 
\vskip 0.1 cm
  and, with   Lemma \ref{easy}$(a)$, 
the relation $ \pi (\cal J)=  1- \nu[W = 0] = 1$ holds. 

\smallskip 
Item $(ii)$ With   Lemma \ref{easy}$(c)$,  for any $k \ge 0$ and  $Y = T^{-k} \cal J$,   the relation holds
\vskip 0.1 cm
\centerline{
$b(Y) =  {\cal J}  \cap  T^{-(k+1} (Y)$. }Moreover the relation ${\mathbf B}[\rho] = \psi$ entails    the equality 
$ \nu (Y) =  \pi( b(Y)) $, and thus $(iii)$. 

\end{proof} 

\medskip This  Section has  provided a proof for Theorem \ref{inv}. 
We now  deal,   in the next Section, with the proof of our second main result,  Theorem   \ref{proexpgenfinbis}.

   \section {Dynamical analysis:  Proof of  the second main result. }
     \label{dynana2}
  
  This Section provides a proof for  Theorem  \ref{proexpgenfinbis}.    Theorem \ref{triA} leads to study  in Section \ref{pertur}   the behaviour of the quasi-inverse $ v \mapsto (I- {\mathbb G}_{v, 1})^{-1} $  as $v \to 1^-$. Via Perturbation theory,  it  is related   later  in Proposition \ref{true} of Section \ref{per}  to  the behaviour of   $v \mapsto 1- \lambda(v, 1)$  (as $v\to 1^-$). With these two estimates, we   study, as $v \to 1^-$,  the generating functions    $N_\mu$   and $M_\mu$. This leads  to   the  final part of the proof of  Theorem   \ref{proexpgenfinbis}.

   \subsection {Perturbation of the transfer operator as $v \to 1^-$.} \label{pertur}

 \begin{proposition} 
 {\rm [Perturbation as $v \to 1^-$]}  \label{perturba}
The  notations are the same as in Proposition \ref{GvsProp}.   For   a   function  $F$  of class ${\cal C}^1({\cal I}^2)$, we denote 
\vskip 0.1 cm 
\centerline{ $I [F]   = 
\int_{\cal I} F(u, u)  du $.  }   
When $I[F]$ is non zero, the  following estimate   holds  for the quasi-inverse,   
$$ (I-{\mathbb G}_{v, 1})^{-1} [F] (x, y)   = 
\frac {I[F] } {1-\lambda( v, 1)}\,  \Psi (x, y)\,  \big( 1 + \delta (v) ||F||\big) , \quad \delta (v) \to 0 \quad (v\to 1^-)  \, .$$
  \end{proposition}

\begin{proof} 
    For  a real $v$ with    $0 \le v < 1$,  
    using the  relations 
 \[
 {\mathbb P}_{v, 1}^2 = {\mathbb P}_{v, 1}, \quad  { \mathbb P}_{ v, 1} \circ {\mathbb Q}_{ v, 1}=   {\mathbb Q}_{ v, 1} \circ {\mathbb P}_{ v, 1}  = 0\, , 
 \]
 the spectral decomposition  of ${\mathbb G}_{v, 1}$ described in Proposition~\ref{GvsProp} \textit{(ii)} leads to the decomposition of the quasi-inverse,  for $F \in {\cal C}^1({\cal I}^2)$, 
 \begin{equation} \label{specdecqinv}
      (I-{\mathbb  G}_{v, 1})^{-1} [F] =   \frac {\lambda( v, 1)} {1-\lambda( v, 1)} {\mathbb  P}_{v, 1}[F] + {\mathbb R}_{v, 1}[F], \qquad {\mathbb R}_{ v, 1} =  (I-{\mathbb Q}_{ v, 1})^{-1} \, . 
        \end{equation}
  We now let $v \to 1^-$ and  use a version  of Abel's Continuity Theorem on Banach spaces. 
 
 \smallskip
 {\sl Abel's Continuity Theorem on Banach spaces.} 
  On  a Banach space $\cal A$, consider  a series $G(v) = \sum_{m \ge 1} v^m G_m$ which involves  a  sequence $G_m \in {\cal A}$
and assume that  the series  $G(v)$ is convergent in $\cal A$ for $v$ real in $[0, 1]$. Then 
$\lim_{v \to 1^-}  G(v) = G(1)\, .$

\smallskip
We apply the previous result to the operator $ G(v) = \mathbb G_{v, 1}$ on the Banach space of the operators which act on ${\cal C}^1({\cal I}^2)$. This proves that the operator 
$ \mathbb G_{v, 1}$ tends to $\mathbb G_{1, 1}$ as $v \to 1^-$. 

\medskip
   Via  the theory of continuous perturbation,   [see for instance \cite{Ka} Theorem 3.16,  page 212],  
  the   three spectral elements $\lambda( v, 1)$, ${\mathbb P}_{ v, 1} $ and ${\mathbb Q}_{ v, 1}$   of  ${\mathbb G}_{ v, 1} $ satisfy the following, as $v\to 1^-$ 
   \vskip 0.1 cm 
  \centerline {$ \lambda( v, 1)$ tends to $\lambda(1, 1) = 1$, \quad  ${\mathbb P}_{v, 1} $ tends to ${\mathbb P}_{1, 1} $ and  ${\mathbb Q}_{v, 1} $ tends to ${\mathbb Q}_{1, 1}$,   }
  \vskip 0.1 cm 
{and thus}
\centerline{ $   \delta(v)  =  \sup \Big[ (\lambda( v, 1) - 1), \   ||{\mathbb P}_{v, 1} -{\mathbb P}_{1, 1}||,\   ||{\mathbb Q}_{v, 1} -{\mathbb Q}_{1, 1}||\Big] \to 0 $\, .}
  \vskip 0.1 cm 
  Due to the spectral gap, the operator $(1- {\mathbb Q}_{1, 1})^{-1}$  has a norm  at most $ 1 /(1-\rho)$. with $\rho<1$.   If we choose $v$ close enough to 1, so that   the operator $ {\mathbb Q}_{v, 1} - {\mathbb Q}_{1, 1}$ has a norm $\delta(v) $  with  $ \delta(v)   <(1- \rho)$ then,  the decomposition      \vskip 0.1 cm 
   \centerline{$ {\mathbb R}_{v, 1} - {\mathbb R}_{1, 1} = (I- {\mathbb Q}_{1, 1} )^{-1} \left[ \left( I + \left( {\mathbb Q}_{v, 1} - {\mathbb Q}_{1, 1}\right) ( I-{\mathbb Q}_{1, 1})^{-1} \right) ^{-1} -I\right]$} entails the bound  
 $  || {\mathbb R}_{v, 1} - {\mathbb R}_{1, 1}|| = O( \delta(v))\, .$ \\
 
   Finally,  
     using  Eqn \eqref{part},  the equality  ${\mathbb P}_{ 1, 1}[F] (x, y) = I[F] \Psi(x, y)$ holds,  and  the decomposition given in \eqref{specdecqinv}   entails the estimate 
   \begin{equation} \label{qinvprec} 
       (I-{\mathbb  G}_{v, 1})^{-1} [F](x, y)  =  \frac {\lambda( v, 1)} {1-\lambda( v, 1)} \,  I[F] \, \Psi(x,  y) + \frac {\delta(v)}{1- \lambda(v , 1)}  \epsilon (v, F) \end{equation} 
  $$\hbox{with} \qquad \epsilon (v, F) = \, ||F||\left( 1 +   ||{\mathbb R}_{1, 1}||  + O(\delta (v) )\,\right)  = O(||F||  ) \, .$$
    This ends the proof.
   \end{proof}

     \subsection 
 {Behaviour of $1- \lambda(v, 1)$ as $v \to 1^-$}  \label{per} The previous result leads to study the dominant  behaviour of  $1- \lambda(v, 1)$ as $v \to 1^-$. We relate this dominant behaviour to the dominant behaviour   of  $Q_\nu(v)$  as $v \to 1^-$.

  \begin{proposition}  \label{true}    The ratio
  \begin{equation} \label {DD2}     D(v) =  \frac {1- \lambda( v, 1)}{1-v}  \quad \hbox{satisfies}\quad  D(v) \sim_{v \to 1^-} Q_\nu(v) \quad (v \to 1^-), 
 \end{equation} 
 and involves the generating  function $Q_\nu(v)$ of the wtd sequence (wrt to the invariant measure $\nu$ of the block source).  When $Q_\nu(1) <\infty$,  the  map $v \mapsto \lambda(v, 1)$ is derivable on the left   at $v = 1$,  with a derivative equal to $Q_\nu(1)$.

    \end{proposition}

    \begin{proof}
   We  begin  with the  expression of Theorem~\ref {triA},  at $(t, s) = (1, 1)$,  with a measure $\mu$  of density $\phi$    $$ \Lambda_\mu(v, 1, 1) = \sum_{n \ge 0}     {v^n} \,  q(n) \,  (I- {\mathbb G}_{v, 1})^{-1} [S[\phi]] (0, q(n))\, . 
   $$
   Proposition \ref{perturba} 
   describes the dominant  part of the quasi-inverse   $(I- {\mathbb G}_{v, 1})^{-1}$ as $v \to 1^-$ and leads to   the  following estimate, 
   \begin{equation}\label{DD} {\rm Dom} _{\ (1^-)} \Lambda_\mu(v, 1, 1)  =  \frac {\tilde D(v)} {1-\lambda(v, 1)},   \qquad   \widetilde D (v) =    \sum_{n \ge 0}      {v^n} q(n)  \,  \Psi(0, q(n))\, .  \end{equation} 
   It involves  the  invariant density $\Psi$   of ${\mathbb G}_{1}$ that  extends  the invariant density $\psi$ of ${\mathbf G}_{1}$. 
   On the other hand, the  equality described in Eqn \eqref{tjrs} 
   introduces  $D(v)$ defined in \eqref{DD2} and proves with \eqref{DD}  the estimate 
     \begin{equation} \label{DD1}   D(v) \sim_{v \to 1^-} \widetilde D(v) \, .
   \end{equation}
      As  the function $\Psi$  belongs to ${\cal C}^1({\cal I}^2)$ and the sequence $q(n)$ tends to zero, the sequence $\Psi (0, q(n)) $ satisfies  
   \vskip 0.1 cm 
   \centerline{ $ \Psi (0, q(n))  = \Psi(0, 0) + O(q(n))$,} \vskip 0.1 cm and, as $\Psi(0, 0) =\psi(0)$ is strictly positive,  the estimate  holds  
  \begin{equation} \label{equ1}     \tilde D(v) =  \psi(0)\,   Q_\tau(v)  +O \left( R(v)\right) \qquad R(v) = \sum_{n \ge 0}  q(n)^2 v^n\, , 
  \end{equation}
   and involves  the generating function  $Q_\tau(v)$ of the  wtd sequence  with respect to the uniform measure $\tau$. 
   
   \medskip 
  There are now  two cases  for the estimate of $\tilde D(v)$ as $v \to 1^-$ obtained in Eqn \eqref{equ1}, and we prove in each case the estimate $\tilde D(v) \sim_{v \to 1^-} Q_\nu(v)
   $.  
   
 \medskip  
    {\sl  First Case:  $Q_\tau (v) \to \infty$ as $v \to 1^-$. } As $q(n) \to 0$,  the sequence $q(n)^2
   $ is $o(q(n))$. Then Lemma  \ref{aux} Item $(i)$ applies and  the series  $R(v)$ in \eqref{equ1}  is $o( Q_ \tau (v)) $ as $v\to 1^-$. This entails, with \eqref{equ1}, 
   \begin{equation} \label{DD3} 
    \tilde D(v) \sim_{v \to 1^-}  \psi(0)\,   Q_\tau(v) \, .  
    \end{equation}
    Furthermore,   Lemma  \ref{aux} Item $(ii) (b)$ entails  the estimate  $Q_\nu(v) \sim_{v \to 1^-}  \psi(0)\,   Q_\tau(v) $,  and  
  the  estimate  $\tilde D(v) \sim_{v \to 1^-} Q_\nu(v)
   $ holds. 
          
    \medskip
      {\sl Second Case:   $Q_\tau (1) < \infty$.}  Then  \eqref{equ1} entails that  $\tilde D(1) < \infty $ too, and, with \eqref{DD1}, the equality  $D(1) = \tilde D(1)$ holds.  We now prove the estimate  
      \vskip 0.1 cm 
      \centerline{ 
      $\tilde D(1) = Q_\nu(1) = \E_\nu [W]$.}
      \vskip 0.1 cm 
      In this case, the  map $v \mapsto \lambda(v, 1)$ is   derivable on the left   at $v = 1$.   
     We denote by   $ \psi_{v, 1}$ the eigenfunction  of ${\mathbf G}_{v, 1}$ relative to the eigenvalue $\lambda(v, 1)$,  with the normalisation $\int_{\cal I} \psi_{v, 1} (t) dt = 1$; 
     we consider the perturbative equation
    \begin{equation} \label{perturb}   {\mathbf G}_{v , 1} [ \psi_{v, 1}]=  \lambda(v, 1)    \psi_{v, 1} \, ,   
    \end{equation}
     and   its derivative with respect to $v$ for $v<1$;    we take the integral on $\cal I$,  and finally we let $v=  1$. We  first obtain, at $v = 1$,   two terms, resp. equal to 
   \begin{equation} \label{two}
    \int_0^1 \left(\frac d{dv} {\mathbf G}_{v, 1}|_{v = 1}\right)   [ \psi](t) dt \, \quad \hbox{and}  \qquad \left( \frac d{dv} \lambda(v, 1) |_{v = 1} \right)   \int_0^1\psi(t) dt =      D(1)
    \, .    \end{equation}      One has  $$\int_0^1 {\mathbf G}_{v ,1} [ \psi](t) dt = \sum_{m\ge 1} v^m \int_0^1 |g'_m(t)| \psi\circ g_m(t) dt = \sum_{m \ge 1} v^m \int_{{\cal J}_m} \psi(u) du  \, ,     $$
    with ${\cal J}_m = [W = m]$. Then  the first term in \eqref{two} is 
    $$  
     \int_0^1  \left(\frac d{dv} {\mathbf G}_{v ,1}|_{v = 1}\right) [ \psi](t) = \sum _{m \ge 1} m\, \nu[W= m] = \E_\nu [W] .$$ 
    There are of course two other terms in the derivative at $v = 1$ of Eqn \eqref{perturb}, one for  the left member and one for the   right member, 
       resp. equal to 
   \vskip 0.1 cm 
    \centerline
    {$\int_0^1  {\mathbf G}_1\left[ \underline \psi\right](t) dt \qquad  \int_0 ^1   \underline  \psi(t)\,  dt\, $,   \qquad with $  \underline  \psi := (d/dv)\,   \psi_{v, 1}|_{v = 1} $,  }
    \vskip 0.1 cm 
       and indeed equal,  as the operator $ {\mathbf G}_1$ is a density transformer.
             \end{proof}

%%%%%%%%%%%%%%%%%%%%%%%%%%%%%%%%%%%%%%%%%%%%%%%%%%%%%%%%%%%%%%%  

In the  previous  two Sections,   along the proof of Propositions \ref{perturba} and  \ref{true}, we have shown the two estimates 
as $v \to 1^-$
\begin{equation} \label{est-tot}
   (I-{\mathbb G}_{v, 1})^{-1} [F] (x, y)   \sim I[F] 
 \frac {\Psi (x, y)} {(1-v) Q_\nu (v)} , \quad \sum_n q(n) v^n \Psi (0, q(n)) \sim Q_\nu(v) \, .
\end{equation}
They involve the gf $Q_\nu(v)$  of the wtd sequence wrt to the invariant measure $\nu$ of the block source, and the invariant density $\Psi$ of the operator $\mathbb G_1$ with $\Psi (0, 1)  = 1$.

\subsection{Study of the generating functions $N_\mu$ and $M_\mu$ as $v \to 1^-$.}  \label{ones}
Starting with the expression  of $\Lambda_\mu$ obtained in  Theorem~\ref {triA},
   $$ \Lambda_\mu(v, t, s) = \sum_{n \ge 0}     q(n)^s \, {v^n}\,   (I- t{\mathbb G}_{v, s})^{-1} [S[\phi]^s] (0, q(n))\, ,  $$ 
    the expressions of   the two  generating functions $N_\mu$ and $M_\mu$ are obtained 
 with  derivatives, as already mentioned in   Eqn \eqref{MN}: for $N_\mu$, the derivative of   $ \Lambda_\mu(v, t, 1)$  with respect to $t$ (at $t= 1$);   for $M_\mu$,  the derivative  of  $ \Lambda_\mu(v, t, s)$ (with respect to $s$, at $(t, s) = (1, 1$). This leads to the expressions      \begin{equation} \label{Nmu1}N_\mu (v) =  \sum_{n \ge 0}     q(n) \,  {v^n} \,  (I-{\mathbb G}_{v, 1})^{-2}\circ {\mathbb G}_{v, 1}   [S[\phi] ] (0, q(n)) \, , 
 \end{equation}
   \begin{equation}\label{Mmu1}
   M_\mu(v) = M_\mu^{[0]} (v) +  M_\mu^{[1]} (v) + M_\mu^{[2]} (v), \qquad \hbox{with} \qquad \qquad  
   \end{equation}
    \renewcommand{\arraystretch}{1.5}
  $$   \left\{ \begin{array} {ll}
     M_\mu^{[0]} (v) &=  {\displaystyle  \sum_{n \ge 0}   q(n)\,  {v^n} \,   (I-{\mathbb G}_{v, 1})^{-1} [S[\phi] \log  S[\phi]] (0, q(n))}\, , \cr
     M_\mu^{[1]} (v) &=  {\displaystyle  \sum_{n \ge 0}   q(n)\, {\log q(n)} \,  {v^n} \,   (I-{\mathbb G}_{v, 1})^{-1} [S[\phi] ] (0, q(n))}\, , \cr
  M_\mu^{[2]}(v) &=     {\displaystyle \sum_{n \ge 0}  q(n) \,  {v^n}  \, (I-{\mathbb G}_{v, 1})^{-1}\circ \widehat {\mathbb G}_{v, 1}\circ  (I-{\mathbb G}_{v, 1})^{-1}  [S[\phi]] (0, q(n))}\, . \cr
    \end{array}\right\}
    $$
 Here, $\widehat {\mathbb G}_{v, 1}$ stands for the  operator $ {(\partial}/{\partial s}){\mathbb G}_{v, s}|_{s = 1}$,  $\phi$ is the initial density, and the operator $S$ is defined in  \eqref{S}.

\medskip  We first explain the general  principles of the analysis:   Any of the four expressions above, in \eqref{Nmu1} or  in \eqref{Mmu1},   presents  the same structure, with two possible cases $(I)$ or $(II)$,   $$  {\displaystyle \sum_{n \ge 0}  q(n) \,  {v^n}  \, \mathbb T_v [F] (0, q(n))}  \ \  (I)
  \quad \hbox{or} \quad  {\displaystyle \sum_{n \ge 0}  q(n) \, \log q(n) {v^n}  \, \mathbb T_v [F] (0, q(n))}\ \  (II)\, , $$
 and involves an operator $\mathbb T_v$  and a function $F$. 
  We first perform an ``inner'' step and replace the operator $\mathbb T_v$ by its dominant  spectral part as $v\to 1^-$:  using  here the first estimate of Eqn \eqref{est-tot}, we  obtain an estimate  of the form 
 $$  {\mathbb T_v} [F] (0, q(n))   \sim 
 \frac  {I[F]}  {\tau (v)}  {\Psi (0, q(n))} , $$
 which associates  with the operator $\mathbb T_v$   a  function $\tau(v)$ which   will be described later, in each case of interest. 
 Collecting all these spectral estimates together, we obtain, for each expression   in \eqref{Nmu1} or  in \eqref{Mmu1},  a global  estimate of the form 
 $$ \frac {I[F]}{\tau(v)}\!\! \left( \displaystyle \sum_{n \ge 0}  q(n) \,  {v^n}  \,   \Psi (0, q(n)) \right) \ (I)
  \quad \hbox{or} \quad   \frac {I[F]}{\tau(v)}\!\!\left(  \displaystyle \sum_{n \ge 0}  q(n) \, \log q(n)\,  {v^n} \,   \Psi (0, q(n))\right) \ (II).$$   In each case,  this global  estimate  satisfies 
\begin{equation} \label{est-part} \frac {I[F]} {\tau(v)} {Q_\nu(v) }  \ \  (I), \ \   \hbox{or} \ \  O\left( \frac {H(v)} {\tau (v)}\right)\hbox{\ \ with\ \ }  H(v) = \sum_n q(n) \log q(n) v^n \  (II).
\end{equation} 
   We now make the computations precise,  and derive  a precise expression  of the term $\tau(v)$ associated with the operator $\mathbb T_v$.
  We obtain 

\begin{proposition} \label {proexpDyn}
Consider a dynamical system   of the Class $\cal{DRBGC}$. 
The following estimates  hold  for the generating functions $ N_\mu (v)$ and $M_\mu(v)$ and   involve the  generating function $Q_\nu(v)$ of the wtd sequence (wrt to the invariant measure $\nu$ of the block source),  \begin{equation} \label{estNmu}  (1-v)  N_\mu(v)\sim_{v \to 1^-}  \frac 1 {(1-v) Q_\nu(v)} \, ; 
\end{equation}
\vskip -0.3cm
 \begin{equation}  \label{Mmu}
 (1-v) M_\mu(v) \sim_{v \to 1^-}  I[S[\phi] \log S[\phi]]+  O\left( \frac {H(v)}{Q_\nu(v)}\right) +   \frac{ {\cal E}({\cal B})}  {(1-v) Q_\nu(v)} \, .  
 \end{equation} 
  
\end{proposition}

\begin{proof} We study each of the four terms of interest. 

\smallskip
{\sl Term $N_\mu$. }  The operator $\mathbb T_v $ equals $(I-{\mathbb G}_{v, 1})^{-2}\circ {\mathbb G}_{v, 1} $ and it is applied to $F = S[\phi]$. Then,  with Eqn \eqref{est-tot}, 
 $${\rm Dom} _{(1^-)} (I-{\mathbb G}_{v, 1})^{-2}\circ \, {\mathbb  G}_{v, 1}[ S[\phi]] 
  =  \frac {\Psi} {(1-v)^2\, Q_\nu(v) ^2} \, , $$
 and $\tau(v) = (1-v)^2 Q_\nu(v)^2$. With \eqref{est-part},  this leads to the estimate   of $(1-v) N_\mu(v)$ given in \eqref{estNmu}.

 %%%%%%%%%%%%%%%%%%%%%%%%%%%%%%%%%%%%%%%%%%%%%%%%%%%%%%%%%%%%%%%  

\smallskip  {\sl Terms  $M_\mu^{[0]} (v)$ and $M_\mu^{[1]} (v)$. } For the two terms $M_\mu^{[0]} (v)$ and $M_\mu^{[1]} (v)$,  the operator $\mathbb T_v$ equals $(I-\mathbb G_{v, 1})^{-1}$  and the function $\tau(v)$ equals $(1-v) Q_\nu(v)$. Then, 
with Eqn \eqref{est-tot} and Eqn \eqref{est-part},  
 $$ 
{\rm Dom}_{(1^-)} (1-v)  M_\mu^{[0]} (v)=  I[S[\phi] \log S[\phi] ]  \, .$$
For the term  $M_\mu^{[1]} (v)$,  we are in case $(II)$, and with 
  Eqn \eqref{est-part},   we obtain 
\begin{equation} \label{H} {\rm Dom}_{ (1^-)}  (1-v) M_\mu^{[1]} (v) = 
O\left(  \frac {H(v)} { Q_\nu(v)}\right)  ,  \qquad  H(v) = \sum_{n \ge 0}  v^n  q(n) \log q(n) \, .
 \end{equation}

    \medskip 
\paragraph {\sl Term  $M_\mu^{[2]} (v)$.}
  The term $M_\mu^{[2]} (v)$ deals with  the operator $\mathbb T_v$ which contains a double quasi-inverse,  
  \begin{equation} \label{double}(I-{\mathbb G}_{v, 1})^{-1}\circ \widehat {\mathbb G}_{v, 1}\circ  (I-{\mathbb G}_{v, 1})^{-1}  \, 
  \end{equation}
  and  we use two times the spectral decomposition  $ (I-{\mathbb G}_{ v, 1})^{-1} $  given in  Proposition
\ref{perturba}.  With the decomposition of the operator $\widehat {\mathbb G}_{v, 1}$ ``in the middle'',  as  
\begin{equation} \label{decwidehat} 
   \widehat{\mathbb G}_{v, 1} = a(v) \, {\mathbb P}_{ v, 1} + b(v)\, {\mathbb Q}_{ v, 1} \, , \quad   \hbox{and thus} \quad  \widehat{\mathbb G}_{v, 1}  \circ   {\mathbb P}_{ v, 1} = a(v) \,   {\mathbb P}_{ v, 1} \, , 
 \end{equation}  
the function $a(v)$ has 
 thus a limit $a(1)$ as $v \to 1^-$ that satisfies, with Proposition \ref{GvsProp} $(iv)$,   
  $$a(1) = a(1) {\mathbb P}_{1, 1} [\Psi]  = {\mathbb P}_{1, 1} \circ  \widehat {\mathbb G}_{1, 1} [\Psi] = \int_0^1 \widehat {\mathbf  G}_{1}[\psi](t) dt  = {\cal E}_\nu({\cal B})\, .$$   Then, 
  the dominant term associated with  the operator in \eqref{double} satisfies    with \eqref{est-tot} $$   {\rm Dom}_{\, (1^-)} \left[  (I-{\mathbb G}_{v, 1})^{-1}  \circ   \widehat{\mathbb G}_{v, 1}   \circ  (I-{\mathbb G}_{v, 1})^{-1} \right]  [S[\phi]]   =  \frac{  {\cal E}_\nu({\cal B})} {(1-v)^2 \, Q_\nu(v)^2}\, 
 \Psi $$ 
$$  \hbox{and} \quad  {\rm Dom}_{\, (1^-)} \left[ (1-v) M_\mu ^{[2]}(v)\right]  =    \frac {{\cal E}_\nu({\cal B})} {(1-v) Q_\nu(v)} \, . $$
The third estimate  of Proposition \ref {proexpDyn} holds.

\end{proof}

 \subsection {End of the proof of  the second main result.   } \label{proofproexpgenfinbis}
For $(1-v) N_\mu$,  Proposition  \ref{proexpDyn}  provides the same  estimates as in  Assertion $(A)$ of Theorem  \ref{proexpgenfinbis}.  For $(1-v) M_\mu(v)$,  Proposition \ref{proexpDyn} provides  estimates  that  are  different  from  Assertion $(A)$ of Theorem  \ref{proexpgenfinbis} : in particular,  the first two terms in \eqref{Mmu}  do not appear in the estimate of  Theorem  \ref{proexpgenfinbis}. 

\smallskip
 The following   lemma proves that the estimates obtained   in Proposition 
\ref{proexpDyn}  and in   Assertion $(A)$ of  Theorem  \ref{proexpgenfinbis} coincide.  It also proves Assertion $(B)$ of Theorem  \ref{proexpgenfinbis}. 
 
 \begin{lemma}  \label{aux} 
 $(i)$ Consider two sequences $a(n),b(n)>0$ and the associated  generating functions  $A(v)$ and $B(v)$.
 Assume  that   the sequences  $a(n)$ and $b(n)$ satisfy 
 \vskip 0.1 cm 
 \centerline{ 
 $a(n) = o(b(n)), \quad \sum_{n=0}^\infty b(n)=\infty$; }
 \vskip 0.1 cm
 then  their gf's satisfy $A(v)=o(B(v))$ as $v\to 1^-$. 
 
 \smallskip 
 $(ii)$ Consider a   measured dynamical system  $(\cal S,\mu)$  in the Class $\cal{DRBGC}$,
     the gf's $H (v)$ defined in \eqref{H} and the  gf  $Q_\nu( v)$.  
  The following holds,  
  \begin{itemize} 
  \item [$(a)$] 
  The functions $Q_\nu(v)$  and  
 $H(v)$  satisfy   
 $$  Q_\nu(v) = o\left(\frac 1 {1-v}\right),  \quad  
 H (v) = o\left(\frac 1 {1-v}\right), \quad (v \to 1^-)$$
 The two first terms in \eqref{Mmu} are negligible wrt to the last term.
 
 \smallskip  
   \item[$(b)$]      
   When  $Q_\tau(v) \to  \infty$ as $v \to 1^-$,     the estimate holds, for any generating function $Q_\mu(v)$ relative to a measure $\mu$ with density $\phi$
 \vskip 0.1 cm 
 \centerline{ $Q_\mu(v) -  \phi(0) Q_\tau (v) =  o\left(Q_\tau  (v)\right), \quad (v \to 1^-)\, . $ }
  \vskip 0.1 cm

  \smallskip  
   \item[$(c)$] The quotient  $D_\mu$  associated in \eqref{Dmu} with any measure $\mu$  is  finite. 
    
  \end{itemize} 
 \end{lemma} 
 \begin{proof}  \ \ 
 
 Item $(i)$. 
 Fix  $\varepsilon>0$. Since $a(n) = o(b(n))$, there is $K=K_\varepsilon > 0$ such that, for $n\geq K$,  the inequality $a(n)\leq \varepsilon b(n)$ holds. Thus
 the generating functions satisfy
 \[
  0\leq A(v) = \sum_{k=1}^{K-1} a(n) v^n + \sum_{k\geq K} a(n) v^n \leq \sum_{k=1}^{K-1} a(n) v^n + \varepsilon \sum_{k\geq K} b(n) v^n\,.
 \]
 As  $B(v)$ is divergent as $v\to 1^-$,  we remark that  the second sum is $\sim B(v)$ as $v \to 1^-$ Hence, dividing through by $B(v)$, the divergence of $B(v)$ as $v\to 1^-$ implies
  \[
  0\leq \limsup_{v\to 1^-} A(v)/B(v) \leq \varepsilon\,.
 \]
 Being $\varepsilon>0$ arbitrary, we have proved that $A(v)/B(v)\to 0$.

 \smallskip
  Item $(ii)(a)$ 
  We use Item $(i)$ with $b(n) = 1$,  and  the  two sequences $q_\nu(n)$ and  $a(n) =  -q(n) \log q(n)$.     As the sequences $q_\nu(n), q(n)$  tend to 0 as $n \to \infty$, the sequences  $q_\nu(n)$ and  $a(n)$  tend to $0$, and  the two  estimates hold. \\
  The two first terms in \eqref{Mmu} satisfy 
  $$  I[S[\phi]\log S[\phi]]= O(1) = o\left( \frac 1 {(1-v) Q_\nu(v)}\right),  \quad   \frac {H(v)}{Q_\nu(v)}  = o\left(    \frac{ 1}  {(1-v) Q_\nu(v)} \right) $$
and are negligible wrt to the third term.     

   \smallskip
  Item $(ii)(b)$.   As   the set  $[W>n]$ coincides with the interval $[0, q_\tau(n)]$,  one has  
 \vskip 0.1 cm
 \centerline
 {$q_\nu (n) = \nu\big[W>n\big] = \int_{[W>n]} \psi d\tau = \int_0^{q_\tau(n)} \psi(x) dx \, .$} 
 \vskip 0.2 cm 
  As $\psi$ is of class ${\cal C}^1$, one has $\psi (x) = \psi (0) + O (x)$  as $x \to 0$. As  $q(n)=q_\tau(n)$ tends to $0$,  we obtain 
  \vskip 0.1 cm
  \centerline{$q_\nu(n)=\int_0^{q(n)} \psi(x) dx =  \psi(0)\,  q(n) + O(q(n)^2) $.}
  \vskip 0.1 cm
  We apply Item$(i)$  to $b(n) = q(n)$ and  $ a(n) = q(n)^2 = o(q(n))$ .
     
   \smallskip 
   Item $(ii)(c)$ In the case when $Q_\tau(v)\to \infty$ as $v \to 1^-$, the quotient $D_\mu$ equals the ratio $\phi(0)/\psi(0)$ and is thus finite.  Otherwise, when $Q_\tau(1)<\infty$,  the sequences  $q_\mu(n)$ and $q_\nu(n)$ satisfy
   $$ q_\mu(n) = \int_0^{q_\tau(n)} \phi(t) dt, \qquad  q_\nu(n) = \int_0^{q_\tau(n)} \psi(t) dt\, . $$
    As the two densities $\phi$ and $\psi$ are $>0$ and belong to $\cal C^1$,   the associated series  are thus convergent with sums  resp. equal to $Q_\mu(1)$ and $Q_\nu(1)$ and the quotient $D_\mu$ equals the ratio $Q_\mu(1)/Q_\nu(1)$.
   
 \end{proof}

%%%%%%%%%%%%%%%%%%%%%%%%%%%%%%%%%%%%%%%%%%%%%%%%%%%%%%%%%%%%%%%  

 Consider a   measured dynamical system  $(\cal S,\mu)$  in the Class $\cal{DRBGC}$. Then,  
   Proposition  \ref{proexpDyn}   and     Lemma \ref{aux}   apply  for a measure $\mu$  with a density $\phi >0$ of Class $\cal{C}^1$. Together, they imply  Theorem   \ref{proexpgenfinbis}.  
   
      \subsection {Two types of renewal equations}  \label{infiniteergodic} 
Aaronson   considers the  invariant measure $\pi$  under $\cal S$ for which $\pi(\cal J) = 1$ -- exactly the measure associated with the function $(I-\bf A)^{-1} [\psi]$ of Proposition \ref{blockJet pasJ}-- and uses the measure $\pi_{\cal J}$, 
 \begin{equation} \label{piJ} \pi_{\cal J} (X) = \pi (X \cap \cal J)
 \, . 
 \end{equation} He   deals with the  measures $ \pi_{\cal J} (X)$  particularly when $X =  T^{-k} (\cal J)$  and    exhibits     the following  renewal equation (see  \cite {aa1} page 1042)   that relates the  two following  generating functions :

  \begin{lemma} {\rm [Aaronson]}\label{renewaa}  The two   generating functions   --whose coefficients  are defined in terms of the sequence $\pi_{\cal J} (T^{-k} (\cal J))$ that  involves  the  measure  $\pi_{\cal J}$  defined in \eqref{piJ}--   satisfy the renewal equation
   \begin{equation}  \label{Aarenter} \left( \sum_{k \ge 0}   \pi_{\cal J}[   T^{-k} {\cal J}] v^k \right) \cdot   \left( 1 -\sum_{ k \ge 0}    \pi_{\cal J } \left[ T^{-k } {\cal J}  -\left(  \bigcup_{\ell <k}  T^{-\ell} {\cal J}\right) \right]   v^k \right) \sim_{v \to 1}  1\,  .
     \end{equation}
 \end{lemma}
 
  \begin{proof}    Aaronson    first proves  that, as soon as the block source  $\cal B_{\cal J}$ is $\psi$-mixing,   
  the  second interval $ {\cal J} = [c, 1]$ of the system  (where the symbol one is emitted) is a Darling-Kac set\footnote{We do not give a precise definition of the notion of Darling-Kac and $\psi$-mixing. For precise definitions, see  the papers of Aaronson, and \cite{DK}}. 
      As $\cal J$  is  a Darling-Kac set, Aaronson obtains 
   the following  renewal equation     between two  functions (here $\tau $ is a real positive and $W$ is the waiting time of ${J}$), 
       \begin{equation} \label{Aarenbis}  \left[ \sum_{k \ge 0}  \pi[{\cal J} \cap  T^{-k} (\cal J)] e^{-\tau k}\right]\left[ \int_{\cal J} (1- e^{-\tau W} ) d\pi (x) \right] \sim_{\tau \to 0}   \pi({\cal J})^2\, .
     \end{equation}
     The estimate of the integral is 
     $$ \int_{\cal J} (1- e^{-\tau W} ) d \pi(x)  = \sum_{k \ge 0}  \int_{\cal J} [\![ W = k]\!] (1-e^{-\tau W}) d \pi(x)
     = \sum _{ k \ge 0} (1- e^{-\tau k}) \pi  \left({\cal J} \cap [\![ W = k]\!]  \right)$$
     $$ 
     = \pi ({\cal J}) -  \sum_{ k \ge 0}   \pi  \left({\cal J} \cap\left[ T^{-k }{\cal  J } -\left(  \bigcup_{\ell <k}  T^{-\ell} {\cal J}\right) \right]\right)  e^{-\tau k}\, .$$
     And finally the left member in  Aaronson's  renewal equation  \eqref{Aarenbis}   is  (with $v = e^{-\tau}$), 
     $$ \left( \sum_{k \ge 0}  \pi[{\cal J} \cap  T^{-k} {\cal J}] v^k \right) \cdot   \left( \pi (\cal J) -\sum_{ k \ge 0}  \pi   \left[ {\cal J} \cap\left[ T^{-k } {\cal J}  -\left(  \bigcup_{\ell <k}  T^{-\ell} {\cal J}\right) \right]\right]   v^k \right)\, ;$$ 
     In other words, dealing with  the measure $ \pi_{\cal J}$  defined in \eqref{piJ}, 
      Aaronson  obtains the renewal equation  described in \eqref{Aarenter}. 
     \end{proof}

  \smallskip  We have also obtained,  for any measure $\mu$  with a strictly positive  density $\phi$ of Class $\cal{C}^1$, 
 a renewal equation between    ``our'' generating functions
\vskip 0.1 cm 
\centerline{ $(1-v) Q_\mu(v), \quad  (1-v) N_\mu(v)$}
\vskip 0.1 cm    whose coefficients are defined via  the measure  $\mu$. 

\medskip     
 There is thus a strong parallelism between the    two renewal equations, the previous equation \eqref{Aarenter} and  our renewal equation described in \eqref{ourren}.

 \begin{lemma}  The two renewal equations, the one obtained by Aaronson in \eqref{Aarenter} and ours described in \eqref{ourren}   for $\mu = \nu$  may be deduced from each one to another one.  
\end{lemma}  
\begin{proof} Due to   Proposition   \ref{blockJet pasJ}$(iv)$. 
\end{proof} 

We remark that the renewal equation obtained in the  present paper is more general than Aaronson's one since it deals with any measure $\mu$  with a strictly positive  density $\phi$ of Class $\cal{C}^1$, and not only with the  invariant measure $\nu$. 

%%%%%%%%%%%%%%%%%%%%%%%%%%%%
%%%%%%%%%%%%%%%%%%%%%%%%%
   
   \section{Proof of  the  third main result.}     \label{3-4}
   
   We first list  in Section \ref{sec:SV}  the main  properties of slowly varying functions. Some of them will be used in the next Section \ref{auxthm3}, which states an auxilliary result 
   of interest in the proof of the third main result, which will be provided in Section \ref{Proof-t3}. 
   
   %%%%%%%%%%%%%%%%%%
   
     \subsection{Main properties of slowly varying functions}  \label{sec:SV}
We  collect here  definitions and main properties   that deal with  slowly varying functions. See  \cite {BoS}, \cite{BGT} and \cite{GS} for more detailed studies.

\smallskip   
$(P1)$ {[\rm  Basic definitions]}  A  positive function  $V: ]0, + \infty[$ belongs to $\cal {SV}$ (slowly varying in $+\infty$)  iff, for every $t>0$, the ratio  
  $  {V(tx)} /{V(x)}$ tends to  1 as $x \to \infty \, .$\\
  A function $V$ belongs to $\cal {SV}_0$ (slowly varying in $0$) if 
    $x \mapsto V(1/x)$  belongs to $\cal{SV}$

    \smallskip 
$(P2)$   A function  $V$ belongs to $\cal {SV}$  iff, for every $t>0$, the convergence in $(P1)$ 
  is uniform wrt to $t \in[a, b]$  with  $0<a<b<\infty$
  
    \smallskip 
$(P3)$   A  function $V$ belongs to $\cal {SV}$ iff it admits the decomposition   
   $$   V(x) = \zeta (x) \exp \left[ \int_1^{x} \frac {\epsilon (t)}{t} dt\right] ,$$
  $  \hbox {with} \quad   \epsilon (x) \to 0, \quad \zeta(x) \to C,   \ \  (0 < C < \infty) \quad (x \to \infty) $\, .

 \smallskip 
 $(P4)$ For   $V\in \cal{SV}_0 \cap {\cal C}^1 $, the decomposition of $(P3)$
  holds with  $ \zeta  \in {\cal C}^1,  \epsilon \in \cal{C}^0$. 
  
     \smallskip 
  $(P5)$   For any $V_1\in \cal {SV}_0\cap {\cal C}^1 (]0, 1])$, there exists $V \in {\cal C}^1(]0, 1])$ such that, 
  \vskip 0.1 cm 
  \centerline{  $V(x)\sim V_1(x), \qquad 
  \frac {xV'(x)}{V(x)} \to 0, \quad  (x \to 0) $\, .}

  \smallskip 
  $(P6)$ 
  If $V \in \cal {SV}_0$, then, for any $\epsilon >0$, there exists $A_\epsilon $ such that, for $0 < x < A_\epsilon$, 
  \vskip 0.1 cm
  \centerline{ $x^{\epsilon} \le V(x) \le x^{-\epsilon}\, .$}

%%%%%%%%%%%%%%%%%%%
\subsection{An auxilliary result.}  
The following result  plays a central  role in  the   proof of Theorem \ref{final1}:     \label{auxthm3}
         
        \begin{proposition} \label{SeqSV}
 When   a sequence $V_n$ satisfies \eqref{qdiv1},   the sequence   $\widetilde V_n$  defined  as 
 \begin{equation} \label{eqtildeVn}
  \widetilde V_n = \frac 1 {n^{1-\beta}} \sum_{k < n} \frac {V_k}{k^\beta} \, , 
 \end{equation} 
 \vskip -0.6 cm
 $$\hbox{satisfies} \qquad \widetilde V_n \sim  \frac 1 {1- \beta} \, V_n  \quad (0 \le \beta< 1),  \qquad \widetilde V_n=  \Theta (\log n)^{\delta+1} \quad (\beta = 1).  \qquad \qquad \qquad $$
 \vskip -0.2 cm  For any $\beta \in [0, 1]$,  the sequence $\widetilde V_n$ belongs to $\cal {SV}$. \end{proposition}  
 
The estimate of Proposition \ref{SeqSV}  holds when the sum in Eqn \eqref{eqtildeVn} is replaced by the integral:  in the book \cite{Fe},  Theorem 1, p. 281,  the precise following result  is  indeed proven,  

\smallskip
{\sl When   a sequence $V_n$ satisfies \eqref{qdiv1} with $\beta \in [0, 1]$,    the sequence   $\widehat  V_n$  defined  as 
 \begin{equation} \label{Qmu}
  \widehat V_n = \frac 1 {n^{1-\beta}}\int_1^n \frac {V(x)}{x^\beta} \, dx \, , 
 \end{equation} 
 \vskip -0.3 cm
 $$\hbox{satisfies} \qquad \
 \widehat V_n \sim  \frac 1 {1- \beta} \, V_n  \quad (0 \le \beta< 1),  
 \qquad \widehat V_n=  \Theta (\log n)^{\delta+1} \quad (\beta = 1). \qquad \qquad \qquad  $$}
 
  \begin{proof}  
   It is then sufficient to prove   the estimate $\widehat  V_n \sim \widetilde V_n$.    With Property $(P6)$, and for $\beta \in [0, 1]$, the series of general term $V(k)/k^\beta$ is divergent. With  the decomposition 
 \[
 \int_1^n\frac{V(x)}{x^{\beta}}dx=\sum_{k=1}^{n-1}\int_k^{k+1}\frac{V(x)}{x^{\beta}}dx \, , 
 \]
  it is sufficient to prove  that   
 \begin{equation} \label{sim}
 \int_k^{k+1}\frac{V(x)}{x^{\beta}}dx  \sim \frac{V(k)}{k^{\beta}}  \quad (k \to \infty)\, .
 \end{equation}
 With  the  change of variable  $x = \lambda k$ in the integral over $[k, k+1]$,  
 \[
 \int_k^{k+1}\frac{V(x)}{x^{\beta}}dx=\frac{k}{k^{\beta}}\int_1^{1+\frac{1}{k}}\frac{V(\lambda k)}{\lambda^{\beta}}d\lambda\, , 
 \]
  the ratio  between the two terms of Eqn \eqref{sim}  equals 
 \begin{equation} \label{sim1}
  k\int_1^{1+\frac{1}{k}}\frac{V(\lambda k)}{V(k)}\frac{1}{\lambda^{\beta}}d\lambda.
 \end{equation}
 Now, as $V\in\cal{SV}$,   with Property $(P2)$, the ratio $V(\lambda k)/{V(k)}$ tends to 1  as $k\to\infty$ uniformly in $\lambda\in[1,2]$ and this implies
 \[
 \forall\epsilon>0, \ \exists N_{\epsilon},\  \forall k>N_{\epsilon}, \ \forall\lambda\in[1,2],   \quad \frac{1-\epsilon}{\lambda^{\beta}}<\frac{V(\lambda k)}{V(k)}\frac{1}{\lambda^{\beta}}<\frac{1+\epsilon}{\lambda^{\beta}}.
 \]
 By integrating with respect to $\lambda \in [1, 1 +(1/k)]$, it follows  that the ratio of Eqn \eqref{sim1} satisfies  for $\epsilon>0 $ and $k>N_{\epsilon},$
 \[\frac{k(1-\epsilon)}{1-\beta}\left[\left(1+\frac{1}{k}\right)^{1-\beta}-1\right]<k\int_1^{1+\frac{1}{k}}\frac{V(\lambda k)}{V(k)}\frac{1}{\lambda^{\beta}}d\lambda<\frac{k(1+\epsilon)}{1-\beta}\left[\left(1+\frac{1}{k}\right)^{1-\beta}-1\right].
 \]
   When  $k$ goes to infinity, the ratio of Eqn \eqref{sim1} is between $(1-\epsilon) [1 + O(1/k)]$ and $(1+ \epsilon)[1 + O(1/k)] $, 
  and when $\epsilon$ goes to zero, this  concludes the proof.

 \end{proof}
 
 %%%%%%%%%%%%%%%%%
 
 \subsection{Proof of the third main result Theorem~\ref{final1}.}   \label{Proof-t3}
 \ \ \ 
 \smallskip

 {\sl  $(A)$} 
  When   $q_\nu(n)\ge 0$  is of the form \eqref{qdiv1},   the partial sum  $ Q^{(\nu)}_n$ of the first $n$ coefficients of the gf $Q_\nu(v)$  is written in terms of $\widetilde V_n^{(\nu)}$ defined in  \eqref{eqtildeVn} 
 \vskip 0.1 cm
 \centerline{ 
$ Q^{(\nu)}_n = \sum_{k <n} q_\nu(n) =  n^{1- \beta} \, \widetilde V_n^{(\nu)} $ \,  ,} 
\vskip 0.1 cm
 and Proposition \ref{SeqSV}   shows  that $\widetilde V_n^{(\nu)}$ belongs to  $\cal{SV}$ for any $\beta \in [0, 1]$.

 \medskip Moreover, with  Thm \ref{HLKboth}, and   for any $\beta \in [0, 1]$,  
  the function $Q_\nu(v)$  satisfies (2)  of Thm \ref{HLKboth} with  exponent  $1-\beta \ge 0$ and  sequence ${\underline V}_n^{(\nu)}$,   equal to 
 \begin{equation} \label{undV} 
 {\underline V}_n^{(\nu)} = \Gamma(2-\beta) \cdot  \widetilde  V_n^{(\nu)} \, .
  \end{equation}
  We  consider the function $\underline V^{(\nu)}$ defined from the sequence $\underline V_n^{(\nu)}$ as in  Theorem \ref{HLKboth}.     Now,  the function $\widehat Q_\nu(v)= (1-v) Q_\nu(v)$  satisfies (1) of Thm \ref{HLKboth} with the exponent  $-\beta \in [-1, 0]$ and the function $\underline V^{(\nu)}$.   
     
    \smallskip 
     We then use Thm \ref{proexpgenfinbis}  that  relates 
 the function $\widehat N_\mu(v)= (1-v) N_\mu(v)$   associated with the generating function $N_\mu(v)$ of the number of ones   and the function $\widehat Q_\nu(v)$ via  the estimate as the real $v \to 1^-$, 
 \vskip 0.1 cm 
 \centerline{ $ \widehat N_\mu(v) \sim    1/  \widehat Q_\nu(v)$.}
 \vskip 0.1 cm 
  Then  $\widehat N_\mu(v)$  satisfies (1) of Thm \ref{HLKboth}   with an exponent $ \beta \in [0, 1]$ 
  and  the function  $1 /\underline V^{(\nu)}$. This fonction   belongs to $\cal {SV}$  due to  Proposition \ref{SeqSV} and  Eqn \eqref{undV}.

 \smallskip As $\widehat N_\mu(v)$ has positive coefficients, whose partial sums of order $n$ coincide  with
$\underline n_\mu(n)$,    this entails, with Thm \ref{HLKboth},  that   the sequence $\underline n_\mu(n) $ satisfies (2) of  Thm \ref{HLKboth} with  exponent $\beta$ and the sequence 
\vskip 0.1 cm 
\centerline {$ (1/\Gamma (\beta+1))1 /\underline V_n^{(\nu)}$. } 
\vskip 0.1 cm    Then, using the expression of $ \underline V_n^{(\nu)}$ in terms of $\widetilde V_n^{(\nu)}$  in  \eqref{undV},   the estimate  holds, 
\begin{equation} \label{estN1} \underline n_\mu(n) \sim 
 \frac {n^{\beta} }  {\Gamma(\beta+1 )}  \frac {1} { \underline V_n^{(\nu)}}  \sim   
 \frac {n^{\beta} }  {\Gamma(\beta+1 ) \Gamma (2- \beta)}  \frac {1} { \widetilde V_n^{(\nu)}}  
  \end{equation} 
Using  the expression $\widetilde V_n^{(\nu)}$ in terms of  $Q^{(\nu)}(n)$ in \eqref{Qmu}  leads to the estimates 
 \begin{equation} \label{finalest} 
\underline n_\mu(n) \sim      \frac 1  {\Gamma(\beta +1 ) \Gamma(2-\beta) } \frac {1} { Q^{(\nu)}_n}\, n \, ;
 \end{equation}
 As  $Q^{(\nu)}_n \to \infty$  by hypothesis,  this proves  the first result of Assertion $(A)$.

 \medskip 
 In the same vein,  
 we apply the  Tauberian step  (Implication $(1) \Longrightarrow (2)$)  to  the function $(1-v)M_\mu(v)$. As the two  generating functions $N_\mu$ and $M_\mu$ satisfy, as an indirect consequence of Theorem~\ref{proexpgenfinbis}, 

\vskip 0.1 cm
\centerline{$(1-v) M_\mu(v) \sim  {\cal E}_\nu({\cal B}) \,  (1-v) N_\nu(v)\, $, }
\vskip 0.1 cm   this gives rise to  the   estimate of Shannon weights  given in Eqn \eqref{heu}
 
  \medskip
$(B)$ 
In the  case when  $q_\nu(n)\ge 0$  is of the form \eqref{qdiv},  the following  holds for $\widetilde V_n^{(\nu)}$
$$ \widetilde V_n^{(\nu)}  \sim  \frac { K_\nu}  {1- \beta} (\log n)^\delta , \quad (0\le \beta <1), \quad \widetilde  V_n^{(\nu)} = \ \frac {K_\nu} {\delta +1} (\log n)^{\delta +1}$$
Using then  the general estimate \eqref{estN1} for $\underline n_\mu(n)$  leads to  the estimates \eqref{estN2}.   
 
 \medskip 
 $(C)$ 
The   partial sum $Q^{(\nu)}_n$  satisfies, 

\vskip 0.1 cm 
\centerline{ 
$ Q^{(\nu)}_n \sim  Q_\nu(1), \quad n \to \infty$. }
 Moreover,  
  the function $Q_\nu(v)$  satisfies (2)  of Thm \ref{HLKboth} with  exponent  $0$ and the constant   function  $ Q_\nu(1)$.     
 The function $\widehat Q_\nu(v)= (1-v) Q_\mu(v)$  satisfies (1) of Thm \ref{HLKboth} with the exponent  $-1$ and the constant  function $Q_\nu(1)$. 
     
    \smallskip 
     We will  use Theorem \ref{proexpgenfinbis}  that  relates 
 the function $\widehat N_\mu(v)= (1-v) N_\mu(v)$   associated with the generating function $N_\mu(v)$ of the number of ones   and the function $\widehat Q_\nu(v)$ via  the estimate as the real $v \to 1^-$, 
 \vskip 0.1 cm 
 \centerline{ $ \widehat N_\mu(v) \sim    \,  1/  \widehat Q_\nu(v)$.}
 \vskip 0.1 cm 
  Then  $\widehat N_\mu(v)$  satisfies (1) of Thm \ref{HLKboth}  
   with the exponent $ 1$ and the function  $1 /Q_\nu(1)$.
   As $\widehat N_\mu(v)$ has positive coefficients, whose partial sums of order $n$ coincide  with
$\underline n_\mu(n)$,    this entails, with Thm \ref{HLKboth},   the estimate 
\begin{equation*}  \underline n_\mu(n) \sim  \frac {1}{Q_\nu(1)}  n , \qquad  \underline n_\mu(n) \sim  \frac {1}{\E_\nu[W]}  n\, .
 \end{equation*}

This ends the proof of  our main third result, Theorem \ref{final1}.

 \section {The Class $\cal{DRIS}$ } \label{sec:dri} 
 
 This Section  aims at exhibiting  explicit  instances  of  sources of the Class $\cal{DRBGC}$   that will be found  in  a subclass   of the class $\cal{DRI}$, denoted as $\cal{DRIS}$.    
 The short name $\cal{DRIS}$ stands for  dynamical recurrent sources which exhibits an  indifferent fixed point with slowly varying behaviour.  Such a class  indeed gathers systems  for which  $0$ is an indifferent fixed point for branch $a$,     %Here, we deal with  a strong notion of indifferent point  which is based on  
  with a precise description of the contact at $x = 0$  between the graph of map $a$ with its tangent (the  line $a(x) = x$) which involves slowly varying functions.

\smallskip
Section \ref{sec:indif}, with Definition \ref{indif},  precisely defines the  (strong) notion of indifferent fixed point and the associated subclass $\cal {DRIS}$.  Section \ref{otherDRI} compares our definition with other definitions of the same spirit that are found in the literature. Then,  Section \ref{qnDR} uses a result due to Aaronson  (for which we provide a precise proof) and obtains first results on the wtd sequence $q(n)$, that  provide  a first estimate on the convergence abscissa   of the Dirichlet series  which intervenes in  Item $(a)$ of the Good Class (see Definition~\ref{good_class}).  Finally, Section \ref{DRGC} with its Theorem \ref{deltaa},    proves  the main  result of the Section:  the block source of a $\cal{DRIS}$ source   belongs to the Good Class.  

 %%%%%%%%%%%%%%%%%%%%%%%%%%%%%%%%%%%%%%%%%%%%%%%%%%

  \subsection{Notion of  strong indifferent fixed point : the class $\cal{DRIS}$.} \label{sec:indif}
 
 Consider  a source of the Class $\cal{DRI}$, for which the branch $a$  admits $0$ as an indifferent fixed point.  In this case, the equalities   $a(0) = 0$ and  $a'(0) = 1$ hold, 
 and  we introduce a more precise notion, that deals with the set $\cal{SV}_0$ of slowly varying functions at $x= 0$, in order to describe the behaviour of the derivative $a'$ near $0$.

  \medskip
  We consider a particular subclass of  $\cal{SV}_0$, already mentioned in $(P5)$ of Section \ref{sec:SV},  and introduced in \cite{dB},  whose definition is recalled here.  
  
  \begin{definition}  The class $\cal{SV}_0^\ast$ gathers the functions $V\in \cal{SV}_0$ of class ${\cal C}^1(]0, 1])$ for which
   \[
 \epsilon (x) =   \frac {xV'(x)}{V(x)} \to 0    \quad (x \to 0) \, . \]
    \end{definition}
A result due to de Bruijn, and   proven in \cite{BGT} page 15,  states the following: 

\begin{quote} for any $V_1\in \cal {SV}_0\cap {\cal C}^1 (]0, 1])$, there exists $V \in \cal {SV}_0^\ast $ 
that satisfies  $V(x)\sim V_1(x)$ as $x \to 0$. 
\end{quote} Most of the natural functions of $\cal{SV}_0 \cap {\cal C}^1 (]0, 1])$ belong to $\cal{SV}_0^\ast$, for instance the functions of $\cal{C}^1([0, 1])$ or any power (positive or negative) $x \mapsto (\log x) ^\delta$.

     \medskip The following definition deals with the set $\cal{SV}_0^\ast$ and provides the notion of order of an indifferent fixed point.
  \begin{definition} [indifferent fixed point of order $\gamma >0$]  \label{indif}
  A map  $a$  defined on the unit interval ${\cal I} = [0, 1]$, and of class ${\cal C}^2([0, 1])$ satisfying $a(0) = 0$ 
    has an indifferent fixed point at $x = 0$ of order $\gamma>0$ if  there exists  a positive function $V \in \cal {SV}_0^\ast$  for which $a$ is   written, for any $x \in {\cal I}$,  as 
   \begin{equation} \label{uu} 
 a(x) = 
 x-  u(x), \quad u'(x) = 
 x^{\gamma}\, V(x),  \quad u'(x) \in [0, 1] \, ,   \end{equation} 
  with  the conditions at $x= 0$,  $u(0) = 0, u'(0) = 0$. 
   One lets $v(x)= u(x)/x$ so that 
   
   \vskip 0.1 cm 
   \centerline{$a(x) = x(1- v(x))$. }
  \end{definition}
  
  We now describe first properties of  such a branch  $a$.
  
   \begin{lemma}  \label{DRgamma1} When $a$ has an indifferent  fixed point of order $\gamma >0$,  with $a'(x) = 1- x^\gamma V(x)$ and $V \in  \cal {SV}_0^\ast $,  
   the following holds: 
   \begin{itemize} 
   \item[$(i)$]  $a(x) \in [0, 1] $ for $x \in [0, 1]$.
   
   \item[$(ii)$]  there exists a neighborhhood of $0$, say $[0, x_0]$  on which the fonction $a$ is concave 
           
            \item[$(iii)$]  
 the function $v$ satisfies 
   $$ 
 v(x) =  \frac {x^{\gamma} } {\gamma +1}   V_1 (x), \quad V_1(x) \sim V(x) \quad (x \to 0)$$
 and $V_1$ is  slowly varying 
 
 \end{itemize} 
    \end{lemma}

  \begin{proof}  
  $(i)$ First, since $u'(x)\in [0, 1]$, it is the same for $1-u'(x)$ and 
  
  \vskip 0.1 cm 
  \centerline{$a(x) = x-u (x) = \int_0^x (1- u'(t)) dt \in [0, 1] $ for $x \in [0, 1]$.}

\smallskip  
 $(ii)$  The derivative  $u''(x)$ satisfies
 \vskip 0.1 cm \centerline{ $ 
 u''(x)= \gamma \, x^{\gamma -1}V(x)+ x^{\gamma } V'(x) = \gamma  x^{\gamma-1} V(x) \left[1 +  \epsilon(x)\right], \quad \epsilon(x) =\frac 1 \gamma  \frac {xV'(x)}{V(x)}  $} 
 \vskip 0.1 cm 
 As $V\in  \cal {SV}_0^\ast$,    this entails  
 \vskip 0.1 cm   
 \centerline{ $a''(x) \sim  - \gamma\,   x^{\gamma-1} V(x)  \quad (x \to 0)\, .$}

\smallskip
$(iii)$  with integration by parts, the function $u$ satisfies 
$$ u(x) = \int_0^x t^\gamma  V(t) dt = \Big [ \frac {t^{\gamma+1} } {\gamma +1}  V(t) \Big]_0^x - \int_0 ^x   \frac {t^{\gamma+1} } {\gamma +1} V'(t) dt \, .$$
This entails 
$$ \left| u(x) -  \frac {x^{\gamma+1} } {\gamma +1}  V(x) \right|  \le  \frac {\epsilon_1 (x)} {\gamma +1}\int_0^x t^\gamma  V(t) dt = \frac {\epsilon_1 (x)} {\gamma +1}u(x)  \, , $$
$$ \hbox{where} \quad  \epsilon_1 (x) = \max_{t \in [0, x]}   \epsilon (t), \quad \epsilon (t) =  \frac {xV'(t)}{V(t)} \quad \hbox{tends to $0$  as $x \to 0$}$$
and thus  Item $(c)$ holds.
 \end{proof} 

Remark: The functions $v$ and $ (\gamma +1) u'$ are thus equivalent as $x \to 0$.

   \begin{definition}  \label{DRI}  
   We denote by   $\cal{DRIS}(\gamma)$    the  subset   of   the class $\cal{DRI}$   which  gathers the systems for which the  first inverse branch    $a$     admits an indifferent  fixed point of order $\gamma >0$, and we let  
   
     \centerline 
   {$\cal{DRIS} = \bigcup_{\gamma >0} \cal {DRIS} (\gamma) .$}
   
   \end{definition} 
   
   As explained in Section \ref{waitbis}, when the branch $a$ admits an indifferent point of order $\gamma >0$,   there are two main cases for the  derivative $b'(0)$  of branch $b$ : 
   -- a    strong  case, called the   case $(s)$,   where 
     $|b'(0)| < 1$, 
     -- a weak case,  called the case $(w)$  where $b'(0) = -1$. 
   According to the two possible cases  for the branch $b$, there   are thus two subclasses 
 \vskip 0.1 cm
 \centerline{$\cal{DRIS}_s(\gamma)$ and $\cal{DRIS}_w(\gamma) $\, .}
   We  have seen that the Farey system has a branch $b$   of the  case $(w)$.

 \medskip  
  The   importance of   the $\cal{DRIS}$ Class in this paper is due to   the following Theorem  \ref{deltaa},  
  which will be proven in Section \ref{DRGC}.  
   \begin{theorem}  \label{deltaa}
When the source $\cal S$ belongs to the $\cal {DRIS}$ Class,  its  block source ${\cal B}$ belongs to the Good Class, and  the inclusion $\cal{DRIS} \subset \cal{DRBGC} $ holds. 
\end{theorem}

%%%%%%%%%%%%%%%%%%%%%

\subsection{About  our Class  $\cal{DRIS}$} \label{otherDRI}
   The idea of dealing with a  direct branch $A$  of class ${\cal C}^2$ with an indifferent point at $0$   seems  first  due to Manneville and Pommeau \cite{PM}. Then, many authors deal with this type of sources,  and make precise the notion of indifferent fixed point they use.  Thaler in \cite{Tha1} gives a precise definition, and obtains  precise results on these sources that will be used later on by Aaronson \cite{aa1}. Then, Holland \cite{Ho} deals with  maps $T$  whose  direct branches $A$ are  of class ${\cal C}^2$ and written  as 
   \vskip 0.1 cm
   \centerline{ $A(x) = x( 1+ x^\gamma \underline V(x))$ with $\underline V$ of  slow  varying. }
   \vskip 0.1 cm 
  We remark that  all the tools Holland   needs  in  the study of ``his''  class  can be  in fact expressed with the inverse branch $a$. 
  
  \smallskip
   This is why   we  directly define our class $\cal{DRIS}$ via  the inverse branch $a$, and more precisely via the derivative $a'$ of the inverse branch $a$.  Then, there are two expressions of the branch $A$ or its inverse $a$,  
  \vskip 0.1 cm 
  \centerline{  $ a' (x) = 1- x^\gamma V(x)  \quad \hbox {[the present approach]},$}
 
 \vskip 0.1 cm  
\centerline{ $      A (x) = x( 1 +x^\gamma  {\underline V}(x), \quad \hbox{[Holland's approach]} \, .$}
  \vskip 0.1 cm
 Via Implicit Function Theorem, there  thus exists a relation   between $V$ and $\underline V$  that is however  not straightforward: this explains that the computations made by Holland are often  less direct than ours, notably when  he deals with convexity.

  %%%%%%%%%%%%%%%%%%%%%%%%%%%%%%%%%%%%%%%%%%
  \subsection{First results  on  the waiting time   in the class $\cal {DRIS}$. } \label{qnDR}
  Relation  \eqref{uu} 
  gives rise to a recurrence  for the sequence $q_\tau(n)=\tau[W > n] = a^n(1)$  associated with  the distribution of the  waiting time with respect to the uniform measure $\tau$, 
  \begin{equation} \label{recqn}
  q_\tau (0) = 1 \quad q_\tau(n+1)  = q_\tau(n) \big (1- v(q_\tau(n)) \big )\, . 
   \end{equation} 
   In the sequel of this section, the index $\tau$  is omitted.
  
  \smallskip
  The next proposition, based on a result due to  Aaronson,  provides  first estimates on the behaviour of  the wtd sequences $q(n)$ and $r(n)$. This  will be useful for  describing the behaviour of the Dirichlet series  $\Delta(s)$
    which plays an important role in Item $(a)$ of  block sources.
 
   \begin{theorem}
  \label{qneval} 
  For  a source of  the 
  $\cal{DRIS}(\gamma)$ Class,  with $\gamma >0$, the following holds: \begin{itemize}  
 \item[$(i)$]  The  distribution of the waiting  time $W$  satisfies, 
$$  v(q(n)) \sim  \frac 1 {\gamma n}\quad   (n \to \infty) \, . $$
where $v$ is the function involved in Definition \ref{indif}

\item[$(ii)$]   For every $\varepsilon\in (0,\gamma)$, there exists $N=N_\varepsilon$ and constants $A,\tilde{A}>0$ for which, for any $n \ge N$, 
\vskip 0.1 cm 
\centerline {
  $ \tilde{A} \, \left( \frac 1 n\right)^{1/(\gamma-\varepsilon)} \leq q(n) \leq {A}\, \left( \frac 1 n\right)^{1/(\gamma+\varepsilon)}  $\, .}
  \vskip 0.1 cm 
  \noindent  In particular,  the mean $\E[W] $  is infinite  for $\gamma>1$ and is finite for $\gamma <1$.
  
  \item[$(iii)$]   When $\gamma <1$,  for any $\widehat \gamma $ with $ \gamma <\widehat \gamma  <1$,  there exist   an integer $L (\widehat \gamma)$  and a constant $A(\widehat \gamma)$ such that for any $\ell >L (\widehat \gamma)$, the inequality  holds 
  \begin{equation} \label{widehatgamma} 
   \ell  \le  A(\widehat \gamma) \,  q(\ell) ^{-\widehat \gamma} \, . 
   \end{equation}

\item[$(iv)$]   There exists $N$ for which, for any $n \ge N$,  the  estimates hold 
\vskip 0.1 cm 
\centerline {
  $ r(n) = q(n+1) - q(n) \le  \displaystyle \left( \frac 1 n\right)^{1 + (2/(2\gamma+1))}  \qquad v(q(n)) -v(q(n+1)) \sim \frac 1 {\gamma}  \frac 1 {n^2} $\, .}
  \vskip 0.1 cm

   \item[$(v)$] The convergence abscissa  of the Dirichlet series $\Delta_\tau(s) = \sum_{m \ge 1} r(m)^s$  is less than 
  $  [1 + (2/(2\gamma+1))]^{-1}$ and  
  is stricly less than 1.
  \end{itemize} 
     \end{theorem}

     \begin{proof} 
     
       $(i)$  
   For a  regularly varying map $v(x)  = x^{\gamma} V_1(x)$ with index $\gamma$ at $x = 0$,  we use \eqref{recqn}  and the two estimates
   $$ 
   \frac 1 {v(q(n+1))}  =  \frac 1 {q(n+1)^\gamma V_1(q(n+1)) } \qquad   \frac 1 {v(q(n))} = \frac 1 {q(n)^\gamma V_1(q(n))}\, .$$
    This entails
      $$ 
    \frac 1 {v(q(n+1))}  = 
  \frac 1 {v(q(n))}    \frac 1 { (1- v(q(n)))^\gamma}  \cdot \frac {V_1(q(n))}{V_1(q(n+1))}\, .$$
   Then, the difference satisfies, 
  \begin{equation} \label{diff}    \frac 1 {v(q(n+1))}   -   \frac 1 {v(q(n))}  =   \frac 1 {v(q(n))}   \left[   \frac 1 { (1- v(q(n)))^\gamma}  \cdot \frac {V_1(q(n))}{V_1(q(n+1))} -1\right]\, .
  \end{equation}
  \smallskip
 For a general slowly varying map $V_1$,  the  asymptotic  study of the ratio  
 \vskip 0.1 cm 
 \centerline{ $  {V_1(q(n))}/{V_1(q(n+1))}$}
 \vskip 0.1 cm  is  thus central and Aaronson   provides
 the following estimate, for which we recall the proof: 
  
  \begin{lemma} \label{AarLemma} Consider  the sequence $q(n)$ defined as in  \eqref{recqn}, with $v(x) = x^\gamma V_1(x)$ and $V_1 \in \cal {SV}_0$ of class ${\cal C}^1$. Then 
 \begin{equation} \label{ratioV1}
  \frac {V_1(q(n))}{V_1(q(n+1))} = 1 + o(v(q(n))\, .
  \end{equation}
 \end{lemma}  
 \begin{proof} 
  When $V_1$ is of class ${\cal C}^1$, with $V_1 (0) \not = 0$,  one  uses Property $(P3)$ and $(P4)$ in \ref{sec:SV} to get
$$
 V_1(x) = \zeta (x) \exp \left[ \int_x^{1} \frac {\epsilon (t)}{t} dt\right] , \qquad \hbox {with} \quad \zeta  \in {\cal C}^1, \ \ \epsilon \in {\cal C}^0, \qquad \lim_{t\to 0}\epsilon(t)=0.  $$
Then, the estimates hold 
 \begin{equation} \label{ratioV11}  \log  \frac {V_1(q(n))}{V_1(q(n+1))}  =  \log \frac {\zeta (q(n))}  {\zeta (q(n+1))}+  o\left( \log  \frac {q(n)} {q(n +1)}\right)\, .
 \end{equation}
 The estimate  
 \vskip 0.1 cm 
 \centerline{$ \displaystyle \log  \tfrac {q(n)} {q(n+1)}  =-  \log (1- v(q(n) ) \sim  v(q(n)) \, , $ \qquad \quad  }
  \vskip 0.1 cm 
 entails that the second term in \eqref{ratioV11}  is $o(v(q(n))$.   Moreover, 
\vskip 0.1 cm 
  \centerline{ as $\zeta \in {\cal C}^1$, with $\zeta(x) = C + O(|x|),  \ \ \ {C \not = 0}$, }
    the  first  term  in \eqref{ratioV11}  satisfies, 
$$ \log \tfrac {\zeta (q(n))}  {\zeta (q(n+1))} =  O( \zeta (q(n+1)) - \zeta (q(n)) = O (q(n+1)- q(n) )  = O(q(n) v(q(n))) \, $$
 and it is of smaller  order than the second term. This entails an  estimate  $o(v(q(n))$ for \eqref{ratioV11} and thus  an estimate for \eqref{ratioV1} as in Lemma \ref{AarLemma}. 
 \end{proof}

   Using this last estimate in \eqref{diff} leads to 
   $$  \frac 1 {v(q(n+1))}   -   \frac 1 {v(q(n))} =  \frac 1 {v(q(n))}    \left[  \frac 1 { (1- v(q(n)))^\gamma}  \cdot \frac {V_1(q(n))}{V_1(q(n+1))} -1\right]  $$
   $$  =  \frac 1 {v(q(n))}  \left( \left[  1+ \gamma v(q(n)   + o(v(q(n)^2 )\right]  \left[ 1 +o(v(q(n)) \right]  -1\right) = \gamma  + o(1) \, .$$
   We take the sum of the previous estimates, 
   and obtain
  $$  \frac 1 {v(q(n))}    = \gamma n + o(n)\, , \quad (n \to \infty), \quad v(q(n))  \sim \frac 1 {\gamma n} \quad   \gamma_2 \frac 1 n\le  v(q(n))  \le \gamma_1 \frac 1 n, $$
  with $\gamma_1$ and $\gamma_2$ arbitrary close to $1/\gamma$.

$(ii)$ The previous estimate  given in $(i)$ leads to 
\begin{equation} \label{deb}
q(n)^\gamma V_1(q(n)) \sim \frac 1 {\gamma n} \qquad   q(n) \sim \left(\frac 1{\gamma n}\right)^{1/\gamma} \left( \frac 1 {V_1(q(n))}\right)^{1/ \gamma}\, .
\end{equation}
As  $V_1$ belongs to $\cal {SV}_0$,  it is the same for $1/V_1$, and $1/V_1(q(n))$ satisfies,  for any $\epsilon >0$, for $N$ larger than $N_\epsilon$,  
$$  q(n)^{\varepsilon} \le \frac 1 {V_1(q(n))} \le q(n)^{-\varepsilon}\, .$$ 
This entails, with \eqref{deb},  the two  inequalities, as soon as $N$ is larger than $N_\epsilon$,  
$$  {\tilde{A}}_\epsilon \left( \frac 1 n \right) ^{1/(\gamma-\epsilon)} \le   q(n) \le A_\epsilon \left( \frac 1 n \right) ^{1/(\gamma+\epsilon)}\, . $$ 
When $\gamma>1$, we choose $\epsilon$ such that $\epsilon<\gamma-1$. This entails  $1/(\gamma-\epsilon)<1$ and the divergence of  $\E[W]$ follows. 
When $\gamma<1$, we choose $\epsilon$ such that $ \gamma < \widehat \gamma = \gamma + \epsilon  <1$. The convergence of $\E[W]$ follows.

$(iii)$. Clear 

  \smallskip $(iv)$   We choose $\epsilon = 1/2$. This leads, with Items $(i)$ and $(ii)$,  for $n \ge N$,  and for some constant $A$,   to  the estimate 
  $$ q(n) -q(n+1)   =  q(n)   \Big( 1- (1-v(q(n)) \Big)\le A \left(\frac 1 {n}\right) ^{1+ 2/(2\gamma+1 )}\, .$$ 
   In the same vein, for some constant $B$ 
   $$v(q(n)) -v(q(n+1)) = 
   v(q(n))  \left( 1- \frac {q(n+1)^\gamma}{q(n)^\gamma} \frac {V_1(q(n+1))}{V_1 (q(n))}\right) $$
   $$ = v(q(n)) \left( 1 - (1-v(q(n)))^\gamma   \frac {V_1(q(n+1))}{V_1 (q(n))}\right) \, .$$
   Using Lemma \ref{AarLemma}  and the same principles as in the proof of Assertion $(i)$, one obtains 
   $$ v(q(n)) -v(q(n+1))  \sim  \gamma v(q(n))^2 \sim (1/\gamma) (1/n^2)\, .$$ 
   
     \smallskip $(v)$ Clear.
  
 \end{proof}

 %%%%%%%%%%%%%%%%%%%

 \subsection{Proof of  the inclusion  $\cal{DRIS} \subset \cal{DRBGC} $. } \label{DRGC}
  This Section is devoted to the proof of  Theorem \ref{deltaa}.

 \begin{proof}
 Property $(a)$ of the Good Class  holds from  Item $(iv)$ of Theorem \ref{qneval}.

\medskip  
{\sl  Property $(b)$ of the Good Class:  strict contraction.}   This  property holds for  the whole Class $\cal{DR}$-- in a strong form--  and thus, in particular,  for the  class $\cal{DRI}$.

\begin{lemma} \label{contrac}  For a source of the  Class $\cal{DRI}$,   the sequence $\eta_n$ of Definition \ref{good_class} $(b)$ satisfies the following\footnote{The Classes  $ \cal{DRI}_s$ and $, \cal {DRI}_w$ are defined in Section \ref{waitbis}.}:

\smallskip$(a)$ 
       For a source of  the Class %$\cal {DR}_{ss}  \cup 
       $ \cal{DRI}_s$, one has $\eta_1 <1$, whereas for a source of  the Class $\cal {DRI}_w$, one has $\eta_1 = 1$ and $\eta_2 <1$.  
        
       \smallskip $(b)$
        For any source in the Class $\cal{DRI}$,   the inequalities hold  for $n \ge 2$, 
        \vskip 0.1 cm 
        \centerline{   $  \eta_n \le \underline {\eta}_2 ^{ 1+ {\lfloor { (n-2) /2 }\rfloor} } \quad \hbox{with} \quad  \underline {\eta}_2 = \min (\eta_1^2, \eta_2)  <1$ \, .}
       
       \end{lemma}
       
       \begin{proof}\ \ 
        
        Item $(a)$.  The derivative $g'_m$  of the   branch $g_m = a^{m-1}\circ b$ of the source $\cal B$ satisfies 
  \begin{equation} \label{der}  |g'_m(x)| = |b'(x)|   \ \  (m = 1), \quad  |g'_m(x)| = |b'(x)|   
  \prod_{\ell = 1}^{m-1}  |a'( a^{\ell-1} \circ b(x))| \ \ (m \ge 2) \, .
  \end{equation}
  With the intervals  ${\cal J}_\ell$ defined in  Proposition \ref{recDR0}$(i)$  for $\ell  \ge 1$,  and   $\cal {J}_0 = \cal I$,   we define
    \vskip 0.1 cm
    \centerline{ $  \alpha_m = \max \{ a'(x) \mid x \in {\cal J}_m\}, \ \ (m \ge 1),  \quad  \beta_m =  \max \{ |b'(x)| \mid x \in {\cal J}_m\}, \ \ (m\ge 0)\, .  $}   \vskip 0.1 cm 
     For $m \ge 1$, the inequalities $\alpha_m   <1$ and $\beta_m <1$ lead to the bounds
     \vskip 0.1 cm 
     \centerline{$\gamma_m  =  \prod_{ \ell = 1}^{m-1} \alpha_\ell < \alpha_1$, \quad $\sup_{m \ge 1}  \beta_m \le 1$.  }

\medskip
  In the case of  a system in $\cal {DRI}_s$, the first factor  in \eqref{der}  is at most $\beta_0 <1$  whereas  the second factor is   at most  equal to 
  $\gamma_m $ and thus at most $\alpha_1 \le 1$.
 
 \medskip  
  In the case of  a system in $\cal {DRI}_w$,   we consider  the branches of depth 2, of the form  $g_m\circ g_n$ associated with an integer pair $(m, n)$, let  
  \vskip 0.1 cm 
  \centerline{ $A(m, n) := \sup\{ |(g_m\circ g_n)'(x)| \mid x \in [0, 1] \}, \quad m \ge 1, n \ge 1 $}   and  consider four  cases:   
\vskip 0.1 cm  
 \centerline{ $(m = n = 1), \ \ (m\ge 2, \  n= 1),  \ \   (n \ge 2, \ m = 1),  \ \ (m \ge 2,\  n\ge 2)$.}
 \vskip 0.1 cm 
 Then, there is a  bound for $A(m, n)$ in each of the four cases, as follows: 
 \vskip -0.3cm $$
  \begin{array}{llll}(m = 1,\  n= 1): & \beta_1  & (m \ge 2,\  n = 1):  & \alpha_1    \cr
   (m = 1, \ n \ge 2):  & \beta_n\,  \alpha_1  & (m \ge 2, \ n \ge 2):  &  \alpha_1^2\, .
\cr 
\end{array} 
$$
 This entails 
 \vskip 0.1 cm \centerline{
$ \eta_2  = \sup \{A (m, n) \mid m, n \ge 1\}  \le  \sup (\beta_1,  \alpha_1,  \alpha_1 \sup_{n\ge 2}  \beta_n, \alpha_1^2 ) \le \max (\beta_1, \alpha_1) <1\, .$} 

\medskip{} Item $(b)$. Clear with multiplicativity of derivatives. 
 
\end{proof}

  \smallskip   
 {\sl Property $(c)$ of the Good Class :   distortion. }
   
  \smallskip
   For each branch $a, b$ we denote by $\delta(a), \delta (b)$ the finite bounds, 
   \begin{equation} \label{bo}
   \delta(a) = \max_{(x, y)\in {\cal I}^2 }\frac{|a'(x)|}{|a'(y)|}, \qquad \delta(b) = \max_{(x, y)\in {\cal I}^2} \frac{|b'(x)|}{|b'(y)|} \, .
   \end{equation}
   We also remark that the distortion ratio is  expressed as  a derivative
  \begin{equation} \label{distorh}  
  \frac {|h''(x)|}{|h'(x)|} = \frac {d}{dx}\Big(\log |h'(x)|  \Big)\quad \hbox{and thus}   \quad \left| \log \frac { |h'(x)|} {  |h'(y)|} \right| = \left|\int_x^y  \frac {|h''(u)|}{|h'(u)|} du\right|  \, .\end{equation}
 With \eqref{der},   the  derivative $g'_m$ of the  inverse  branch $g_m = a^{m-1}\circ b$  satisfies, 
    \begin{equation} \label{distor} \quad   \log  \frac { |g'_m(x)| }{|g'_m(y)|} =  \log \frac{|b'(x)|} {|b'(y)|} + \sum_{\ell = 0}^{m-2} A_ \ell, \quad A_\ell =  \log \frac  { |a'( a^{\ell} \circ b(x))|  } { |a'( a^{\ell} \circ b(y))|  } \, .
  \end{equation} 
  Each  term $A_\ell$   involves  two points 
  that belong to
  
  \vskip 0.1 cm 
  \centerline{ ${\cal J}_{\ell+1} =  [a^\ell  \circ b(1)  , a^\ell  \circ b(0)]= [q(\ell+1), q(\ell)]$. } 

\medskip
Consider the point $x_0$  in  Lemma \ref{DRgamma1} (item $(ii)$). As each interval ${\cal J}_\ell $  satisfies ${\cal J}_{\ell +1}\subset  [0, q(\ell)]$ and since $q(\ell) \to 0$,     there is an integer  $M_1$  for which, for any $ \ell \ge  M_1$, one has $q(\ell)   \le x_0$.  Then,  
   we  may apply  \eqref{distorh} to any pair of points of the interval $\cal {J}_\ell$.  Moreover $a'(0) = 1$  and $a'(x) >0$ on $[0, x_0]$.    This   entails the estimate for $\ell  \ge M_1$, 
  \begin{equation} \label  {qltheta}|A_\ell| 
   \le K \left| \int_{ a^{\ell} \circ b(x))}^{a^{\ell} \circ b(y))}\, a''(u) du\right| \le  K |v(q(\ell) -v(q(\ell +1)) |\, .
   \end{equation} 
         We apply   Item $(iii)$ of Theorem \ref{qneval},  
    and,    choosing  $M = \max (M_1, N)$, 
    the sum of estimates in \eqref{qltheta} for $\ell \ge M$  is of order  
  $ \Theta [M ^{-1}]$.

 \smallskip 
    For $\ell \le M$,  using the bound \eqref{bo},   the sum for indices $\ell  \le  M$ in \eqref{distor} is less  than 
    \vskip 0.1 cm
    \centerline{$M  \log \delta (a) + \log \delta(b)$.}

  This entails,  for any $(x, y) \in {\cal I}^2 $ and for any  integer $m$,  the final distortion bound, 
     $$  
\log    \frac {|g'_m(x)|}{|g'_m(y)| } \le  \left( \log\delta(b) + M \log\delta(a) \right) +\tfrac K M  $$  that leads to the result.  
  
   \end{proof}

%%%%%%%%%%%%%%%%

 %%%%%%%%%%%%%%%%%%%%%%%%%%%
 
 \section{The Class $\cal{DRIL}$.}\label{sec:DRIL}

We  introduce  in   Section  \ref{DRIL1} 
 the subclass $\cal{DRIL}$, where the letter $\cal{L}$  is relative to a logarithm behaviour of the  inverse branch $a$.  Proposition  \ref{Aar1}
of   Section \ref{wtdprec} provides  precise estimates for the wtd sequence.  Then, Section \ref{exh2} answers the main question of the paper and explains  how  to   find a  source of the Class $ \cal{DRIL}$  with a prescribed Shannon weight and Section \ref{remarks}  made important remarks on this result. The Section ends in Section \ref{exDRIL} with  the description of some explicit instances of the subclass $\cal{DRIL}$, with their graphs.

    \subsection{Definition of the Class $\cal{DRIL}$.}\label{DRIL1} 
We  consider  a subclass of the class $\cal{DRI}$ 
 that gathers functions $u$ in \eqref{uu} whose derivative $u'$  is of logarithmic type. This explains the ``${\cal L}$'' at the end of the name of the subclass.

  \begin{definition}  The class $\cal {DRIL}(\gamma, \delta)$ is the  subset of the class $\cal{DRI}$ which gathers  the functions $u \in \cal C^2$  whose derivative  $u'$ satisfies $u'(x) \in [0, 1]$ and is written,  instead of \eqref{uu}, as $u'(x) =  x^\gamma V_\delta (x)$, with 
  \begin{equation} \label{uuu} 
 V_\delta (x) = A  |\log \tfrac  x 2| ^\delta, \quad (\delta \not = 0), \quad V_0(x) \in {\cal C}^1,\ \ V_0(x) >0 \quad (\delta = 0) \, , \end{equation}
  with some constant $A$, and some function $V_0$. 
    \end{definition}
  
   The following  inclusions hold,  
     $$ \hbox{$(\gamma = 1, \delta \le 0)$ } \ \   \cal {DRIL}(1, \delta) \subset \cal {DRIS} (1), \quad 
   \hbox{$(\gamma > 1, \delta \in \mathbb R)$}  \ \   \cal {DRIL}(\gamma, \delta)  \subset \cal {DRIS} (\gamma)\, ,  $$
and    lead us to  the set  $\Gamma_S$ for  possible parameters  $(\gamma, \delta)$ of the source $\cal S$, together with the subset $\cal{DRIL}$, 
   \begin{equation} \label{GammaS}
  \Gamma_S =  \big( ] 1, + \infty[ \times {\mathbb  R}\big)  \bigcup\big(  \{1\} \times ]-\infty, 0]  \big)\, , 
 \qquad  \cal{DRIL} =  \bigcup_{(\gamma, \delta) \in \Gamma_S}  \cal{DRIL} (\gamma,  \delta)\, .
 \end{equation}

  \begin{lemma}  \label{precgammadelta} The following holds
  \begin{itemize} 
  \item[$(i)$] 
   For any  
 $(\gamma, \delta)  \in  \Gamma_S$,    there are choices for the  constant $A$, or the function $V_0$, which entail that 
 \vskip 0.1 cm 
 \centerline{ $u'(x) = x^\gamma  V_\delta(x)= A x^\gamma  |\log \tfrac  x 2| ^\delta  \in [0, 1]$ for $ x \in {\cal I}$.}
For $\delta \not  = 0$, any 
  $A<A_{\gamma, \delta}^{-1}$,  with  $A_{\gamma, \delta}$ defined in \eqref{Agammadelta}, is a valid choice.
  \item[$(ii)$] For any $\delta\not = 0$, the functions $u$ and $v$ satisfy $$   \frac {u(x)} {x} = v(x) =  \tfrac { A}{\gamma +1} \,  {x^\gamma}  
   |\log \tfrac  x 2| ^\delta \, \left[ 1 + O\left( \tfrac 1 {|\log x|}\right)\right] \, .$$
    For (strictly)  positive  integer values of $\delta$, 
 the two functions $u(x)$,  and $v(x) = u(x)/x$  are  written as the product  of  the power function $x^\gamma$ by a polynomial  $P_\delta $ in  the variable $| \log  \frac x 2|$ of degree at most $\delta$.  
  \end{itemize}
 \end{lemma}

 \begin{proof} 
  First remark that  $u'$ is  well defined on $[0, 1]$  for any $\delta$.  
  
\smallskip
 $(i)$  When $\delta \not = 0$, we  make precise  the conditions on the constant $A$.    The derivatives  satisfy, 
    $$ V'_\delta(x) =    \delta \left( \tfrac {-1} {x}\right) (-\log \tfrac x 2)^{\delta -1}  = -\tfrac {\delta} {x}  (-\log \tfrac x 2)^{\delta -1}, \quad \frac {x V_\delta'(x)}{V_\delta(x)} = -  \delta { |\log \tfrac x 2|}^{-1}=  \tfrac { \delta}  {\log \tfrac x 2}\, ,  $$
\begin{equation} \label{epsilon1}  \hbox{and thus} \quad u''(x) =  x^{\gamma-1} V_\delta(x) \left[\gamma  +  \epsilon(x)\right], \qquad \hbox{with} \quad \epsilon(x) = \frac {\delta}{\log  \tfrac x 2} \, .
\end{equation}
For $\gamma = 1$, we ask  a finite limit for $V_\delta$ as $x \to 0$, and thus $\delta \le 0$. 
 For any value of $\gamma \ge 1$, there is a zero  for $u''$ at  the point $x_0$ for which 
\vskip 0.1 cm 
\centerline{ $ \log \tfrac {x_0} 2 = -\delta/\gamma$, i.e., 
$x_0 = 2 e^{-\delta/\gamma}$.} 
\vskip 0.1 cm
When $x_0 \le 1$ 
the function $u'$ has a maximum at $x_0$, 
 whereas, for $x_0 \ge 1$,  the maximal value  of $u'$   is attained at $x= 1$.   The condition $x_0 \le 1$ is 
 \vskip 0.1 cm
 \centerline{ $    e^{-\delta/ \gamma} \le 1/2$ or, equivalently,  $\delta \ge (\log 2)\,  \gamma > 0\, .$}
  \vskip 0.2 cm 
 In  all the cases with $\delta \not = 0$,  the maximal value  of $u'$ is $A\times A_{ \gamma, \delta}$, where
\begin{equation} \label {Agammadelta}  A_{ \gamma, \delta} =  2^\gamma e^{-\delta}  (\delta /\gamma)^\delta\,, \quad \delta \geq ( \log 2)\,  \gamma; \qquad A_{ \gamma, \delta} = (\log 2)^\delta\,, \quad \hbox {otherwise}   \, . 
\end{equation}
 For $\delta \not = 0$, we   are thus led to  deal with   functions $u_{ \gamma, \delta} $  
 associated with a constant $A$  which satisfies 
    $A<A_{\gamma, \delta}^{-1}$  and $A_{\gamma, \delta}$ defined in \eqref{Agammadelta}.
  
  \medskip 
  For $\delta = 0$, we ask  the function 
   $u'_{\gamma, 0} =   x^\gamma V_0(x)$ to belong to the interval $[0, 1]$.  
   
   \medskip 
   $(ii)$  Due to \eqref{epsilon1}, this is Item $(c)$ of Lemma  \ref{DRgamma1}  in the case when $\epsilon(x) =  {\delta}(\log  \tfrac x 2)^{-1} $ and thus $\epsilon_1(x) = \epsilon (x)$.  The explicit expressions are  obtained with  successive integrations by parts. 
   \end{proof}

    %%%%%%%%%%%%%%%%%%%%%%%%%%%%%%%%%%%%%%%%%%%%%%%%%%
   
   \subsection{Precise estimates of the wtd sequence}    \label{wtdprec}
   
 For  a system of the   subclass $\cal{DRI}$, the block source    $\cal B$ belongs to the Good Class. However, there are no precise results on  the asymptotic behaviour of its wtd sequence. We only know  the asymptotic behaviour of the sequence $v(q(n))$ from Theorem \ref
  {qneval}$(i)$.  This is why we introduce the subclass $\cal{DRIL}$ where it is possible to obtain precise results on the asymptotic behaviour of the  wtd sequence. 
This is the main interest of the  $\cal{DRIL}$ Class.  
  
     \begin{proposition} \label{Aar1}  For  a source $\cal S$ of the $\cal {DRIL} (\gamma, \delta)$ Subclass, the    wtd $q(n)$  satisfies   $$ q(n) \sim \gamma ^{(\delta-1)/\gamma}\, 
 n^{- 1/\gamma}  \, (\log n)^{-\delta/\gamma} \ \ (\delta \not = 0), \quad   q(n) \sim B n^{-1/\gamma}  \ \ (\delta  = 0) \, . $$ 
 \end{proposition} 
 
       \begin{proof} We start with the general estimate for $v(q(n))$ obtained in Item  $(i)$ of Theorem \ref{qneval}. We want to show that we can invert this relation asymptotically, yielding the asymptotics for $q(n)$.
       
       For $\delta = 0$, the estimate is direct. When $\delta \not = 0$, one has        
       \vskip 0.1 cm
       \centerline{$ v(x) = x^\gamma V_1 (x), \quad  V_1(x) =  A |\log \tfrac x 2|^\delta\left[1+  O\left( \frac 1 {|\log x|}\right)\right]\,. $} 
       \vskip 0.1 cm 
  Consider pairs $(x, y)$  for which $  v(y)\sim x$ with $x,y\to 0$. Such pairs $(x_n, y_n)$ exist, 
\begin{equation} \label{est-i} y_n=q(n),\quad x_n= 1/(\gamma n)
\end{equation} Notice that: 
   $$x \sim v(y)=  A\,  y^\gamma  |\log \tfrac y 2|^{\delta} (1+ O( \tfrac 1{|\log y|})), \qquad x^{1/\gamma} \sim A^{1/\gamma}  y \, |\log  \tfrac y 2|  ^{\delta/\gamma} $$ 
      $$
    \qquad \tfrac 1 \gamma \log x =  \log y  -  \tfrac {\delta}  {\gamma} \log |\log\tfrac  y 2|  + o(1)  =   \log y\,   (1+ \epsilon(y))   $$
  with   $\epsilon (y) \to 0$   as $y \to  0$,  so that   the estimates hold,  (as $x$, $y$ tend to $0$)$$ \log \tfrac y 2 \sim \log y \sim \tfrac 1 \gamma \log x,  \qquad  y  \left|\tfrac 1 \gamma \log x\right|^{\delta/ \gamma} \sim   A^{-1/\gamma}\,  x^{1/\gamma},  $$
  $$ \hbox{and finally}  \quad   y \sim  A^{-1/ \gamma}\   \left( \tfrac 1 \gamma\right)^{-\delta/\gamma} x^{1/\gamma}  |\log x|^{-\delta/\gamma} \, .$$
 Applying to the pairs $(x_n, y_n)$ given in \eqref{est-i} leads  to the estimate  of Proposition \ref{Aar1} in the case $\delta \not = 0$.  
      \end{proof}

 %%%%%%%%%%%%%%%%%%%%%%%%%%%%%%%%%%%%%%%%%%%%%%
 
    \subsection {Exhibiting a $\cal{DRIL}$ source with a prescribed Shannon weight. } \label{exh2}

   \medskip 
       We  associate with  a  source $\cal S$ of the subset $\cal{DRIL}$ the set 
    $\Gamma_S$  of parameters $(\gamma, \delta)$ defined in \eqref{GammaS}    together with  the pair $(\beta_Q,  \delta_Q) $ of parameters that describe the wtd $q(n)$ of such a source $\cal{S}$.  Proposition  \ref{Aar1} exhibits an injective map     defined on $\Gamma_S$ 
      $$   (S \rightarrow Q) : \qquad  (\gamma_S, \delta_S) \mapsto (\beta_Q, \delta_Q), \quad \beta_Q:= 1/ \gamma_S, \quad \delta_Q= -\delta_S /\gamma_S \, ,$$
   whose image is 
\vskip 0.1 cm
\centerline{     $
 \Gamma_Q^\star = \left( ]0, 1[ \times {\mathbb R} \right) \bigcup \left( \{1\} \times [0, \infty[\right) \, . $}
 
 \smallskip
The set $\Gamma_Q^\star$  is  (strictly) contained in the set $\Gamma_Q$ defined in \eqref{condSR}, recalled here
\vskip 0.1 cm
\centerline{     $
 \Gamma_Q = \left( ]0, 1[ \times {\mathbb R} \right) \bigcup \left( \{1\} \times [-1, \infty[\right)   \bigcup \left( \{0\} \times [-\infty,  0[\right) \, . $}
\vskip 0.1 cm
In other terms, there does not exist a source of the Class $\cal{DRIL}$  whose  wtd parameters  $(\beta, \delta)$   satisfy  
\vskip 0.1 cm 
\centerline { $ \big(\beta  = 1, \, \delta  \in  [-1, 0[ \big) $ \quad or  \quad  $\beta = 0$\, .}
\vskip 0.1 cm

\medskip Theorem \ref{final1} applies  to  sources $\cal S$ in the Class $\cal{DRIL}$, and   the  injective map  described in \eqref{QM}
$$ (Q \rightarrow M): \qquad    (\beta_Q, \delta_Q)\mapsto (\beta_M, \delta_M)$$
 is defined  on  $\Gamma_Q^\star$.  Finally,   for a   source $\cal S$ of the subclass  $\cal{DRIL}$,  there is an injective map  $(S \rightarrow M)$ defined  on $\Gamma_S$  equal to  the composition of the two maps  $(S \rightarrow Q)$ and  $(Q \rightarrow M)$.
 One has
$$  (S \rightarrow M) : (\gamma_S, \delta_S)\mapsto (\beta_M, \delta_M), \quad \hbox{with} $$
$$   \beta_M = 1/ \gamma_S,  \qquad \delta_M =- \delta_S/\gamma_S  \quad (\gamma_S > 1), \quad \hbox{or} \quad \delta_M = - (\delta_S/\gamma_S) -1 \quad  (\gamma_S = 1) $$
and the image   $(S \rightarrow M) (\Gamma_S)$ equals   $ \Gamma_M^\star =  \left( ]0, 1[ \times {\mathbb R} \right)   \bigcup \left( \{1\} \times [-1, \infty[\right)$.

\smallskip
The inverse map is  defined on $\Gamma_M^\star$ and satisfies: 
$$  (M \rightarrow S)  :  (\beta_M, \delta_M) \mapsto (\gamma_S, \delta_S) \quad \hbox{with} $$
\begin{equation} \label{SM}  \gamma_S = 1/\beta_M, \quad \delta_S = - \delta_M /\beta_M   \quad (\beta_M < 1) \quad \hbox{or} \quad \delta_S= -( \delta_M +1)/\beta_M   \quad (\beta_M = 1) \, .
\end{equation}

 We have  thus proven the following  result, which provides the final answer to our ``exhibition'':   

\begin{theorem}  \label{exh1}
 For any  pair
 \vskip 0.1 cm 
 \centerline{ $(\beta_M, \delta_M) \in \left( ]0, 1[ \times {\mathbb R} \right)   \bigcup \left( \{1\} \times [-1, \infty[\right)$,}
 \vskip 0.1 cm  there exists  a source $S$ in the subclass $\cal {DRIL} (\gamma, \delta)$ with   the pair

\vskip 0.1 cm 
\centerline{  $(\gamma, \delta)  \in  \big( ] 1, + \infty[ \times {\mathbb  R}\big)  \bigcup\big(  \{1\} \times ]-\infty, 0]  \big) $}
\vskip 0.1 cm 
  which admits Shannon weights $m(n) $ of the form $m(n)  = \Theta (n^{\beta_M} (\log n)^{\delta_M})$. 
  
  \smallskip
  The pair $(\gamma, \delta)$  is related to the pair $(\beta_M, \delta_M)$ via  Equation \eqref{SM}
  \end{theorem} 
  
  \subsection{Remarks on the  final result} \label{remarks} There are three remarks :  
 
 \smallskip 
$(a)$  The limit case $(\beta_M, \delta_M) = (1, -1)$ is obtained with sources of the subclass $\cal{DRIL}(1, 0)$, and for instance with the Farey source.

\smallskip  
$(b)$    The previous result  does not provide a  $\cal{DRIL}$ source    whose parameters $(\beta_M, \delta_M)$  satisfy  
\begin{equation} \label{othercases} ( \beta_M = 1,  \delta_M \in ]-1, 0[) \quad \hbox{or} \quad 
(\beta_M = 0, \delta_M >0)\, . 
\end{equation}
The first   case would correspond to Shannon weights very close to a linear behaviour $ \Theta (n (\log n) ^{\delta_M})$ with $\delta_M$ negative  near  0.  
%The $\cal{DRIL}$ sources   with  $\beta_Q = 1$ have  indeed their parameters $\delta_Q  >0$ as  Proposition \ref{Aar1} shows. 

\medskip
However, in another paper \cite{VLMC}, the authors study the model   called Variable Length Memory Chains  (VLMC) and exhibit instances of sources   whose parameters satisfy \eqref{othercases}. 
These results are not developed here.

\smallskip
$(c)$  We give some hints on the constant   that intervenes  in  Theorem \ref{exh1}. 
When  the system  belongs to $\cal{DRIL}( \gamma, \delta)$, then, from  Proposition \ref{Aar1},   the sequence $q_\tau(n)$ has a constant $K_\tau$ 
that is expressed in terms of the pair $(\beta = \beta_M, \delta = \delta_M)$, 
 $$ K_\tau =  \beta^{\beta + \delta+1}   \quad (\beta = 1) \quad  K_\tau =  \beta^{\beta + \delta}   \quad (\beta < 1) \, , $$ 
 and, from Theorem \ref{final1},  the precise asymptotic  for $\underline m_\mu(n)$ holds in terms of the pair $(\beta = \beta_M, \delta = \delta_M)$ 
 \begin{equation*} 
  \underline m_\tau(n) %\sim  \underline m_\tau (n)  
  \sim \left \{\begin{array} {llll} \displaystyle 
  \frac {{\cal E}_\nu(\cal  B)}{\psi(0)}  &  \displaystyle 
    (\delta +1)\,  \beta^{-(\beta + \delta+1)}  & \displaystyle \frac { n^{\beta}}{(\log n)^{\delta +1}} \    &(\beta = 1) \cr
 \displaystyle    \frac{{\cal E}_\nu(\cal B)}{\psi (0)}&     \ 
\displaystyle  \frac  {\beta^{-(\beta + \delta)} } {\Gamma (\beta +1 ) \Gamma (1-\beta)} \, &\displaystyle  \frac{n^{\beta} }{(\log n)^{\delta }}  &  (0\le  \beta < 1)  \cr 
  \end{array} \right\}.
  \end{equation*}
  There are two kinds of constants. the second group of constants (in the second column)  are absolute constants, that only depend on the pair $(\beta, \delta)$. But there is  also  (in the first column) the quotient $ {\cal E}_\nu(\cal  B)/\psi(0)$ that involves the block dynamical system itself. There is thus an important question: how  does  the block dynamical system  of a system of the  $\cal{DRIL} (\gamma, \delta)$ Class  evolve as a function of the pair $(\gamma, \delta)$ --namely,  the quotient of its entropy with  its invariant density--?

      \subsection{Some instances of the $\cal{DRIL}$ Class.} \label{exDRIL}
       \paragraph {\sl Our emblematic instance: the Farey source}
       The  inverse branch $a$ of the Farey map is $x/(1+x) $ and its derivative  equals $a'(x) =(1+x)^{-2}$. Then the fonction 
$u'(x)= 1- (1+x)^{-2} $ satisfies $u'(x) \in [0, 1]$. Moreover,   one has

\centerline{ 
$u'(x) = x \cdot V_0 (x) \quad \hbox{with} \quad V_0(x) =  \displaystyle  \frac {2+x}{(1+x)^2}$, } 
\vskip 0.1 cm
and the Farey map belongs to the Class $\cal {DRIL}(1, 0) \cap \cal {DR}_w$.

 \medskip     
 \paragraph{\sl Other instances}      We also describe  in Figure  \ref{fig:dri}  other  instances of the class $\cal{DRIL}(\gamma, \delta)$ with $(\gamma, \delta) \in \Gamma_{\cal S}$.  We first remark that the branch $b$  has no influence  on the pair $(\beta_M, \delta_M)$ which defines the asymptotic behaviour of Shannon weights.   The branch $b$ just intervenes     via  its value $b'(0)$, which may be equal to  $-1$ --case $\cal{DRI}_w$-- or strictly larger than $-1$ --Case $\cal{DRI}_s$-- (See Section \ref{waitbis}).   As Lemma \ref {precgammadelta}$(ii)$ suggests, we limit ourselves to  (small) positive  integer values of $\delta$, for instance, $\delta = 1$ or $2$  where the computations become simpler, with 
 \vskip 0.1 cm 
 \centerline{ $V_1(x) =  A|\log\frac x2|$ and $V_2(x) =  A(\log\frac x2)^2$, }
 \vskip 0.1 cm
  for  a precise choice of constants $A$, which  depends on the pair $(\gamma, \delta)$ and   is described in  Lemma \ref{precgammadelta}$(i)$. 
      
   \smallskip
    Figure   \ref{fig:dri}  deals with five explicit  pairs $(\beta, \delta)$  for which we indicate the corresponding  Shannon pairs $(\beta_M, \delta_M)$. 
   The first two pairs are relative to 
   the Class $(\cal{DRIL})_s$, and  the branch $b$  is the linear complete map.  
   The last three pairs  are relative to the Class $(\cal{DRIL})_w$, and we then consider  branches $b$ 
with $b'(0) = -1$.  The final case corresponds to  the Farey source.  

 \begin{figure}[h]
    \centering
    \begin{tabular}{ccc}
    \includegraphics[scale=0.40]{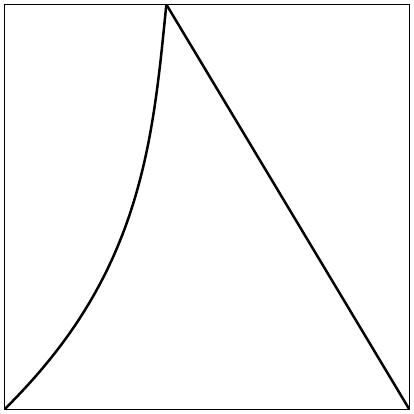} &  \includegraphics[scale=0.40]{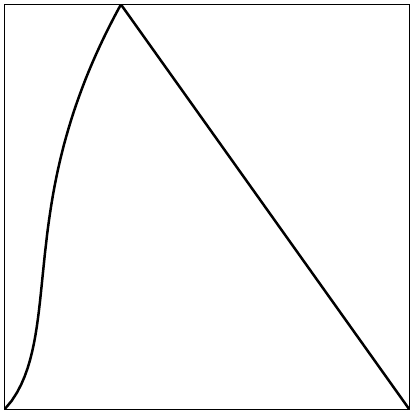} &
  \includegraphics[scale=0.40]{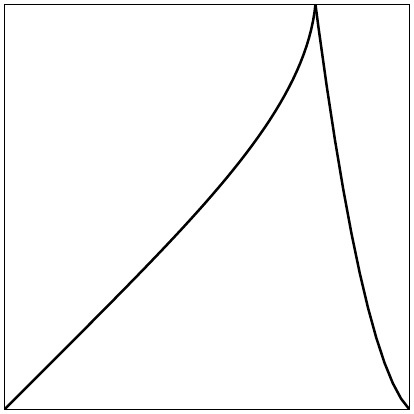}\\
  $\left(\tfrac {14}{10},1\right)$ $\left(\tfrac 5 7, -\tfrac 5 7\right)$&
  $\left(\tfrac{11}{10},2\right)$ $\left(\tfrac{10}{11},-\tfrac{20}{11}\right)$&
  $(4,1)$ $\left(\tfrac{1}{4},-\tfrac{1}{4}\right)$\\
    \end{tabular}
    \vskip 10pt
    \begin{tabular}{cc}
    \includegraphics[scale=0.40]{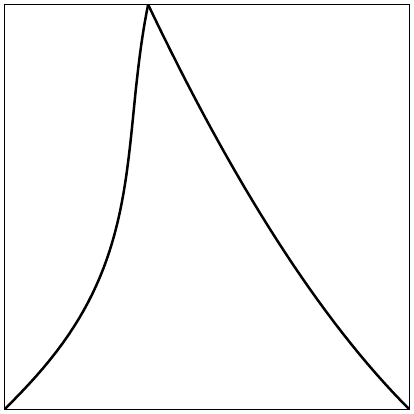}&
         \includegraphics[scale=0.40]{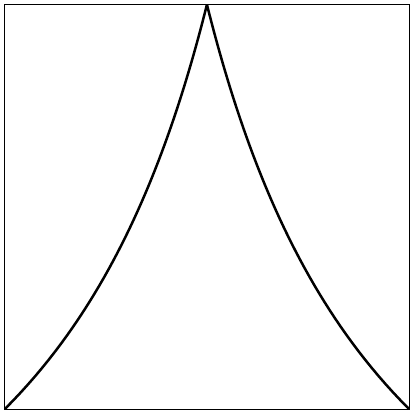}\\
    $(2,2)$ $\left(\tfrac{1}{2},-1\right)$&$(1,0)$ $(1,-1)$\\
    \end{tabular}
    \caption{Five pairs $(\gamma, \delta)$,  their corresponding Shannon pairs $(\beta_M, \delta_M)$, and their corresponding graphs}
    \label{fig:dri}
\end{figure}

  \subsection{Application to  the Farey source.}  \label{Fareycase3}
 
 The Farey source is described in Section \ref{Fareycase}; it is  proven there that the Farey source  belongs to the subclass  $\cal {DR}_w$  and  its block source  is defined by the Gauss map.  It is  proven  in Section \ref {exDRIL} that  the Farey source belongs to the subclass $\cal{DRIL}(1, 0)$.  Thus,  its block source (the Gauss source)
  belongs to the Good Class.  

Moreover,  the  main spectral objects of the Gauss map are explicit  : \\
 --   its  stationary density 
 is $\psi(x) =  1/ (\log 2) \,    1/(1+x)$ \\
  --   its  entropy  
    equals $\pi^2 /(6 \log 2)  $
  
  \medskip
 Then  we may apply our  three main theorems,  and  obtain the  final estimate for the Shannon weights: 

\begin{theorem} 
Denote by $\mu$ a mesure with a strictly positive density $\phi $ of class ${\cal C}^1$.   The Shannon weights  $\underline m_\mu (n)$ of the Farey source satisfy 
$$\underline m_\mu(n) \sim  \phi(0)  \frac {\pi^2}{6}  \frac {n} {\log n}, \qquad (n\to \infty) \, .$$
\end{theorem}

As we explain in the next subsection,  and to the best of our knowledge, this result is new.

%%%%%%%%%%%%%%%%%%%%%%%%%%

  \subsection{Return to  Shannon weights} \label{SW}

 The notion of coincidence intervals is central in  the context of dynamical systems, here defined on the unit interval $\cal I$. 
We first recall this notion: 
for any $x \in \cal I$ and any depth $k$, the interval 
${\cal I}_k(x)$ gathers all the  reals $y \in \cal I$  that belong to the same fundamental  ${\cal I}_w$ as $x$. Any $y \in {\cal I}_k(x)$  then emits the same prefix of length $k$ as $x$. The random variable $X_k := \log \mu({\cal I}_k)$ is thus of great interest.

\smallskip The present paper studies Shannon weights and describes  their   asymptotic behaviour. Observe  that, for $k \ge 1$,  the function $x \mapsto  {\cal I}_k(x)$ is a staircase function which is constant on each interval $\cal I_w$ with $|w| = k$, and  thus  Shannon weights exactly coincide with the  expectations  of the variable $x \mapsto X_k(x) $, 
\vskip 0.1 cm 
\centerline{ $  \underline m_\mu(k) = \E_\mu  [X_k], \quad X_k (x) =  |\log \mu( {\cal I}_k)(x)| $. }

\medskip
 The paper  \cite{les4}   studies  other  probabilistic characteristics of the random variable  $X_k := |\log \mu({\cal I}_k)|$.   It   introduces   in Definitions 2.1 and 2.2 of  \cite{les4} two notions of weights  (AEW) and (MW)   that   provide  extensions of respective entropies introduced in the paper  of Dajani and  Fieldsteel \cite{DaFi}.  
 It does not consider the notion of Shannon weights.
 
 \smallskip   We now rewrite  all these definitions   in the ``dynamical'' context: 
 
 \begin{definition}\label{def1} {\rm \cite{les4}} Consider a dynamical source ${\cal S}$,  a measure $\mu$ on the unit interval,    for any depth $k$, the random variable  
 $X_k := |\log \mu({\cal I}_k)|$, and a  real positive sequence $f: k \mapsto f_k$.

\medskip 
{\rm (SW)}$[f]$ \ \ 
The  source ${\cal S}$ has   a Shannon weight $f_k$ if and only if 
 \vskip 0.1 cm 
 \centerline{$ \E_\mu[  X_k ] \sim   f_k \quad (k \to \infty)$}

\smallskip
{\rm (AEW)}$[f]$  \ \ 
The  source ${\cal S}$   has a  almost everywhere weight  $f_k $  if and only if 
 \vskip 0.1 cm
\centerline { $ X_k (x)     \sim 
 f_k    \quad \hbox {\rm [$x$ almost everywhere]}  \quad (k \to \infty)$}
 
\smallskip
 {\rm(MW)}$[f]$  \ \ 
The source ${\cal S}$  has a  weight  (in measure)\footnote{We first recall that  a sequence $X_k$ converges to 0 in measure if and only if   $$ \forall \epsilon>0, \lim_{k \to \infty}  \left| \{ x \in {\mathcal I}\mid |X_k (x)|> \epsilon \} \right| = 0\, . 
 \quad (k \to \infty)$$} $f_k $  if and only if  
  \vskip 0.1 cm
\centerline{ $  X_k(x)  \sim 
f_k  \quad \hbox  
 {\rm [in measure]} \quad (k \to \infty)$}

  \smallskip
  In each of the three cases,  if the sequence $k \mapsto f_k$ is linear, of the form $f_k = c\cdot k$, the source ${\cal S}$ is said to have  entropy $c$  (resp. Shannon entropy, almost everywhere entropy, entropy  in measure).
  
  \end{definition}
  
   \medskip 
  The paper \cite{les4}  focuses on the notions  (AEW) or   (MW) 
 and  considers many classical sources defined in Number Theory contexts;   
   it  studies in particular   the Farey dynamical source  (called in  \cite{les4}  the Stern-Brocot source) from these two  points of view and proves  the following:  
   
\smallskip   
   -- the Farey source    satisfies (MW)$[f]$  with
    $  \displaystyle f_k  =   \displaystyle \phi(0)  ({\pi^2}/{6})  (  {k} /{\log k} )$ ;  
 
 \smallskip   
    -- the Farey source does not satisfy (AEW).

\medskip The  two papers (the \cite{les4} paper and the present one)   thus study    the same  sequence  \vskip 0.1 cm 
\centerline{ $Y_k(x) =  \displaystyle \frac {X_k(x)}{f_k} -1$\, , } 
\vskip 0.1 cm 
(and both  in an indirect way).  The paper \cite{les4}  shows  in its  Proposition 4.7 that  $Y_k$  converges to 0  in measure,  with  tools that are  quite specific to the Farey source and its close relation to continued fraction expansions (using for instance a  deep result due to Khinchin), while  the present one obtains the convergence of the expectation $\E[Y_k]$  to 0, with  quite different  tools (notably Abelian and Tauberian theorems, together with  a renewal theorem) that  can be applied to more general contexts.

\medskip  There  is now an important question:  Is it possible to directly relate the two results in our setting?  There does not exist  any  general relation between \rm (MW)$[f]$ and  \rm (SW)$[f]$, 
and one needs extra {\sl sufficient} assumptions on $Y_k$:  for instance,   general domination properties, described  in \cite{Tao}, or a  more  precise information on  the second moment of $X_k$, would  allow to  relate \rm (MW)$[f]$ to \rm (SW)$[f]$.   

Assume, for instance,  that  the second moment  $\E[X_k^2]$ of $X_k$ is asymptotically equivalent to  the square $f_k^2 = (\E[X_k])^2$ of the first moment; then the variance  ${\tt Var}(X_k)$ would satisfy  ${\tt Var}(X_k)=\E[X_k^2]-f_k^2=o(f_k^2)$,  and   Chebyshev's inequality  would apply and imply 
\begin{equation} \label{Che}
\mu[|Y_k|\geq \varepsilon] = \mu[|X_k-f_k|\geq f_k \varepsilon]  \leq \frac{{\tt Var}(X_k)}{f_k^2 \varepsilon^2} = o(1)\,,
\end{equation}
for any fixed $\varepsilon>0$. This would show  that $Y_k$ converges to 0 in measure.

This relation between \textrm{(MW)} and \textrm{(SW)} will be addressed in our setting in a forthcoming paper

\bigskip
\bigskip

Ali Akhavi\\
LIPN, Universit\'e Sorbonne Paris Nord, France\\
\texttt{ali.akhavi@lipn.univ-paris13.fr}

\medskip

Eda Cesaratto\\
Universidad Nac. de Gral. Sarmiento\\
Dept. de Matem\`atica, Fac. Cs. Exactas y Naturales, Universidad de Buenos Aires\\
CONICET Argentina\\
\texttt{ecesaratto@campus.ungs.edu.ar}

\medskip

Frédéric Paccaut\\
LAMFA, CNRS UMR 7352, Universit\'e de Picardie Jules Verne, France\\
\texttt{frederic.paccaut@u-picardie.fr}

\medskip

Pablo Rotondo\\
LIGM CNRS UMR 8049, Universit\'e Gustave Eiffel, France\\
\texttt{pablo.rotondo@u-pem.fr}

\medskip

Brigitte Vallée\\
GREYC,CNRS and Universit\'e de Caen-Normandie, France\\
\texttt{brigitte.vallee@unicaen.fr}

\end{document}